\documentclass[10pt]{articleHJ}

\usepackage[centertags]{amsmath}
\usepackage{amsfonts,amssymb,amsthm}
\usepackage{caption,graphicx}
\usepackage[usenames]{xcolor}
\usepackage[text={6.0in,8.6in},centering,letterpaper]{geometry}
\usepackage[numbers,comma,square,sort&compress]{natbib}

\newlength{\defbaselineskip}
\newcommand{\setlinespacing}[1]%
           {\setlength{\baselineskip}{#1 \defbaselineskip}}

\newcommand{\singlespacing}{\setlength{\baselineskip}{\defbaselineskip}}

\def\Re{\mathop\mathrm{Re}\nolimits}			
\def\Im{\mathop\mathrm{Im}\nolimits}			
\newcommand{\Real}{\mathbb{R}}							
\newcommand{\Complex}{\mathbb{C}}							
\newcommand{\Wholes}{\mathbb{Z}}							
\newcommand{\abs}[1]{\left\vert#1\right\vert}			
\newcommand{\norm}[1]{\left\Vert#1\right\Vert}		
\newcommand{\sref}[1]{(\ref{#1})}                       

\newcommand{\R}{\mathbb{R}}
\newcommand{\Z}{\mathbb{Z}}
\newcommand{\C}{\mathbb{C}}							
\newcommand{\<}{\left\langle}
\renewcommand{\>}{\right\rangle}
\newcommand{\eps}{\varepsilon}
\newcommand{\be}{\begin{equation}}
\newcommand{\ee}{\end{equation}}
\newcommand{\bea}{\begin{eqnarray}}
\newcommand{\eea}{\end{eqnarray}}
\newcommand{\ba}{\begin{array}}
\newcommand{\ea}{\end{array}}

\newtheorem{thm}{Theorem}[section]
\newtheorem{cor}[thm]{Corollary}
\newtheorem{lem}[thm]{Lemma}
\newtheorem{prop}[thm]{Proposition}

 {\begin{trivlist}\item[]\textbf{Acknowledgments }}{\end{trivlist}}

\begin{document}

\begin{frontmatter}
\title{Multi-Dimensional Stability of Waves \\ Travelling through Rectangular Lattices \\ in Rational Directions}
\journal{...}
\author[OL]{A. Hoffman},
\author[LD]{H. J. Hupkes\corauthref{coraut}},
\corauth[coraut]{Corresponding author. }
\author[KU]{E. S. Van Vleck}
\address[OL]{
  Franklin W. Olin College of Engineering   \\
  1000 Olin Way; Needham, MA 02492; USA \\ Email:  {\normalfont{\texttt{aaron.hoffman@olin.edu}}}
}
\address[LD]{
  Mathematisch Instituut - Universiteit Leiden \\
  P.O. Box 9512; 2300 RA Leiden; The Netherlands \\ Email:  {\normalfont{\texttt{hhupkes@math.leidenuniv.nl}}}
}
\address[KU]{
  Department of Mathematics - University of Kansas \\
  1460 Jayhawk Blvd; Lawrence, KS 66045; USA \\
  Email: {\normalfont{\texttt{erikvv@ku.edu}}}
}
\date{\today}

\begin{abstract}
\singlespacing
We consider general reaction diffusion systems
posed on rectangular lattices in two or more spatial dimensions.
We show that travelling wave solutions to such systems
that propagate in rational directions are nonlinearly stable under small perturbations.
We employ recently developed techniques
involving point-wise Green's functions estimates for functional differential
equations of mixed type (MFDEs), allowing our results to be applied
even in situations where comparison principles are not available.
\end{abstract}



\end{frontmatter}


\numberwithin{equation}{section}
\renewcommand{\theequation}{\thesection.\arabic{equation}}

\section{Introduction}
\label{sec:int}

A classical line of inquiry in applied mathematics concerns the mechanisms by which energy or activity is transported across a spatial domain.
In this paper we study the dynamics in higher-dimensional
spatially discrete models such as the Nagumo
lattice differential equation (LDE)
\begin{equation}
 \label{eq:int:nagumo:lde}
 \dot{u}_{ij} = u_{i+1,j}+u_{i-1,j}+u_{i,j+1}+u_{i,j-1}-4u_{ij} + g(u_{ij}), \qquad (i,j) \in \Z^2,
 \end{equation}
with prototypical cubic nonlinearity $g(u) = u(u-1)(\rho-u)$ for some $\rho \in (0,1)$.
We are interested in initial data that are close to a
planar travelling wave solution
\begin{equation}
\label{eq:TWansatz}
 u_{ij}(t) = \Phi( i \sigma_1  + j \sigma_2 + ct); \quad \Phi(-\infty) = 0, \qquad \Phi(\infty) = 1.
\end{equation}

The Nagumo equation is a phenomenological model in which two stable equilibria compete for dominance in a spatial domain.
In modelling contexts one often thinks of these equilibria as representing material phases or chemical or biological species.
The LDE \eqref{eq:int:nagumo:lde} occurs both as a model in its
own right as well as a standard semi-discretization of the Nagumo PDE.
Planar travelling waves \sref{eq:TWansatz}
can be thought of as a mechanism of transport in which the
fitter species or more energetically favourable phase invades the spatial domain, albeit for very specialized initial data.
Establishing the stability of these waves,
as we do here in this paper,
shows that this transport mechanism
is available
for an open set of initial
conditions in phase space.

We note that much of the prior work on stability of travelling waves in (one-dimensional) lattice reaction-diffusion
equations relies heavily on comparison principles.
Indeed, one can often leverage the monotonicity inherent in the problem
by using Lyapunov functions or sub and super-solutions
that sandwich the travelling wave; see e.g. \cite{ChenFuGuo,SHEN1999a}.
It is only recently that existence and stability
results have become available for LDEs that do not admit a comparison principle
\cite{BatesInfRange,HSdFHN1,HJHSTBFHN}.
In order to incorporate such LDEs into our analysis,
we avoid the use of comparison principles
in our stability results. Indeed, in the bulk of the paper we study
a general class of models that includes systems of
LDEs such as the discrete Fitzhugh-Nagumo equation, possibly with nonlinear and non-symmetric coupling.
For concreteness however, we focus exclusively on the Nagumo equation
for the remainder of this introduction,
but restrict our technical discussion to spectral analysis.

\subsection*{Existence of Planar Fronts}

To motivate \eqref{eq:int:nagumo:lde} and discuss the mathematical difficulties which arise in its study,
we recall the classical Nagumo PDE
\begin{equation}
 u_t = u_{xx}+ u_{yy} + g(u) \label{eq:nagumoPDE}
\end{equation}
in two space dimensions. This PDE features a thresholding nonlinearity that promotes high frequencies
and has to compete with the diffusion operator, which by contrast attenuates high frequencies.
The resolution of this competition is the formation of travelling waves,
i.e. solutions of the form
\begin{equation}
 u(t,x,y) = \Phi(x \cos \zeta + y \sin \zeta +ct); \qquad \Phi(-\infty) = 0, \; \Phi(\infty) = 1
\end{equation}
which propagate with speed $-c$  in the direction $(\cos \zeta, \sin \zeta)$
with a fixed monotone wave profile $\Phi$.
The existence of such waves can be established via phase-plane analysis,
since the wave profile
necessarily satisfies the second order ODE
\begin{equation}
  \label{eq:int:eq:waveProfile:pde}
 c\Phi' = \Phi'' + g(\Phi).
\end{equation}

We remark that in moving from the continuum to the lattice, the first
main difference already appears
when formulating the analogue of the travelling wave ODE \sref{eq:int:eq:waveProfile:pde}.
Indeed, the wave profile $\Phi : \R \to \R$ featuring in \eqref{eq:int:nagumo:lde}-\eqref{eq:TWansatz} necessarily satisfies
the system
\begin{equation}
   \label{eq:MFDE}
  \begin{array}{lcl}
  c\Phi'(\xi) &= & \Phi(\xi + \sigma_1) + \Phi(\xi - \sigma_1) + \Phi(\xi + \sigma_2) + \Phi(\xi - \sigma_2) - 4\Phi(\xi)
     + g(\Phi(\xi)); \\[0.2cm]
  & & \qquad \Phi(-\infty) = 0, \qquad \Phi(\infty) = 1.
  \end{array}
\end{equation}
Inspection of this system immediately raises two important considerations
that must be addressed.

The first issue is that due to the shifted arguments, the system \sref{eq:MFDE} is a
mixed-type functional differential equation (MFDE).
The theory for MFDEs is both more subtle and less well known than the theory for ODEs.
Indeed, the equation \eqref{eq:MFDE} is not well-posed when regarded as an evolution equation on its natural state space
$C^0([-\sigma_{\mathrm{max}},\sigma_{\mathrm{max}}],\R)$, with $\sigma_{\max} = \max\{ \abs{\sigma_1}, \abs{\sigma_2} \}$.
Nevertheless, the theory for MFDEs is
sufficiently mature \cite{MPA,MPVL,HJHLIN,Harterich} for our purposes in this paper.
For example, it is known that for each bistable
$g$ and for each $\sigma = (\sigma_1, \sigma_2) \in \R^2$ there is a unique $c$ for which \eqref{eq:MFDE} admits a
monotone solution \cite{MPB}. We emphasize
that we do not regard the current paper as making any significant contribution towards handling MFDEs.

The second issue is that the direction of propagation $\sigma$ features
prominently in \sref{eq:MFDE},
in contrast to the ODE \sref{eq:int:eq:waveProfile:pde} which is completely isotropic.
%
In fact, this $\sigma$ dependence can be quite rough, particularly for choices of $g$ at which the two
spatially constant equilibria are nearly but not exactly energetically equivalent.
For such choices of $g$ the waves move slowly enough to feel the lattice discreteness as a kind of friction which
further slows down the fronts, sometimes to the point of pinning them \cite{Cahn,Bell1984,Elmer2006}.
Moreover, both numerical and analytical work \cite{CMPVV,MPCP,HJHVL2005,HoffMPCrys} points to this additional friction
being felt anisotropically, more strongly in directions such as $\sigma = (1,0)$ and $\sigma = (1,1)$ for which there is stronger
resonance with the lattice and weakening rapidly as the resonance weakens.
We regard the main contribution of this work as the extension of techniques
in nonlinear stability theory to handle the peculiarities
that arise in anisotropic systems.

\subsection*{Stability of Travelling Fronts}

A natural question, and a chief motivator of the present study, regards the extent to which the slowing of
planar fronts facing lattice directions governs the evolution of initial data which are not themselves exact travelling waves.
Preliminary numerical experiments show compactly supported initial data expanding to occupy a square whose
sides are aligned with the lattice directions and which propagate outwards at the (slow) speed of the planar fronts.
This suggests that the anisotropy
in the wave speeds of planar fronts governs the dynamics for a large set of initial data.
Indeed, one may conjecture that
fronts facing the lattice directions $(1,0)$ and $(0,1)$ play a
distinguished role in governing the dynamics of a large class of initial data.

In this study we take a small step towards understanding this behaviour
by establishing the multidimensional stability of planar fronts moving through the lattice $\Z^2$.
In particular, we show that the wave speeds of the planar fronts govern the dynamics of initial data which are themselves
close to planar fronts.  Furthermore, our results suggest that the waves facing directions in strongest resonance with the lattice,
i.e. $\sigma = (1,0)$ or $(0,1)$, are more stable in the following sense.
Although we obtain the same slow $t^{-1/4}$ rate
for the relaxation of
transversal phaseshift deformations
regardless of the direction that the wave is facing,
we do obtain the relatively fast $t^{-3/2}$ rate for the rest of the perturbation to fronts facing a lattice direction
as compared to the relatively slow $t^{-3/4}$ that we are able to obtain for perturbations to fronts facing an oblique direction.

These observations seem to support the conjecture mentioned above.
Of course, much work remains to be done
and a rigorous examination of initial data that is far from a planar front is beyond the scope of this study.
Such an analysis may nevertheless be regarded as a natural next step,
especially in view of the fact
that non-planar travelling wave solutions to PDEs have been actively investigated of late;
see e.g. \cite{HaragusScheel,Taniguchi}.

%

\paragraph{One-Dimensional Stability}
In order to illustrate the extra complications that the anisotropy
of \sref{eq:int:nagumo:lde} imposes on our nonlinear stability analysis,
we start by discussing the stability of waves in one spatial dimension.
%
The linearization of the one-dimensional version of
the PDE \eqref{eq:nagumoPDE} about its travelling wave
has a simple eigenvalue at zero coming from the translational symmetry in the problem.
The remainder of the spectrum is contained away from
the imaginary axis in the left half-plane.  That the only neutral mode arises from translation suggests the Ansatz
$u(t,x) = \Phi(x+ct + \theta(t)) + v(x,t)$.  The degree of freedom generated by allowing $\theta$ to vary
can be used to demand that the rapidly decaying part of the perturbation $v$ is orthogonal to the (adjoint)
neutral mode, implying that the linear part of the evolution for $v$ generates an exponentially stable semigroup.
One can then set up and close a bootstrapping argument with $v$ decaying exponentially fast and $\theta$ remaining small.

We remark that the argument sketched above can be transferred from the one-dimensional continuum $\R$ to the
one-dimensional lattice $\Z$.
The chief hurdle to be overcome is that unlike in the continuum, on the lattice travelling waves are no longer equilibria in a moving
frame but rather periodic modulo the shift.  This observation is not new; techniques to overcome it have been developed, for example in \cite{VL9,FP3,BGV2003}.  Here we proceed along the lines of the program pioneered by Benzoni-Gavage
and coworkers \cite{BGV2003} in which Green's functions for linear LDEs
are constructed by `filling in' the space between lattice points,
restoring the relative equilibrium point of view.

\paragraph{Two-Dimensional Stability}
In moving from one spatial dimension to two, an additional mathematical difficulty is that the interface
is no longer compact.  Instead of occupying a point or interval, it now occupies a line or strip.
One must now be concerned not only with rigid translations of the front but with long wave
deformations of the interface which decay slowly at the linear level.
The slow decay of these long wave deformations manifests as a curve of essential spectrum
which touches the origin.  The presence of essential spectrum is a complicating factor in the
nonlinear stability analysis of travelling waves because the semigroups associated with the linear
part of the evolution now exhibit only algebraic decay and thus a bootstrapping argument is no longer
guaranteed to close if (even bounded) quadratic nonlinear terms are present. 
Nonlinear stability in the
presence of essential spectrum which touches the imaginary axis is presently an active area of research;
see e.g. \cite{BECKNGUSANZUM2012}.

Essential spectrum is also present in
\sref{eq:nagumoPDE}, but several authors have
managed to either skillfully manage
this complication \cite{KAP1997}
or avoid it altogether \cite{XIN1992,LEVXIN1992}.
Nevertheless, in the current paper we need
to reconsider this hurdle and develop techniques to
deal with the essential spectrum.
%
The chief difficulty in our setting is that although the spectral curves are no worse
than in the continuum case, the spectral projections are complicated by the anisotropy of the lattice.
To better see this, we now give a brief description of how the proof of stability in \cite{KAP1997} can be adapted to
the case where the front faces a lattice direction $\sigma = (1,0)$ and how this proof must be extended to cope with
the anisotropy of the lattice when the front faces an oblique direction.

\paragraph{Lattice Directions}
In the case when the wave is facing the lattice direction $\sigma = (1,0)$,
we search for solutions to \sref{eq:int:nagumo:lde} of the form
\begin{equation}
\label{eq:int:hor:waves:ansatz}
u_{ij}(t) = \Phi(i+ct + \theta_j(t)) + v_{ij}(t),
\end{equation}
in which the travelling wave $(\Phi,c)$ satisfies the MFDE
\begin{equation}
  \begin{array}{lcl}
  c\Phi'(\xi) &= & \Phi(\xi + 1) + \Phi(\xi - 1) - 2 \Phi(\xi)
     + g(\Phi(\xi)); \\[0.2cm]
  & & \qquad \Phi(-\infty) = 0, \qquad \Phi(\infty) = 1.
  \end{array}
\end{equation}
In particular, the functions $\theta_j$ allow the phase of the wave
to vary in the transverse direction,
meaning the direction perpendicular to the motion of the wave and parallel to the interface which separates regions
dominated by the respective equilibria. The functions $v_{ij}$ capture
any remaining perturbations that are orthogonal to these phase shifts in an appropriate sense.

Substituting the Ansatz \sref{eq:int:hor:waves:ansatz} into the evolution
\sref{eq:int:nagumo:lde}, one obtains
\begin{equation}
 \label{eq:int:evolutionForVLatticeDir}
 \dot{v}_{ij} = [L_{ct} v]_{ij} + \Phi'(\xi)(\theta_{j+1} + \theta_{j-1} - 2\theta_j) - \dot{\theta}_j \Phi'(\xi+\theta_j)
 + N_{ij}(v,\theta)
\end{equation}
with $\xi = i + ct$, together with the identity
\begin{equation}
\label{eq:int:defL}
 [L_{ct} v]_{ij} = v_{i+1,j} + v_{i-1,j} + v_{i,j+1} + v_{i,j-1} - 4v_{ij} + g'(\Phi(i + ct))v_{ij}
\end{equation}
and the estimate
\begin{equation}
\begin{array}{lcl}
N_{ij}(v,\theta) & = & O\big((\theta_{j+1}-\theta_j)^2 + (\theta_{j-1} - \theta_j)^2 \big)
  + O \big( \theta_j(\theta_{j+1} + \theta_{j-1} - 2 \theta_j) \big) \\[0.2cm]
& & \qquad + O\big( \theta_jv_{ij} \big)
 + O \big( v_{ij}^2\big).
\end{array}
\end{equation}

Our choice of notation emphasizes the fact that the operator $L_{ct}$ is non-autonomous.
As a consequence, it is not immediately
clear how to leverage information about the spectrum of $L_{ct}$ to obtain decay rates on the semiflow it generates.
The key is to notice that the operator $L_{ct}$ is constant-coefficient with respect to $j$.
Performing a Fourier transform in this transverse direction decouples modes
with different frequencies $\omega$ and reduces \sref{eq:int:defL} to
\begin{equation}
\label{eq:int:defLOmega}
 [L_{\omega, ct} \widehat{v}(\omega)]_{i} = \widehat{v}_{i+1}(\omega) + \widehat{v}_{i-1}(\omega)
  + e^{i \omega } \widehat{v}_{i}(\omega) + e^{- i \omega} \widehat{v}_{i}(\omega) - 4\widehat{v}_{i}(\omega) + g'(\Phi(i +ct))\widehat{v}_{i}(\omega).
\end{equation}
The work of \cite{BGV2003,HJHNLS} essentially says that the spectrum of $\mathcal{L}_{\omega}:= - c\partial + L_{\omega, 0}$,
regarded as an autonomous operator from $W^{1,\infty}(\R,\R) \to L^\infty(\R,\R)$, governs the evolution semiflow of
$\frac{d}{dt}\widehat{v}(\omega) = L_{\omega, ct} \widehat{v}(\omega)$
in a fashion similar to the usual spectral mapping theorem.
To obtain information regarding the operator $\mathcal{L}_{\omega}$, it suffices
to drop the hats in \sref{eq:int:defLOmega}
and write
\begin{equation}
 [\mathcal{L}_\omega v](\xi) = - cv'(\xi) + v(\xi+1) + v(\xi-1) - 2(2-\cos \omega)v(\xi) + g'(\Phi(\xi))v(\xi).
\end{equation}

At this point, one notes that $\mathcal{L}_0$ is the linearization of the one-dimensional problem about
$\Phi$ which has received ample attention from investigators.  In particular, it is known that $\mathcal{L}_0$ is Fredholm with
index 0, kernel spanned by $\Phi'$ and range orthogonal to an adjoint eigenfunction $\psi$, which can be normalized
to $\<\Phi',\psi\> = 1$ \cite{MPA,MPB}.  Moreover, $\mathcal{L}_0$ has no spectrum in the closed right half plane other than its simple eigenvalue at zero \cite{VL9,MPCP}.  Furthermore, since $\mathcal{L}_\omega$ depends upon $\omega$ only by a spectral shift,
the spectral projection onto the principal eigenvector is $\omega$-independent.  This point is
absolutely crucial.

Indeed, we can now resolve
the ambiguity inherent in splitting
perturbations from the planar wave between $\theta_j$ and $v_{ij}$
by demanding
\begin{equation}
\label{eq:int:normalization}
\sum_i v_{ij}(t) \psi(i+ct) \equiv 0
\end{equation}
uniformly in $j$. 
This requirement can be interpreted as
the equivalent of a spectral condition
associated to
the non-autonomous operator $L_{ct}$, demanding that $v$
belongs to a spectral region contained in
$\{\Re \lambda < -\beta\} \subset \C$. In particular,
we can expect $L_{ct}$ to generate an exponentially stable semiflow on
$\{v \; | \; \sum_i v_{ij} \psi(i+ct) \equiv 0\}$.
Under this restriction, we will see in {\S}\ref{sec:nl}
that solutions to \sref{eq:int:evolutionForVLatticeDir}
must satisfy the fixed point condition
\begin{equation}
\label{eq:int:fixpoint}
\begin{array}{lcl}
\left(\begin{array}{l}v(t) \\ \theta(t)\end{array}\right)
& = &
\left(\begin{array}{cc } \mathcal{G}_{vv}(t, 0) & \mathcal{G}_{v \theta}(t, 0) \\
                        \mathcal{G}_{\theta v}(t, 0) & \mathcal{G}_{\theta \theta}(t,0)
\end{array}\right)
\left(\begin{array}{l}v(0) \\ \theta(0)\end{array}\right)  \\[0.4cm]
& & \qquad +  \int_{0}^t
\left(\begin{array}{cc } \mathcal{G}_{vv}(t, t_0) & \mathcal{G}_{v \theta}(t, t_0) \\
                        \mathcal{G}_{\theta v}(t, t_0) & \mathcal{G}_{\theta \theta}(t,t_0)
\end{array}\right)
\left(\begin{array}{l} \mathcal{N}_{v}\big(t_0;  v(t_0), \theta(t_0) \big) \\  \mathcal{N}_\theta\big( t_0; v(t_0), \theta(t_0) \big)\end{array}\right) \, d t_0.
\end{array}
\end{equation}
In view of the discussion above, we expect to have the exponential bounds
$\mathcal{G}_{v*}(t, t_0) =  O (e^{ - \beta (t - t_0)} )$, with $* = v, \theta$.
In addition, we expect $\mathcal{G}_{\theta \theta}(t, t_0)$
to behave similarly to a heat kernel, on account of the
leading order terms of the expansion
\begin{equation}
\dot{\theta}_j = \theta_{j+1} +  \theta_{j-1} - 2 \theta_j  +O( \abs{v} +  \abs{ \theta }^2)
\end{equation}
one obtains after imposing \sref{eq:int:normalization} upon \sref{eq:int:evolutionForVLatticeDir}.
The nonlinear terms can be roughly\footnote{
We are deliberately suppressing all mention of the precise
norms and sequence spaces we are referring to.
}
bounded as
\begin{equation}
\mathcal{N}_*(t; v , \theta) = O ( \abs{v}^2) + O( \abs{v} \abs{ \theta }) + O\big( \abs{\theta^\diamond}^2 \big)
 + O \big( \abs{ \theta } \abs{ \theta^{\diamond \diamond} } \big),
\end{equation}
again with $*  = v, \theta$. Here we have $\theta^{\diamond}_j = (\theta_{j+1} - \theta_j, \theta_{j -1} - \theta_j )$
and $\theta^{\diamond \diamond}_j = \theta_{j+1} + \theta_{j-1} - 2 \theta_j $.

%


In view of the exponential bounds on $\mathcal{G}_{v*}$
and the integral inequality
\begin{equation}
\int_0^t e^{-\alpha t} (1 + t_0)^{-\beta}dt_0 \le C(1 + t)^{-\beta},
\label{eq:liuduhamelexp}
\end{equation}
a rough estimate of the top component of the integral in
\sref{eq:int:fixpoint} suggests that $v$ will decay at the slowest algebraic rate appearing among the terms $\mathcal{N}_v$ and $\mathcal{N}_\theta$.
Using the heuristic that
$\theta$ decays like a one-dimensional heat kernel, which means its $k^{th}$ difference decays
like $t^{-\frac{2k+1}{4}}$, we see that the slowest
decay in $\mathcal{N}_*$ comes from the terms $\abs{\theta^\diamond}^2$ and
$\abs{\theta} \abs{ \theta^{\diamond \diamond} }$,
which behave as $t^{-3/2}$.

Assuming now that indeed $v \sim t^{-3/2}$, one can easily
recover the decay $\theta \sim t^{-1/4}$
from the bottom component of the integral in \sref{eq:int:fixpoint}
by
substituting $\alpha = 1/4$ and $\beta = 3/2$ into the integral inequality
\begin{equation}
\int_0^t (1 + t-t_0)^{-\alpha}(1 + t_0)^{-\beta}dt_0 \le C(1 +t)^{-\min\{\alpha + \beta - 1,\alpha,\beta\}};
  \quad \mathrm{ when }\; \alpha,\beta \ne 1,
\label{eq:liuduhamelalg}
\end{equation}
which can be found in \cite[Lem. 3.2]{ChernLiu1987}.
Of course, these computations merely show that the postulated
decay rates are \textit{consistent} with \sref{eq:int:fixpoint}.
Fortunately, a straightforward bootstrapping argument can be used
to rigorously recover these decay rates without too much extra work.

%
%

This rough sketch of the $\sigma = (1,0)$ case follows \cite{KAP1997} very closely and matches exactly the
results obtained there in the case of two dimensions.  In fact, in this special case the LDE
stability analysis is easier than the PDE stability analysis simply because all of the operators are
bounded and higher-order Sobolev norms are not needed.  We remark that this analysis is also
reminiscent of the work in \cite{BatesChen2002} in which fronts facing the horizontal direction
$\sigma = (1,0,\ldots,0)$ are shown to be stable in $\R^n$ with $n \ge 4$ for the nonlocal equation
$u_t = J\ast u + g(u)$ with $J$ appropriately elliptic.

\paragraph{Oblique Directions}
The chief difficulty in our work arises when we consider waves that are not facing the lattice directions.
In order to gain some insight, let us consider a wave facing a general direction
$(\sigma_1,\sigma_2) \in \Wholes^2$, with $\sigma^2_1 \neq \sigma^2_2$ and $\sigma_1 \sigma_2 \neq 0$.
 %
%
We seek an Ansatz which decomposes the general perturbation
into a slowly decaying perturbation associated with
deformations of the interface and a rapidly decaying perturbation which accounts for everything else.  That is,
we write
\begin{equation}
 u_{ij}(t) = \Phi(n+ct + \theta_l(t)) + v_{nl}(t); \qquad n = i\sigma_1 + j\sigma_2, \qquad l = j\sigma_1-i\sigma_2.
\end{equation}
Upon substituting this Ansatz into the evolution \sref{eq:int:nagumo:lde} and
using the wave profile equation \sref{eq:MFDE},
we obtain
\begin{equation}
 \dot{v}_{nl} = [L_{ct} v]_{nl} + [M_{ct}\theta]_{nl} + N_{nl}(v,\theta) + \dot{\theta}_l\Phi'(n+ct+\theta_l),
\end{equation}
where
\begin{equation}
 [L_{ct} v]_{nl} = v_{n+\sigma_1,l+\sigma_2} + v_{n+\sigma_2,l-\sigma_1} +  v_{n-\sigma_1,l-\sigma_2} + v_{n-\sigma_2,l+\sigma_1} - 4v_{nl}
  + g'(\Phi(n+ct))v_{nl}
\end{equation}
and
\begin{equation}
\begin{array}{lcl}
 [M_{ct} \theta]_{nl} & = & \Phi'(n+\sigma_1+ct)(\theta_{l+\sigma_2}-\theta_l)  + \Phi'(n+\sigma_2+ct)(\theta_{l-\sigma_1}-\theta_l) \\[0.2cm]
 & & \qquad   + \Phi'(n-\sigma_1+ct)(\theta_{l-\sigma_2}-\theta_l)  + \Phi'(n-\sigma_2+ct)(\theta_{l+\sigma_1}-\theta_l),
\end{array}
\end{equation}
as well as
\begin{equation}
  \begin{array}{lcl}
 N_{nl}(v,\theta) & = & \sum_{k=1}^2 O\big((\theta_{l+\sigma_k}-\theta_l)^2 + (\theta_{l-\sigma_k}-\theta_l)^2 \big)
  \\[0.2cm]
& & \qquad + \sum_{k = 1}^2 O\big( \theta_l (\theta_{l + \sigma_k} - \theta_l) \big)
           + \sum_{k=1}^2   O\big( \theta_l (\theta_{l - \sigma_k} - \theta_l) \big) \\[0.2cm]
& & \qquad +  O \big( \theta_lv_{nl} \big) + O\big( v_{nl}^2 \big).
  \end{array}
\end{equation}

The anisotropy of the lattice manifests itself in several places.
Firstly, the term $M\theta$ is linear in the first differences of $\theta$, rather than the second differences.
In particular, we expect the nonlinear estimates
\begin{equation}
\label{eq:int:estNStarGenDir}
\mathcal{N}_*(t; v , \theta) = O ( \abs{ v }^2) + O( \abs{ v } \abs{ \theta })
+ O \big(  \abs{ \theta^\diamond}^2 \big) + O \big( \abs{\theta} \abs{ \theta^{\diamond} } \big),
\end{equation}
now with
\begin{equation}
\theta^{\diamond}_l = \big( \theta_{l + \sigma_2} - \theta_l, \theta_{l - \sigma_1} - \theta_l,
\theta_{l - \sigma_2} - \theta_l, \theta_{l + \sigma_1} - \theta_l \big).
\end{equation}
The term $\abs{\theta} \abs{ \theta^{\diamond} }$ is particularly dangerous,
since we can expect it to behave as $t^{-1}$,
which puts us in
the $\beta = 1$ case for which \sref{eq:liuduhamelalg} fails to hold.

This particular difficulty is often encountered when studying the stability of
travelling waves for PDEs with a conservation
law structure, e.g. viscous shocks \cite{ZUM2009}.  We make use of the simple, yet crucial, identity
\begin{equation}
\theta_l (\theta_{l+1} - \theta_l ) = \frac{1}{2} [ \theta_{l+1}^2 - \theta_l^2 - (\theta_{l+1} - \theta_l)^2 ],
\label{eq:uux}
\end{equation}
which can be regarded as a discrete form of $uu_x = \big(\frac{1}{2}u^2\big)_x$.
The square of the first difference decays like $t^{-3/2}$ and hence is safe.  The difference of the squares
though only decays like $t^{-1}$. However, we can sum by parts, transferring the difference to the heat kernel
and in so doing move from the bad estimate
\begin{equation}
\label{eq:int:badEstimate}
\int_0^t (1 + t-s)^{-1/4} (1 +s)^{-1} ds \sim \log(1 + t)(1 +t)^{-1/4}
\end{equation}
to the good estimate
\begin{equation}
 \int_0^t (1 + t-s)^{-3/4} (1 + s)^{-1/2} ds \sim (1 + t)^{-1/4}.
\end{equation}
In particular, we must carefully keep track of the nonlinear terms of order
$O( \abs{ \theta } \abs{ \theta^\diamond })$
to make sure that they have the correct form.

Before we can proceed further, we need to decide on a normalization condition
to use for the phases $\theta$. The linear operators $\mathcal{L}_{\omega}$
are now given by
\begin{equation}
\label{eq:int:Lomega}
\begin{array}{lcl}
 [\mathcal{L}_\omega v](\xi) &=& -cv'(\xi) + e^{i\omega\sigma_2}v(\xi+\sigma_1) +
  e^{-i\omega \sigma_1}v(\xi+\sigma_2) + e^{-i\omega \sigma_2}v(\xi-\sigma_1)
  + e^{i\omega \sigma_1}v(\xi-\sigma_2) \\[0.2cm]
& & \qquad - 4v(\xi)  + g'\big(\Phi(\xi)\big)v(\xi).
\end{array}
\end{equation}
In particular, the dependence on the frequency $\omega$ is now non-trivial.
It is true that $\mathcal{L}_0$ still has a simple
eigenvalue at zero with the rest of the spectrum in the left half-plane and thus regular perturbation
yields a curve of eigenvalues $\lambda_\omega$ and eigenvectors $\phi_\omega$.  However,
because the
leading eigenvectors $\phi_\omega$ depend upon $\omega$, the natural orthogonality condition now
lives in the Fourier domain. Indeed, it is natural to require that for each frequency
$\omega$,
the spectral projection onto $\phi_\omega$ of the Fourier coefficient
$\widehat{v}(\omega)$ must vanish.
On the range of such a projection,
one would expect the operator $L_{ct}$ to generate an exponentially stable semiflow.

However, we feel that the frequency mixing inherent in such a projection will
lead to prohibitively complicated expressions for the nonlinearities
that will most likely prevent us from
verifying the conservation-law structure of the $O(\abs{\theta} \abs{ \theta^{\diamond}})$
terms. This would thwart any attempts to
circumvent the $\beta = 1$ problem in \sref{eq:liuduhamelalg} through the use of \sref{eq:uux}.
For this reason, we continue to use the frequency independent
normalization
\begin{equation}
 \sum_n v_{nl}\psi(n+ct) \equiv 0
\end{equation}
that worked so well in the $\sigma = (1,0)$ case.
The price that needs to be paid is that the semiflow $\mathcal{G}_{v \theta}(t, t_0)$
now decays algebraically at the rather slow rate of $O( (1 + t - t_0)^{-3/4} )$.
The frequency dependence of the eigenfunctions $\phi_{\omega}$ is directly to blame for this.

At first glance, this means that the $O( \abs{\theta} \abs{ v })$ term in \sref{eq:int:estNStarGenDir}
behaves as $t^{-1}$, which of course is problematic. By carefully
inspecting the part of $\mathcal{G}_{v \theta}$ that decays at the slowest rate,
we are able to show that the troublesome portion of $v$ is directly
proportional to $\theta^{\diamond}$. Together with a detailed analysis of
the $O( \abs{\theta} \abs{ v })$ terms in the nonlinearities, this allows the use of \sref{eq:uux}
to prevent the logarithm from appearing in \sref{eq:int:badEstimate}.

\subsection*{Overview}

In the following sections we flesh out in great detail the sketch presented above.  In {\S}\ref{sec:mr} we state our hypotheses and main theorems.
Our theorems are of the form ``spectral stability implies nonlinear stability.''
Our hypotheses relate largely to the existence of a travelling wave and spectral assumptions on the analog of
$\mathcal{L}_0$ and $\mathcal{L}_\omega$.  Section \ref{sec:cds} is devoted to the development of coordinate systems,
both the $(n,l)$ coordinates for the lattice and the $(v,\theta)$ coordinates for $u$ as well as to the derivation
of a system of evolution equations in these coordinates and the estimation of the nonlinear terms.
Section \ref{sec:green} is devoted to linear analysis, namely pointwise estimates on the Green's functions associated with
the $(v,\theta)$ evolution at the linear level.  Section \ref{sec:nl} is devoted to nonlinear stability,
showing that so long as all of the terms on the right hand side decay as they ought for $t \in [0,T]$,
then they actually decay as they ought for $t \in [0,\infty)$.
Finally in {\S}\ref{sec:nag} we show that our theorem is not vacuous.
More concretely, we verify our abstract hypotheses for the Nagumo equation \eqref{eq:int:nagumo:lde}.
All of the hypotheses with analogs in the continuum case \cite{KAP1997} we can verify analytically.
However we need to make an additional hypothesis $(HS1)_\omega$ to control terms which arise due to the anisotropy of the lattice.
With the exception of directions close to the $\sigma = (1,0)$ or $\sigma = (1,1)$ cases,
in which these terms are absent or benign,
we cannot analytically verify this hypothesis and instead use numerics.
We conclude in {\S}\ref{sec:dis} by discussing the
decisions we made that limit the scope of our work
and suggesting potential areas for further research.



\paragraph{Acknowledgments}
Hoffman acknowledges support from the NSF (DMS-1108788).
Hupkes acknowledges support from the Netherlands Organization for Scientific Research (NWO).
Van Vleck acknowledges support from the NSF (DMS-0812800 and DMS-1115408).

\section{Main Results}
\label{sec:mr}
Consider the nearest-neighbour lattice differential equation
\begin{equation}
\label{eq:mr:mainLDE}
\dot{u}_{ij}(t) = f\big(  u_{i+1, j}(t), u_{i, j+1}(t), u_{i-1, j}(t), u_{i, j-1}(t) , u_{ij}(t) \big),
\end{equation}
for which $u_{ij}(t) \in \Real^d$ for all $(i,j) \in \Wholes^2$.
For convenience, we introduce the symbol $\pi^+_{ij}: \ell^\infty(\Wholes^2; \Real^d) \to (\Real^d)^5$
for any pair $(i,j) \in \Wholes^2$,
which acts as
\begin{equation}
\pi^+_{ij} u = \big( u_{i+1, j}, u_{i, j+1}, u_{i-1, j}, u_{i, j-1} , u_{ij} \big) \in (\Real^d)^5.
\end{equation}
In particular, these operators produce a vector of neighbouring grid points
that are arranged according to the ''+''-shaped stencil.
The system \sref{eq:mr:mainLDE} can now be written in the condensed form
\begin{equation}
\label{eq:mr:pimainLDE}
\dot{u}_{ij}(t) = f\big( \pi^+_{ij} u(t) \big).
\end{equation}
The conditions we need on the nonlinearity $f$ are summarized
in the following assumption.
\begin{itemize}
\item[(Hf)]{
  The nonlinearity $f: (\Real^d)^5 \to \Real^d$ is $C^2$-smooth and there exist two points $u_\pm \in \Real^d$ with
  \begin{equation}
    f( u_\pm, u_\pm, u_\pm, u_\pm, u_\pm) = 0.
  \end{equation}
}
\end{itemize}
Let us now pick an arbitrary pair $(\sigma_1, \sigma_2) \in \Wholes^2$ with $\mathrm{gcd} ( \sigma_1, \sigma_2) = 1$.
Our next assumption states that \sref{eq:mr:mainLDE} admits a wave solution that travels
in the direction $(\sigma_1, \sigma_2)$ through the lattice.
\begin{itemize}
\item[(H$\Phi$)]{
  There exists a wave speed $c \neq 0$ and a wave profile $\Phi \in C^3(\Real, \Real^d)$ that satisfies the limits
  \begin{equation}
     \lim_{\xi \to \pm \infty} \Phi(\xi) = u_\pm \label{eq:mr:wavelim}
  \end{equation}
  and yields a solution to \sref{eq:mr:mainLDE} upon writing
  \begin{equation}
    u_{ij}(t) = \Phi\big( i \sigma_1 + j \sigma_2 + c t \big). \label{eq:mr:TWansatz}
  \end{equation}
}
\end{itemize}
We note that the pair $(c, \Phi)$ solves the functional differential equation of mixed type (MFDE)
\begin{equation}
\label{eq:mr:waveprofileMFDE}
c \Phi'(\xi) = f \big( \Phi(\xi + \sigma_1), \Phi(\xi + \sigma_2), \Phi(\xi - \sigma_1), \Phi(\xi - \sigma_2), \Phi(\xi) \big).
\end{equation}
In particular, the $C^3$-continuity mentioned in (H$\Phi$) is automatic upon assuming that $\Phi$ is merely continuous.
For conciseness, we rewrite \sref{eq:mr:waveprofileMFDE}
in the form
\begin{equation}
c\Phi'(\xi) = f\big( [\tau\Phi](\xi) \big),
\end{equation}
in which the operator $\tau: C(\Real, \Real^d) \to C(\Real, (\Real^d)^5)$ acts as
\begin{equation}
\label{eq:mr:defTauWithoutOmega}
[\tau p](\xi) = \big(p(\xi+\sigma_1),p(\xi+\sigma_2),p(\xi-\sigma_1),p(\xi-\sigma_2),p(\xi) \big) \in (\Real^d)^5.
\end{equation}

A standard approach towards establishing the stability of the wave $(c, \Phi)$
under the nonlinear dynamics of the LDE \eqref{eq:mr:mainLDE} is to consider the linear variational problem
\begin{equation}
\dot{v}_{ij}(t) = Df( [\tau \Phi](i \sigma_1 + j \sigma_2 + ct) \big) \pi^+_{ij} v(t).
\end{equation}
As can be seen, the linear operator on the right hand side of this system does
not have constant coefficients and hence cannot be diagonalized via Fourier transform.
It can however be partially diagonalized if one takes the Fourier transform
in the direction that is perpendicular to the propagation of the wave, i.e.,
upon taking  $i\sigma_1 + j\sigma_2 + ct =  \mathrm{constant}$.
As explained in \S\ref{sec:green},
one arrives for each transverse spatial frequency $\omega \in [-\pi, \pi]$ at an LDE posed on a
one dimensional lattice that is parallel to the direction of propagation. This LDE is given by
\begin{equation}
\label{eq:mr:1dLDEParallel}
\dot{v}_n(t) =  Df( [\tau \Phi](n + ct) \big) \Big( e^{i \sigma_2 \omega } v_{n+ \sigma_1}(t), e^{- i \sigma_1 \omega }  v_{n + \sigma_2}(t),
  e^{-i \sigma_2 \omega} v_{n - \sigma_1}(t), e^{i \sigma_1 \omega } v_{n - \sigma_2}(t), v_n(t)    \Big),
\end{equation}
in which we have $n =  i \sigma_1 +  j \sigma_2$.
As explained in detail in \cite[{\S}2]{HJHSTBFHN}, there is a close
relationship between the Green's function for the LDE \sref{eq:mr:1dLDEParallel}
and the linear operators
\begin{equation}
\mathcal{L}_{\omega}: W^{1, \infty}(\Real, \Complex^d) \to L^{\infty}(\Real, \Complex^d), \qquad \omega \in [-\pi, \pi],
\end{equation}
that act as
\begin{equation}
[\mathcal{L}_{\omega} p](\xi) = -c p'(\xi) + Df\big( [ \tau  \Phi](\xi) \big) [\tau_{\omega} p](\xi),
\end{equation}
in which the operator $\tau_{\omega}: C(\Real, \Complex^d) \to C(\Real, (\Complex^d)^5)$ is defined by
\begin{equation}
\label{eq:mr:defTauOmega}
[\tau_{\omega} p](\xi) = \big( e^{i\sigma_2\omega} p(\xi+\sigma_1), e^{-i\sigma_1 \omega} p(\xi+\sigma_2),
          e^{-i\sigma_2\omega} p(\xi-\sigma_1),  e^{i\sigma_1 \omega}p(\xi-\sigma_2),p(\xi) \big) \in (\Complex^d)^5.
\end{equation}
Written out in the traditional form used in the landmark paper \cite{MPA}, we have
\begin{equation}
[\mathcal{L}_{\omega} p](\xi) = - c p'(\xi) + \sum_{j=1}^5 A_{\omega,j}(\xi) p(\xi + r_j),
\end{equation}
in which the shifts $r_j$ and operators $A_{\omega,j}$ are given by
\begin{equation}
\begin{array}{lclclcl}
r_1 & = & \sigma_1, & &
  A_{\omega,1}(\xi) & = & e^{i \sigma_2 \omega} D_1 f\big( \Phi(\xi +\sigma_1), \Phi(\xi + \sigma_2), \Phi(\xi - \sigma_1), \Phi(\xi - \sigma_2),\Phi(\xi) \big), \\[0.2cm]
r_2 & = & \sigma_2, & &
  A_{\omega,2}(\xi) & = & e^{ - i \sigma_1 \omega} D_2 f\big( \Phi(\xi +\sigma_1), \Phi(\xi + \sigma_2), \Phi(\xi - \sigma_1), \Phi(\xi - \sigma_2),\Phi(\xi) \big), \\[0.2cm]
r_3 & = & -\sigma_1, & &
  A_{\omega,3}(\xi) & = & e^{ - i \sigma_2 \omega} D_3 f\big( \Phi(\xi +\sigma_1), \Phi(\xi + \sigma_2), \Phi(\xi - \sigma_1), \Phi(\xi - \sigma_2),\Phi(\xi) \big), \\[0.2cm]
r_4 & = & -\sigma_2, & &
  A_{\omega,4}(\xi) & = & e^{ + i \sigma_1 \omega} D_4 f\big( \Phi(\xi +\sigma_1), \Phi(\xi + \sigma_2), \Phi(\xi - \sigma_1), \Phi(\xi - \sigma_2),\Phi(\xi) \big), \\[0.2cm]
r_5 & = & 0, & &
  A_{\omega,5}(\xi) & = &  D_5 f\big( \Phi(\xi +\sigma_1), \Phi(\xi + \sigma_2), \Phi(\xi - \sigma_1), \Phi(\xi - \sigma_2),\Phi(\xi) \big).
\end{array}
\end{equation}
The formal adjoints of these operators are written as
\begin{equation}
\mathcal{L}^*_\omega: W^{1, \infty}(\Real, \Complex^d) \to L^{\infty}(\Real, \Complex^d), \qquad \omega \in [-\pi, \pi],
\end{equation}
and act as
\begin{equation}
[\mathcal{L}^*_\omega q](\xi) =  c q'(\xi) + \sum_{j=1}^5
A^*_{\omega,j}(\xi - r_j) q(\xi - r_j).
\end{equation}
Indeed, the designation of formal adjoint is justified by an easy computation which shows that
\begin{equation}
\int_{-\infty}^{\infty} \langle q(\xi),  [\mathcal{L}_{\omega} p ](\xi) \rangle \, d \xi
= \int_{-\infty}^{\infty} \langle [\mathcal{L}^*_\omega q](\xi), p(\xi) \rangle \, d \xi
\end{equation}
holds for all pairs $p,q \in W^{1, \infty}(\Real, \Complex^d)$, where $\langle \cdot, \cdot \rangle$
denotes the standard inner product on $\Complex^d$.  We emphasize that for the rest of the paper we will reserve the symbol $\<\cdot,\cdot\>$ to denote the standard inner product on $\C^d$ or $\R^d$ with $\ell^2$ ($L^2$) inner products always written out as sums (integrals).

In view of the fact that $r_j \in \Wholes$ holds for all the shifts, an easy computation shows that
\begin{equation}
\label{eq:mr:spCompact}
e_{-2 \pi i  \ell } ( \mathcal{L}_\omega -\lambda) e_{2 \pi i \ell} = \mathcal{L}_\omega - 2 \pi i \ell c - \lambda
\end{equation}
for all $\lambda \in \Complex$ and $\ell \in \Wholes$, in which the exponential shift operator $e_{\nu}$ is defined by
\begin{equation}
[e_{\nu} v](\xi) = e^{\nu \xi} v(\xi).
\end{equation}
In particular, for all $\omega \in [-\pi, \pi]$, the spectrum of $\mathcal{L}_\omega$ is invariant under the operation
$\lambda \mapsto \lambda + 2 \pi i c$. We now proceed to describe the precise assumptions
we need to impose on the spectra of these operators.

Since $\Phi(\xi)$ approaches $u_\pm$ as $\xi \to \pm \infty$,
we can define the characteristic $\Complex^{d \times d}$-valued functions
\begin{equation}
\Delta^\pm_{\omega}(z) = - c  z I + \sum_{j=1}^5 [\lim_{\xi \to \pm \infty} A_{\omega, j}(\xi)] e^{ z r_j}
\end{equation}
for each $\omega \in [-\pi, \pi]$. Our first spectral assumption states that
the characteristic functions associated to $\mathcal{L}_{\omega} - \lambda$ cannot have roots on the imaginary axis if $\Re \lambda \ge 0$.
\begin{itemize}
\item[$\mathrm{(HS)}_{\mathrm{ess}}$]{
  For all $\omega \in [-\pi, \pi]$ and $\lambda \in \Complex$ that have $\Re \lambda \ge 0$,
  we have
  \begin{equation}
    \label{eq:mr:hs1:charEq}
    \det [ \Delta^\pm_{\omega}(i \nu) - \lambda I] \neq 0
  \end{equation}
  for all $\nu \in \Real$.
}
\end{itemize}
Similarly to \sref{eq:mr:spCompact}, we have
\begin{equation}
\Delta^\pm_{\omega}\big(i (\nu + 2 \pi ) \big) =
\Delta^\pm_{\omega}\big( i \nu \big) -   2 \pi  i c I.
\end{equation}
In particular, when verifying \sref{eq:mr:hs1:charEq}
one can restrict $\Im \lambda$ to a compact interval.
The {\it a priori} bound on solutions $\nu \in \Real$ to
the identity $\det [ \Delta^\pm_{\omega}(i \nu) - \lambda I] = 0$
obtained in \cite[Lem. 3.1]{HJHCM}
now shows that \sref{eq:mr:hs1:charEq} in fact holds whenever $\Re \lambda \ge -\beta$ for some $\beta > 0$.
In particular, in view of \cite[Thm. A]{MPA}, $\mathrm{(HS)}_{\mathrm{ess}}$  implies that the operators $\mathcal{L}_{\omega} - \lambda$
are Fredholm whenever $\Re \lambda \ge -\beta$.

%
%


Recall from \sref{eq:mr:defTauWithoutOmega}
and \sref{eq:mr:defTauOmega}
that $\tau_0 = \tau$. It is natural to expect
that the operator $\mathcal{L}_0$ encodes
stability properties of the wave $(c, \Phi)$ under
perturbations that are constant in the direction transverse to propagation.
Indeed, our next three spectral assumptions on $\mathcal{L}_0$
are sufficient to guarantee the nonlinear stability of the wave under such perturbations
\cite{SHEN1999a}.
These assumptions basically state that $\mathcal{L}_0$ has isolated simple point spectrum at $\lambda \in \{2\pi i c \Wholes\}$
and no other spectrum in the half-plane $\Re \lambda \ge -\beta$.
%
\begin{itemize}
\item[$\mathrm{(HS1)}_p$]{
  There exists $\beta > 0$ such that $\mathcal{L}_0 - \lambda$ is invertible as a map from $W^{1, \infty}(\Real, \Complex^d)$
  into $L^{\infty}(\Real, \Complex^d)$ for all $\lambda \setminus 2 \pi i c \Wholes$ that have $\Re \lambda \ge - \beta$.
}
\item[$\mathrm{(HS2)}_p$]{
  The only nontrivial solutions $p \in W^{1, \infty}(\Real, \Complex^d)$ of $\mathcal{L}_0 p = 0$
  are $p = \Phi'$ and scalar multiples thereof.
}
\item[$\mathrm{(HS3)}_p$]{
  The equation $\mathcal{L}_0 p = \Phi'$ does not admit a solution $p \in W^{1, \infty}(\Real, \Complex^d)$.
}
\end{itemize}
%
Throughout the remainder of this paper, we will use the shorthand $\phi:=\Phi'$.

As we have discussed above, the operator $\mathcal{L}_0$ is Fredholm. The continuity of the Fredholm index
together with $\mathrm{(HS1)}_p$ imply that the Fredholm index of $\mathcal{L}_0$ must be zero.
Using \cite[Thm. A]{MPA} together with $\mathrm{(HS2)}_p$, one sees that the kernel of $\mathcal{L}^*_0$ must be one-dimensional.
Another application of \cite[Thm. A]{MPA} now in combination with $\mathrm{(HS3)}_p$ shows that
the kernel of $\mathcal{L}^*_0$ is spanned by a function $\psi \in W^{1, \infty}(\Real,\Real^d)$
that can be uniquely fixed by the identity
\begin{equation}
\label{eq:mr:normalizatonConditionOnPsi}
\int_{-\infty}^{\infty} \langle \psi(\xi),  \phi(\xi) \rangle d \xi = 1,
\end{equation}
since $\phi$ does not lie in the range of $\mathcal{L}_0$.
In addition, both $\phi(\xi)$ and $\psi(\xi)$
decay exponentially as $\xi \to \pm \infty$.
We now provide a useful identity that allows
the normalization \sref{eq:mr:normalizatonConditionOnPsi} to be extended
to the discrete setting.
\begin{lem}
For every $\vartheta \in \Real$ we have
\begin{equation}
\label{eq:mr:innProduct:alwaysEq}
\sum_{n \in \Wholes } \langle  \psi(n + \vartheta),  \phi(n + \vartheta )  \rangle =
 \int_{-\infty}^\infty \langle \psi(\xi) ,  \phi(\xi)  \rangle d \xi = 1.
\end{equation}
\end{lem}
\begin{proof}
Inspired by the Hale inner product, we define
\begin{equation}
\begin{array}{lcl}
\mathcal{J}(\xi) & = &
c \langle \psi(\xi),  \phi(\xi) \rangle
- \sum_{j=1}^5 \int_0^{r_j} \langle \psi(\xi + \zeta - r_j), A_{0,j}(\xi + \zeta - r_j) \phi(\xi + \zeta) \rangle \, d \zeta
\\[0.2cm]
& = &
c \langle \psi(\xi),  \phi(\xi) \rangle
- \sum_{j=1}^5 \int_{\xi}^{\xi+ r_j} \langle \psi( \zeta - r_j), A_{0,j}( \zeta - r_j)
\phi( \zeta) \rangle \, d \zeta
\end{array}
\end{equation}
and compute
\begin{equation}
\begin{array}{lcl}
\mathcal{J}'(\xi) & = &
c \langle \psi'(\xi), \phi(\xi) \rangle + c \langle \psi(\xi), \phi'(\xi) \rangle
\\[0.2cm]
& & \qquad - \sum_{j=1}^5 \langle \psi( \xi) , A_{0,j}( \xi) \phi(\xi + r _j) \rangle
 + \sum_{j=1}^5 \langle \psi(\xi - r_j), A_{0,j}(\xi - r_j) \phi(\xi) \rangle
\\[0.2cm]
& = & - \langle \psi(\xi), [\mathcal{L}_0 \phi](\xi) \rangle
+ \langle [\mathcal{L}_0^* \psi](\xi), \phi(\xi) \rangle
\\[0.2cm]
& = & 0.
\end{array}
\end{equation}
Since $\lim_{\xi \to \pm \infty} \mathcal{J}(\xi) = 0$, we conclude that
in fact $\mathcal{J}(\xi) = 0$ for all $\xi \in \Real$.
Exploiting the fact that $r_j \in \Wholes$, we hence have
\begin{equation}
\begin{array}{lcl}
c \sum_{n \in \Wholes } \langle  \psi(n + \vartheta),  \phi(n + \vartheta )  \rangle
& = &
\sum_{j=1}^5 \sum_{n \in \Wholes }
\int_{n + \vartheta}^{n + \vartheta + r_j} \langle \psi( \zeta - r_j), A_{0,j}( \zeta - r_j) \phi(\zeta) \rangle \, d \zeta
\\[0.2cm]
& = & \sum_{j=1}^5 r_j
\int_{-\infty}^{\infty} \langle \psi( \zeta - r_j), A_{0,j}( \zeta - r_j) \phi(\zeta) \rangle \, d \zeta.
\end{array}
\end{equation}
Similarly, we have
\begin{equation}
\begin{array}{lcl}
c \int_{-\infty}^{\infty} \langle  \psi(\xi),  \phi(\xi )  \rangle \, d \xi
& = &
\sum_{j=1}^5 \int_{-\infty}^{\infty}
\int_{\xi}^{\xi + r_j} \langle \psi( \zeta - r_j), A_{0,j}( \zeta - r_j) \phi(\zeta) \rangle \, d \zeta \, d \xi
\\[0.2cm]
& = &
\sum_{j=1}^5 \int_{-\infty}^{\infty}
\int_{\zeta - r_j}^{\zeta} \langle \psi( \zeta - r_j), A_{0,j}( \zeta - r_j) \phi(\zeta) \rangle \, d \xi \, d \zeta
\\[0.2cm]
& = & \sum_{j=1}^5 r_j
\int_{-\infty}^{\infty} \langle \psi( \zeta - r_j), A_{0,j}( \zeta - r_j) \phi(\zeta) \rangle \, d \zeta,
\end{array}
\end{equation}
which completes the proof.
\end{proof}




Before we proceed, we introduce for any $\zeta \in \Real$ the exponentially weighted spaces
\begin{equation}
\begin{array}{lcl}
BC_{ \zeta } ( \Real, \Complex^d )  & = & \{ p \in C( \Real, \Complex^d) \mid
\sup_{\xi \in \Real } e^{ - \zeta \abs{\xi } } \abs{p(\xi) } < \infty \}, \\[0.2cm]
BC^1_{ \zeta } ( \Real, \Complex^d )  & = & \{ p \in C^1( \Real, \Complex^d) \mid
\sup_{\xi \in \Real } e^{ - \zeta \abs{\xi } } [\abs{p(\xi) } + \abs{p'(\xi)} ] < \infty \}.
\end{array}
\end{equation}
The following perturbation result is a direct
consequence of the assumptions $\mathrm{(HS1)}_p$-$\mathrm{(HS3)}_p$.
It basically states that
there is a branch of simple eigenvalues $\lambda_\omega$,
for $\mathcal{L}_\omega$ with corresponding eigenfunctionss
$\phi_\omega$ and adjoint eigenfunctions $\psi_\omega$.
These eigenvectors decay exponentially at a rate that can be
bounded away from zero uniformly in $\omega$. In addition,
with the exception of simple eigenvalues at $\lambda_\omega + 2\pi i c \Z$, the spectrum of $\mathcal{L}_\omega$
lies to the left of the line $\Re \lambda = -\beta$.

\begin{prop}
\label{prp:mr:melnikov}
Assume that (Hf), (H$\Phi$), $\mathit{(HS)}_{\mathrm{ess}}$ and  $\mathit{(HS1)}_p$-$\mathit{(HS3)}_p$ all hold. Then
there exists a constant $0 < \omega_0 \ll 1$
together with triplets
\begin{equation}
  (\lambda_{\omega}, \phi_{\omega}, \psi_{\omega} ) \in \Complex \times W^{1, \infty}(\Real, \Complex^d) \times W^{1, \infty}(\Real, \Complex^d),
\end{equation}
defined for each  $\omega \in (-\omega_0, \omega_0)$, such that the following hold true.
\begin{itemize}
\item[(i)]{
  There exists $\beta > 0$ such that
  for all $\omega \in (-\omega_0, \omega_0)$,
  the operator $\mathcal{L}(\omega) - \lambda$
  is invertible as a map from $W^{1, \infty}(\Real, \Complex^d)$ into $L^\infty(\Real, \Complex^d)$,
  for all
  $\lambda \in \Complex$ that have $\Re \lambda \ge -\beta$
  and for which $\lambda - \lambda_{\omega} \notin  2 \pi i c \Wholes$.
}
\item[(ii)]{
  The only nontrivial solutions $p \in W^{1, \infty}(\Real, \Complex^d)$ of $(\mathcal{L}_\omega - \lambda_{\omega})p = 0$
  are $p = \phi_{\omega}$ and scalar multiples thereof.
}
\item[(iii)]{
  The only nontrivial solutions $q \in W^{1, \infty}(\Real, \Complex^d)$ of $(\mathcal{L}^*_\omega - \lambda^*_{\omega})q = 0$
  are $q = \psi_{\omega}$ and scalar multiples thereof.
}
\item[(iv)]{
  For all $\omega \in (-\omega_0, \omega_0)$ we have
  the normalization
  \begin{equation}
  \int_{-\infty}^{\infty} \langle \psi_{\omega}(\xi) , \phi_{\omega}(\xi) \rangle \, d \xi = 1.
  \end{equation}
  In particular, the equation $(\mathcal{L}_\omega - \lambda_{\omega} )v = \phi_{\omega}$
  does not admit a solution $v \in W^{1, \infty}(\Real, \Complex^d)$.
}
\item[(v)]{
  There exists $\eta > 0 $ such that the maps
  \begin{equation}
    \omega \mapsto \left\{ \begin{array}{ll}
      \lambda_{\omega} & \in \Complex, \\
      \phi_{\omega} & \in BC_{ - \eta }(\R, \Complex^d ) \\
      \psi_{\omega} & \in BC_{ - \eta }(\R, \Complex^d ) \\
    \end{array}\right.
  \end{equation}
  are all well-defined and $C^2$-smooth.
}
\item[(vi)]{
  We have $\lambda_{0} = 0$, $\phi_0 = \phi$ and $\psi_0 = \psi$.
}
\end{itemize}
\end{prop}
\begin{proof}
The key observation is that for sufficiently small $\eta > 0$,
the operator $\mathcal{L}_0$ is also a Fredholm operator with index zero
when viewed as a map from $BC^1_{- \eta}(\Real, \Complex^d)$ into $BC_{-\eta}(\Real, \Complex^d)$.
This can be seen as in \cite[Lem. 4.7]{HJHNLS} by observing that the operator
\begin{equation}
  \widetilde{\mathcal{L}}_0: p \mapsto \mathcal{L}_0 [\cosh ( \eta \cdot )]^{-1} p( \cdot)
\end{equation}
is an operator of the form covered by the theory in \cite{MPA} that in addition
is asymptotically hyperbolic. In particular, it is a Fredholm operator
mapping $BC^1_0(\Real, \Complex^d)$ into $BC_0(\Real, \Complex^d)$
with a one-dimensional kernel spanned by the function
$\xi \mapsto \cosh(\eta \xi) \phi(\xi)$, which is bounded
provided that $\eta > 0$ is chosen to be smaller than the exponential
decay rates of $\phi$.
In addition, the index of $\widetilde{\mathcal{L}}_0$ viewed as a map between these two spaces is zero,
which can be concluded from the spectral flow formula in \cite[Thm. C]{MPA}.

All the properties stated above now follow from regular perturbation.
Indeed, the map $\omega \mapsto \mathcal{L}_\omega$ is analytic and the perturbation
$\mathcal{L}_0 - \mathcal{L}_\omega$ is bounded.  More explicitly,
consider the range
\begin{equation}
\mathcal{R} = \mathcal{L}_0 \big( BC^1_{-\eta}(\Real, \Complex^d) \big) \subset BC_{-\eta}(\Real, \Complex^d)
\end{equation}
and define $\mathcal{F} : \mathcal{R} \times \C \times \R \to BC_{-\eta}(\Real, \Complex^d)$
by
\begin{equation}
\mathcal{F}(p,\lambda,\omega) = (\mathcal{L}_\omega - \lambda)(\phi+p).
\end{equation}
Observe that $\mathcal{F}(0,0,0) = 0$ and compute the derivative
\begin{equation}
D_{p,\lambda}\mathcal{F}(0,0,0): (\tilde{p},\tilde{\lambda}) \mapsto \mathcal{L}_0 \tilde{p} - \tilde{\lambda}\phi,
\end{equation}
which is injective as a consequence of $\phi \notin \mathcal{R}$ and $\mathrm{(HS2)}_p$
and surjective because $BC_{-\eta}(\Real, \Complex^d) = \mathcal{R} \oplus \mathrm{span} \{ \phi \}$.
The desired curves $\omega \mapsto (\lambda_{\omega}, \phi_\omega)$
are now a direct consequence of the implicit function theorem, upon writing $\phi_{\omega} = \phi + p_{\omega}$.
The remaining curve $\omega \mapsto \psi_{\omega}$ can be constructed in a similar fashion.
\end{proof}

The first of our two final assumptions ensures
that the curve $\omega \to \lambda_{\omega}$ touches the origin in a quadratic tangency
that opens up on the left side of the imaginary axis.
The second assumption guarantees that the only part of the spectrum of the operators $\mathcal{L}_{\omega}$
with $\omega \in [- \pi, \pi]$
that can potentially lie in the half-plane $\Re \lambda \ge -\beta$
are the simple isolated eigenvalues $\lambda_{\omega} + 2 \pi i c \Wholes$
defined for $\omega \approx 0$.

\begin{itemize}
\item[$\mathrm{(HS1)}_\omega$]{
Recalling the curves $\omega \mapsto (\lambda_{\omega}, \phi_{\omega})$
defined in Proposition \ref{prp:mr:melnikov}, we have the Melnikov identity
\begin{equation}
\label{eq:mr:melnikov}
\begin{array}{lcl}
\int_{-\infty}^\infty \langle \psi(\xi),  \big[ [\partial_{\omega}^2 \mathcal{L}_\omega]_{\omega = 0} \phi ](\xi) \rangle \, d \xi & & \\[0.2cm]
 \qquad + 2\int_{-\infty}^\infty \langle \psi(\xi), \Big[ \big( [\partial_\omega \mathcal{L}_\omega]_{\omega = 0} -
[\partial_\omega \lambda_\omega]_{\omega = 0} \big) [\partial_\omega\phi_\omega]_{\omega = 0}\Big](\xi) \rangle \, d \xi & \ne & 0.
\end{array}
\end{equation}
}
\item[$\mathrm{(HS2)}_\omega$]{
  For all $\omega \in [-\pi, \pi] \setminus \{ 0 \}$
  and $\lambda \in \Complex$ that have $\Re \lambda \ge 0$,
  the operator $\mathcal{L}_\omega - \lambda$
  is invertible as a map from $W^{1, \infty}(\Real, \Complex^d)$ into $L^\infty(\Real, \Complex^d)$.
}
\end{itemize}

\begin{prop}
\label{prp:mr:locRoots}
Consider the setting of Proposition \ref{prp:mr:melnikov}.
We have $[\partial_{\omega}\lambda_\omega]_{\omega = 0} \in i\R$
and $[\partial_{\omega} \phi_{\omega}]_{\omega = 0}$ takes values in $i \Real^d$.
In addition, if $\mathrm{(HS1)}_\omega$ also holds, then there exists a  constant $\kappa > 0 $ such that
\begin{equation}
  \label{eq:prp:mr:locRoots:Quadratic}
 \Re \lambda_\omega < -\kappa \omega^2
\end{equation}
for all $\omega \in (-\omega_0, \omega_0)$.
\end{prop}
\begin{proof}
To address the inclusion claims, we Taylor expand the quantities
$\lambda_\omega$, $\phi_\omega$ and $\psi_\omega$ in $\omega$ about $\omega = 0$ in the expression
$(\mathcal{L}_\omega - \lambda_\omega)\phi_\omega = 0$ to obtain
\begin{equation}
\label{eq:mel1}
\begin{array}{lcl}
 O(\omega^3) & = &
\mathcal{L}_0\phi + \omega\Big(  \mathcal{L}_0 [\partial_{\omega}\phi_\omega]_{\omega = 0}
+ [\partial_{\omega}\mathcal{L}_\omega- \partial_{\omega} \lambda_\omega]_{\omega = 0} \phi \Big)
\\[0.2cm]
& & \qquad
+ \frac{\omega^2}{2}\Big( \mathcal{L}_0 [\partial_\omega^2 \phi_\omega]_{\omega = 0}
+ 2[\partial_\omega \mathcal{L}_\omega - \partial_\omega \lambda_\omega]_{\omega = 0} [\partial_\omega \phi_\omega]_{\omega = 0}
+ [\partial_\omega^2\mathcal{L}_\omega - \partial_\omega^2 \lambda_\omega]_{\omega = 0} \phi \Big).
\end{array}
\end{equation}
Taking the inner product against $\psi$ and recalling that $\mathcal{L}_0^* \psi= 0$ leaves
\begin{equation}
 [\partial_{\omega}\lambda_\omega]_{\omega = 0} = \int_{-\infty}^\infty \langle \psi(\xi),
   \big[[\partial_{\omega}\mathcal{L}_\omega]_{\omega = 0} \phi\big](\xi) \rangle \, d \xi.
\end{equation}
Inspecting the definition of $\mathcal{L}_\omega$
shows that $i[\partial_{\omega} \mathcal{L}_{\omega}]_{\omega = 0}$
maps real-valued functions to real-valued functions.
This together with the fact that $\phi$ and $\psi$ are real-valued implies that $i [\partial_\omega \lambda_\omega]_{\omega = 0}$ is real.

In order to obtain \sref{eq:prp:mr:locRoots:Quadratic},
it suffices to integrate the $O(\omega^2)$ term in \sref{eq:mel1} against $\psi$, yielding
\begin{equation}
\begin{array}{lcl}
[\partial^2_{\omega} \lambda_{\omega} ]_{\omega = 0} & = &
 \int_{-\infty}^\infty \langle \psi(\xi),  \big[ [\partial_{\omega}^2 \mathcal{L}_\omega]_{\omega = 0} \phi \big](\xi) \rangle \, d \xi \\[0.2cm]
& & \qquad + 2\int_{-\infty}^\infty \langle \psi(\xi), \Big[ [\partial_\omega \mathcal{L}_\omega -
\partial_\omega \lambda_\omega]_{\omega = 0} [\partial_\omega\phi_\omega]_{\omega = 0} \Big](\xi) \rangle \, d \xi.
\label{eq:mel2}
\end{array}
\end{equation}
\end{proof}


\begin{prop}
\label{prp:mr:omegaLarge}
Assume that (Hf), (H$\Phi$), $\mathit{(HS)}_{\mathrm{ess}}$,
$\mathit{(HS1)}_p$-$\mathit{(HS3)}_p$ and $\mathit{(HS2)}_{\omega}$  all hold
and recall the constant $0 < \omega_0 \ll 1$ from Proposition \ref{prp:mr:melnikov}.
Then there exists a constant $\beta > 0$
such that for all $\lambda \in \Complex$ that have  $\Re \lambda \ge - \beta$
and all $\omega \in [-\pi, \pi]$ that have $\abs{ \omega } \ge \omega_0$,
the operator $\mathcal{L}_\omega - \lambda$
is invertible as a map from $W^{1, \infty}(\Real, \Complex^d)$ into $L^\infty(\Real, \Complex^d)$.
\end{prop}
\begin{proof}
This follows from the continuous dependence of $\mathcal{L}_\omega$ on $\omega$,
the fact that resolvent sets are open and the
identity \sref{eq:mr:spCompact}
which shows we only have to consider spectral parameters $\lambda$ in compact sets.
\end{proof}

We are now almost ready to state our main results.
First, we introduce for any $\vartheta \in \Real$ the function $\chi(\vartheta) \in \ell^{\infty}(\Wholes; \Real^d)$
defined by
\begin{equation}
\chi_n(\vartheta)  = -i [\partial_{\omega} \phi_{\omega}]_{ \omega = 0} (n + \vartheta)
 + i \phi(n + \vartheta)
  \sum_{\ell \in \Wholes} \langle \psi(\ell + \vartheta) , [\partial_{\omega} \phi_{\omega}]_{\omega = 0} (\ell + \vartheta) \rangle.
\end{equation}
This can be seen as the residual of the $\Real^d$-valued function $-i [\partial_{\omega} \phi_{\omega}]_{ \omega = 0}$ after projecting out
the part along $\phi(\cdot + \vartheta)$.

\begin{thm}
\label{thm:mr:stbAnyDirection}
Consider the LDE \sref{eq:mr:mainLDE} and suppose that
(Hf), (H$\Phi$),  $\mathit{(HS)}_{\mathrm{ess}}$,
$\mathit{(HS1)}_p$-$\mathit{(HS3)}_p$ and $\mathit{(HS1)}_\omega$-$\mathit{(HS2)}_{\omega}$ are all satisfied. Pick any $ \alpha < \frac{5}{4}$ and $1 \le p \le \infty$.
Then there exist constants $\delta > 0$ and $C > 0$
such that for any $v^0 \in \ell^{\infty}(\Wholes^2 ; \Real^d)$
that can be bounded\footnote{
 The $\ell^p$ norm is taken with respect to the first index and the $\ell^1$ norm with respect to the second index.
}
by
\begin{equation}
\norm{v^0}_{\ell^p\big(\Wholes; \ell^1(\Wholes; \Real^d) \big) } \le \delta,
\end{equation}
there exist $C^1$-smooth functions $\theta: [0, \infty) \to \ell^\infty(\Wholes ; \Real)$
and $w: [0, \infty) \to \ell^{\infty}(\Wholes^2; \Real^d)$
that satisfy the following properties.
\begin{itemize}
\item[(i)]{
  The function
  \begin{equation}
    \begin{array}{lcl}
    u_{ij}(t) & = & \Phi(i\sigma_1  + j\sigma_2  + ct + \theta_{i\sigma_2  - j\sigma_1  }(t) \big) \\[0.2cm]
     & & \qquad + \chi_{i \sigma_1 + j \sigma_2} (ct)
       \big( \theta_{i\sigma_2  - j \sigma_1  + 1 }(t) - \theta_{ i\sigma_2  - j \sigma_1  }(t) \big) \\[0.2cm]
     & & \qquad + w_{ i\sigma_1  + j \sigma_2  , i\sigma_2  - j \sigma_1  }(t) \\[0.2cm]
     \end{array}
  \end{equation}
  satisfies \sref{eq:mr:mainLDE} for all $t \ge 0$ with the initial condition
  \begin{equation}
    u_{ij}(0) = \Phi(i\sigma_1  +  j\sigma_2\big)
     + v^0_{ i\sigma_1  + j\sigma_2  , i \sigma_2 - j\sigma_1  }.
  \end{equation}
}
\item[(ii)]{
  For all $t \ge 0$ and all $l \in \Wholes$ we have
  \begin{equation}
    \sum_{n \in \Wholes} \langle \psi( n + ct ) , w_{n l}(t) \rangle = 0.
  \end{equation}
}
\item[(iii)]{
  The function $\theta$ satisfies the estimates
  \begin{equation}
  \begin{array}{lcl}
    \norm{\theta(t)}_{\ell^2(\Wholes; \Real)} & \le & C (1 + t)^{-1/4}
       \norm{v^0}_{\ell^p\big(\Wholes; \ell^1(\Wholes; \Real^d) \big) }, \\[0.2cm]
    \norm{\theta(t)}_{\ell^\infty(\Wholes; \Real)} & \le & C (1 + t)^{-1/2}
       \norm{v^0}_{\ell^p\big(\Wholes; \ell^1(\Wholes; \Real^d) \big) }. \\[0.2cm]
  \end{array}
  \end{equation}
}
\item[(iv)]{
  The function $\theta_{\mathrm{diff}}: [0, \infty) \to \ell^\infty(\Wholes; \Real)$
  defined by
  \begin{equation}
    [\theta_{\mathrm{diff}}]_l(t) = \theta_{l+1}(t) - \theta_l(t)
  \end{equation}
  satisfies the estimates
  \begin{equation}
  \begin{array}{lcl}
    \norm{\theta_{\mathrm{diff}}(t)}_{\ell^2(\Wholes; \Real)} & \le & C (1 + t)^{-3/4}
       \norm{v^0}_{\ell^p\big(\Wholes; \ell^1(\Wholes; \Real^d) \big) }, \\[0.2cm]
    \norm{\theta_{\mathrm{diff}}(t)}_{\ell^\infty(\Wholes; \Real)} & \le & C (1 + t)^{-1}
       \norm{v^0}_{\ell^p\big(\Wholes; \ell^1(\Wholes; \Real^d) \big) }. \\[0.2cm]
  \end{array}
  \end{equation}
}
\item[(v)]{
  The function $w$ satisfies the estimates
  \begin{equation}
  \begin{array}{lcl}
    \norm{w(t)}_{\ell^p\big(\Wholes; \ell^2(\Wholes; \Real^d) \big) } & \le & C (1 + t)^{-\alpha}
      \norm{v^0}_{\ell^p\big(\Wholes; \ell^1(\Wholes; \Real^d) \big) }, \\[0.2cm]
    \norm{w(t)}_{\ell^p\big(\Wholes; \ell^\infty(\Wholes; \Real^d) \big) } & \le & C (1 + t)^{- 5/4 }
      \norm{v^0}_{\ell^p\big(\Wholes; \ell^1(\Wholes; \Real^d) \big) }. \\[0.2cm]
  \end{array}
  \end{equation}
}
\end{itemize}
\end{thm}

Our final two results cover situations where the wave travels in
a horizontal, vertical or diagonal direction,
in which case the decay rates on $v(t)$ can be improved.
Notice that we only recover the bounds obtained by Kapitula \cite{KAP1997}
when analyzing waves that travel horizontally or vertically
and for which the perpendicular discrete diffusion coefficients
are multiples of the identity matrix.
In fact, in this case we even get slightly faster decay
rates, because we can exploit the embedding
$\ell^1(\Wholes; \Real) \subset \ell^2(\Wholes; \Real) \subset \ell^\infty(\Wholes; \Real)$,
which does not generalize to the function spaces $L^{q}(\Real; \Real)$.

\begin{thm}
\label{thm:mr:stbDiagonal}
Consider the LDE \sref{eq:mr:mainLDE} and suppose that
(Hf), (H$\Phi$),  $\mathit{(HS)}_{\mathrm{ess}}$,
$\mathit{(HS1)}_p$-$\mathit{(HS3)}_p$ and $\mathit{(HS1)}_\omega$-$\mathit{(HS2)}_{\omega}$ are all satisfied.
Pick any matrix $A \in \Real^{d \times d}$.
Suppose that
either $(\sigma_1, \sigma_2) = (1, 0)$ with
\begin{equation}
f\big(  u_{i+1, j}, u_{i, j+1}, u_{i-1, j}, u_{i, j-1} , u_{ij} \big)
= A [ u_{i, j + 1} + u_{i, j-1} ] + g\big(u_{i+1, j} , u_{i - 1, j} , u_{ij} \big),
\end{equation}
or that $(\sigma_1, \sigma_2) = ( 0, 1)$ with
\begin{equation}
f\big(  u_{i+1, j}, u_{i, j+1}, u_{i-1, j}, u_{i, j-1} , u_{ij} \big)
= A [ u_{i +1, j} + u_{i-1, j} ] + g\big(u_{i, j+1} , u_{i , j -1} , u_{ij} \big),
\end{equation}
or that
$(\sigma_1, \sigma_2) = ( 1, 1)$ with
\begin{equation}
f\big(  u_{i+1, j}, u_{i, j+1}, u_{i-1, j}, u_{i, j-1} , u_{ij} \big)
= A [   u_{i+1, j} + u_{i, j+1} + u_{i-1, j} + u_{i, j-1}    ] + g\big( u_{ij} \big).
\end{equation}
Then we have $[\partial_{\omega} \phi_{\omega}]_{\omega = 0} = 0$.
In addition, the statements in Theorem \ref{thm:mr:stbAnyDirection} all hold,
after replacing (v) by
\begin{itemize}
\item[(v)']{
  The function $v = w$ satisfies the estimates
  \begin{equation}
  \begin{array}{lcl}
    \norm{w(t)}_{\ell^p\big(\Wholes; \ell^2(\Wholes; \Real^d) \big) } & \le & C (1 + t)^{-5/4}
      \norm{v^0}_{\ell^p\big(\Wholes; \ell^1(\Wholes; \Real^d) \big) }, \\[0.2cm]
    \norm{w(t)}_{\ell^p\big(\Wholes; \ell^\infty(\Wholes; \Real^d) \big) } & \le & C (1 + t)^{- 3/2 }
      \norm{v^0}_{\ell^p\big(\Wholes; \ell^1(\Wholes; \Real^d) \big) }. \\[0.2cm]
  \end{array}
  \end{equation}
}
\end{itemize}
\end{thm}

In the special case that $\sigma$ is one of the lattice directions (1,0) or (0,1) and moreover the coupling in the transverse direction is linear with weights that are equal both across components in the vector $u$ and from one neighbour to the other, we obtain stronger decay.
While this case is certainly restrictive, it also holds in our motivating example \eqref{eq:int:nagumo:lde} and so we state it here as a theorem.

\begin{thm}
\label{thm:mr:stbParallel}
Consider the LDE \sref{eq:mr:mainLDE} and suppose that
(Hf), (H$\Phi$),  $\mathit{(HS)}_{\mathrm{ess}}$,
and $\mathit{(HS1)}_p$-$\mathit{(HS3)}_p$
are all satisfied.
Pick any constant $\kappa > 0$.
Suppose that either $(\sigma_1, \sigma_2) = (1, 0)$ with
\begin{equation}
f\big(  u_{i+1, j}, u_{i, j+1}, u_{i-1, j}, u_{i, j-1} , u_{ij} \big)
= \kappa [ u_{i, j + 1} + u_{i, j-1} ] + g\big(u_{i+1, j} , u_{i - 1, j} , u_{ij} \big),
\end{equation}
or that $(\sigma_1, \sigma_2) = ( 0, 1)$ with
\begin{equation}
f\big(  u_{i+1, j}, u_{i, j+1}, u_{i-1, j}, u_{i, j-1} , u_{ij} \big)
= \kappa [ u_{i +1, j} + u_{i-1, j} ] + g\big(u_{i, j+1} , u_{i , j -1} , u_{ij} \big).
\end{equation}
Then the triplets $(\lambda_{\omega} , \phi_{\omega}, \psi_{\omega})$
are defined for all $\omega \in [-\pi, \pi]$ and we have $\phi_{\omega}  = \phi$ for all $\omega \in [-\pi, \pi]$.
In addition, the statements in Theorem \ref{thm:mr:stbAnyDirection} all hold,
after replacing (v) by
\begin{itemize}
\item[(v)'']{
  The function $v = w$ satisfies the estimates
  \begin{equation}
  \begin{array}{lcl}
    \norm{w(t)}_{\ell^p\big(\Wholes; \ell^2(\Wholes; \Real^d) \big) } & \le & C
      (1 + t)^{-7/4}
      \norm{v^0}_{\ell^p\big(\Wholes; \ell^1(\Wholes; \Real^d) \big) }, \\[0.2cm]
    \norm{w(t)}_{\ell^p\big(\Wholes; \ell^\infty(\Wholes; \Real^d) \big) } & \le & C (1 + t)^{- 7/4 }
      \norm{v^0}_{\ell^p\big(\Wholes; \ell^1(\Wholes; \Real^d) \big) }. \\[0.2cm]
  \end{array}
  \end{equation}
}
\end{itemize}
\end{thm}

\section{Coordinate System}
\label{sec:cds}

In this section we develop the coordinate system that we
use to analyze the planar nonlinear LDE \sref{eq:mr:mainLDE}.
We start by introducing the new coordinates
\begin{equation}
\label{eq:cds:directions}
\begin{array}{lcl}
n & = & i \sigma_1 + j \sigma_2, \\[0.2cm]
l & = & i \sigma_2 - j \sigma_1.
\end{array}
\end{equation}
The first of these coordinates represents the direction parallel to the propagation of the wave,
while the second coordinate represents the direction perpendicular to wave motion.  In the sequel we refer to $n$ as the wave coordinate and $l$ as the transverse coordinate.

We emphasize that we always have $(n,l) \in \Wholes^2$
due to our assumption that $(\sigma_1, \sigma_2) \in \Wholes^2$. However,
the inverse of the transformation \sref{eq:cds:directions}
is given by
\begin{equation}
\begin{array}{lcl}
i & = & [\sigma_1^2 + \sigma_2^2]^{-1}\big( n \sigma_1 + l \sigma_2 \big), \\[0.2cm]
j & = & [\sigma_1^2 + \sigma_2^2]^{-1}\big( n \sigma_2 - l \sigma_1 \big), \\[0.2cm]
\end{array}
\end{equation}
which means that the pairs $(n,l)$ in the range of the transformation \sref{eq:cds:directions}
represent only a sublattice of $\Wholes^2$. Indeed, rewriting the LDE \sref{eq:mr:mainLDE}
in terms of our new coordinates, we obtain the system
\begin{equation}
\label{eq:cds:ldeNewCoordsLoc}
\dot{u}_{nl}(t) = f\big( \pi^\times_{nl} \, u(t) \big),
\end{equation}
in which the linear operator $\pi^\times_{nl} :\ell^{\infty}(\Wholes^2; \Real^d) \to (\R^d)^5$
is given by
\begin{equation}
\pi^\times_{nl} u = \big( u_{n + \sigma_1, l + \sigma_2} , u_{n + \sigma_2, l - \sigma_1} , u_{n - \sigma_1, l - \sigma_2} ,  u_{n - \sigma_2, l + \sigma_1} ,u_{nl}  \big).
\end{equation}
As an example, the new coordinates transform the discrete Nagumo LDE \sref{eq:int:nagumo:lde}
into the system
\begin{equation}
\dot{u}_{nl} = [\Delta^\times u]_{nl} + g ( u_{nl}(t) ),
\label{eq:lde:newCoords}
\end{equation}
in which
\begin{equation}
[\Delta^\times u ]_{nl} = u_{n + \sigma_1, l + \sigma_2} + u_{n + \sigma_2, l - \sigma_1} + u_{n - \sigma_1, l - \sigma_2} + u_{n - \sigma_2, l + \sigma_1} - 4 u_{nl}.
\end{equation}

To avoid clutter, we also introduce the operator
\begin{equation}
\pi^\times: \ell^{\infty}(\Wholes^2; \Real^d) \to \ell^\infty(\Z^2,(\R^d)^5)
\end{equation}
that acts as
\begin{equation}
[\pi^\times u]_{nl} = \pi^\times_{nl} u.
\end{equation}
The operator $\pi^\times$ should be regarded as a shorthand for generating a stencil of grid points.
Note that when $(\sigma_1, \sigma_2)$ is equal to $(0,1)$ or $(1,0)$,
we have $\pi^\times = \pi^+$ as introduced in \S\ref{sec:mr}.
Lifting the nonlinearity $f$ in the obvious fashion
to act on $\ell^{\infty}(\Wholes^2; (\Real^d)^5)$ instead of $(\Real^d)^5$,
we can now rephrase \sref{eq:cds:ldeNewCoordsLoc}
in the compact form
\begin{equation}
\label{eq:cds:ldeNewCoordsGlb}
\dot{u}(t) = f\big( \pi^\times  u(t) \big).
\end{equation}
In the sequel, we often refer to
\sref{eq:cds:ldeNewCoordsLoc} as an equation in local form,
while \sref{eq:cds:ldeNewCoordsGlb} is called an equation in global form.

Let us now consider any bounded function $\zeta: \Real \to \Real^s$,
where we will use both $s = d$ and $s = 5d$. For any shift $\vartheta \in \Real$,
we introduce the sequence
$T_{\vartheta} \zeta \in \ell^{\infty}(\Wholes; \Real^s)$
that is defined by
\begin{equation}
[T_{\vartheta} \zeta]_n = \zeta( n + \vartheta ).
\end{equation}
In particular, $T_{\vartheta}$ converts a function defined on $\Real$ into a sequence
by evaluating the function at the points $\Wholes + \vartheta$.
Furthermore, for any $\theta \in \ell^{\infty}(\Wholes; \Real)$,
we write $T_{\theta} \zeta \in \ell^{\infty}(\Wholes^2; \Real^s)$
for the planar sequence that is defined by
\begin{equation}
[T_{\theta} \zeta]_{nl} = \zeta( n + \theta_l ).
\end{equation}
By interpreting a shift $\vartheta \in \Real$ as a constant sequence in $\ell^{\infty}(\Wholes; \Real)$,
we note that we can also write
$T_{\vartheta} \zeta \in \ell^{\infty}(\Wholes^2; \Real^s)$ to
represent the planar sequence
\begin{equation}
[T_{\vartheta} \zeta]_{nl} = \zeta( n + \vartheta ),
\end{equation}
which is constant in the second coordinate.

Using this notation, we observe that the travelling wave solution
$u_{nl}(t) = \Phi(n + ct)$ to \sref{eq:cds:ldeNewCoordsGlb} can be written as
\begin{equation}
u(t) = T_{ct} \Phi.
\end{equation}
In order to understand the stability of this travelling wave, we
set out to seek solutions to \sref{eq:cds:ldeNewCoordsGlb} that
can be written as
\begin{equation}
\label{eq:cds:ansatzGlb}
u(t) = T_{ct + \theta(t) } \Phi + v(t),
\end{equation}
which in local form is equivalent to
\begin{equation}
\label{eq:cds:ansatzLoc}
u_{nl}(t) = \Phi\big(n + ct + \theta_l(t) \big) + v_{nl}(t).
\end{equation}
Here the function
\begin{equation}
\theta: [0, \infty) \to \ell^\infty(\Wholes; \Real)
\end{equation}
represents a deformation of the travelling wave profile in the transverse direction,
while the remainder function
\begin{equation}
v: [0, \infty) \to \ell^{\infty}(\Wholes^2; \Real^d)
\end{equation}
contains any perturbations that are 'transverse' to this deformation in an appropriate sense.
Indeed, allowing $\theta$ to vary in $\ell^\infty(\Wholes; \Real)$ buys us $\Z$-many degrees of freedom which we can use to restrict
$v$ to live in a favourable subspace.

In order to describe the different norms we employ for the function $v$,
we introduce the sequence spaces
\begin{equation}
\mathcal{X}_{p, q} = \ell^{p}(\Wholes; \ell^q(\Wholes; \Real^d)),
\end{equation}
parametrized by pairs of exponents $1 \le p \le \infty$ and $1 \le q \le \infty$.
A sequence $v \in \mathcal{X}_{p, q}$ can be interpreted as a planar sequence
$v \in \ell^{\infty}(\Wholes^2; \Real^d)$ by the identification
\begin{equation}
v_{nl} =  (v_{n})_{l}.
\end{equation}
In particular, the $\ell^p$ norm is taken along the direction of the wave
and the $\ell^q$ norm is taken in the transverse direction.

Before we can describe the relation between $\theta$ and $v$, we need
to introduce a number of operators. First of all, for any $\vartheta \in \Real$
and any $1 \le p \le \infty$,
we write
\begin{equation}
Q_{\vartheta}: \ell^p(\Wholes; \Real^d) \to \Real
\end{equation}
for the bounded linear operator that acts as
\begin{equation}
Q_{\vartheta} v = \sum_{n \in \Wholes} \langle \psi(n + \vartheta) ,  v_n \rangle,
\end{equation}
in which we have recalled the function $\psi(\xi)$
that spans the kernel of the adjoint operator
$\mathcal{L}^*_0$, normalized by \sref{eq:mr:normalizatonConditionOnPsi}.
The operator $Q_{\vartheta}$ is well-defined as a consequence of the fact
that $\psi(\xi)$
decays at an exponential rate as $\xi \to \pm \infty$.
This decay also allows us to view $Q_{\vartheta}$
as a bounded linear map
\begin{equation}
Q_{\vartheta}: \mathcal{X}_{p, q} \to \ell^q(\Wholes; \Real)
\end{equation}
by writing
\begin{equation}
[Q_{\vartheta} v ]_l = \sum_{n \in \Wholes} \langle \psi(n + \vartheta) ,  v_{nl} \rangle.
\end{equation}
In particular, $Q_{\vartheta}$ acts only on the direction of propagation of the wave.
Associated to $Q_{\vartheta}$, we can introduce the operator
\begin{equation}
P_{\vartheta} = [T_{\vartheta} \phi] Q_{\vartheta},
\end{equation}
which can be seen either as a  bounded linear map
\begin{equation}
P_{\vartheta}: \ell^p(\Real, \Real^d) \to \ell^p(\Real, \Real^d)
\end{equation}
or as a bounded linear map
\begin{equation}
P_{\vartheta}: \mathcal{X}_{p, q} \to \mathcal{X}_{p, q}
\end{equation}
by writing
\begin{equation}
[P_{\vartheta} v]_{nl} = [T_{\vartheta} \phi]_n [Q_{\vartheta} v]_l
   = \phi(n + \vartheta) \sum_{n' \in \Wholes} \langle \psi(n' + \vartheta) ,  v_{n'l} \rangle.
\end{equation}
For convenience, we also introduce the operators
\begin{equation}
Q'_{\vartheta}: \mathcal{X}_{p, q} \to \ell^{q}(\Wholes; \Real)
\end{equation}
which arise in the course of computing time derivatives of $Q_{ct}$
and act as
\begin{equation}
[Q'_{\vartheta} v]_l = \sum_{n \in \Wholes} \langle \psi'(n + \vartheta), v_{nl} \rangle.
\end{equation}
Note that $\psi'(\xi)$ also decays exponentially as $\xi \to \pm \infty$, which can be deduced
from the MFDE satisfied by $\psi$.
Finally, we write
\begin{equation}
P'_{\vartheta}: \mathcal{X}_{p, q} \to \mathcal{X}_{p, q}
\end{equation}
for the associated operators $P'_\vartheta = [T_\vartheta \phi] Q'_\vartheta$ that act as
\begin{equation}
[P'_{\vartheta} v]_{nl} = [T_{\vartheta} \phi]_n [Q'_{\vartheta} v]_l
   = \phi(n + \vartheta) \sum_{n' \in \Wholes} \langle \psi'(n' + \vartheta) ,  v_{n'l} \rangle.
\end{equation}


Notice that \sref{eq:mr:innProduct:alwaysEq} implies that
\begin{equation}
Q_{\vartheta} T_{\vartheta} \phi = 1
\end{equation}
for all $\vartheta \in \Real$.
In particular, we have the identity $P_{\vartheta}^2 = P_{\vartheta}$,
both in $\mathcal{L}\big( \ell^p(\Wholes; \Real^d) \big)$
and $\mathcal{L}\big( \mathcal{X}_{p, q} \big)$.
In fact, the operator $P_\vartheta$ can be seen as the spectral projection onto the neutral mode
for the linearization of \sref{eq:mr:mainLDE} about $T_\vartheta \Phi$ when posed on
the subspace of
$\ell^\infty(\Z^2; \Real^d)$ that contains only sequences that are constant in the transverse direction.
Here we are interested in the case when arbitrary perturbations are considered,
in which case zero is no longer an isolated eigenvalue and the spectral projection does not exist in a classical sense.
Rather the projection $P_\vartheta$ can be thought of as a crude approximation of what the spectral projection onto the origin wants to be.
See \cite{BECKWAYNE2012,ZUMHOW1998} for more examples of extending spectral projections through the essential spectrum.
Let us emphasize here that $P'_\vartheta$ is not a projection.

\subsection{Equations for $v$ and $\theta$}

We are now ready to describe the relation between $v(t)$ and $\theta(t)$.
The goal is to eliminate the neutral mode in the best way that we can that does
not complicate the eventual Fourier analysis.
To this end, we introduce
for any $1 \le p \le \infty$ and $q \in \{2 , \infty \}$
the subspace
\begin{equation}
\label{eq:cds:XperpDef}
  \mathcal{X}^\perp_{p, q}(t) = \{ v \in \mathcal{X}_{p, q} \mid  Q_{ct} v = 0 \in \ell^q(\Wholes; \Real) \}
\end{equation}
and demand $v(t) \in \mathcal{X}^\perp_{p, q}(t)$.
As we shall see below, this condition is sufficient to fix the dynamics of both $\theta$ and $v$
as long as these functions are both sufficiently small.  In particular, the pair
$(v,\theta) \in \mathcal{X}^\perp_{p,q} \times \ell^q(\Wholes; \Real)$
form coordinates for the solution $u$ in a tubular neighbourhood of the wave profile $\Phi$.

Differentiating the Ansatz \sref{eq:cds:ansatzLoc}, we find that
\begin{equation}
\begin{array}{lcl}
\dot{u}_{nl}(t) & = & \big(c + \dot{\theta}_l(t) \big) \Phi'\big(n + ct + \theta_l(t) \big) + \dot{v}_{nl}(t) \\[0.2cm]
& = & f\big( [\tau \Phi]\big(n + ct + \theta_l(t) \big)  \big) + \dot{\theta}_l(t) \phi\big(n + ct +\theta_l(t) \big) + \dot{v}_{nl}(t),
\end{array}
\end{equation}
which we abbreviate as
\begin{equation}
\dot{u}(t)  =
 f\big( T_{ct + \theta(t)} \tau \Phi  \big) + \dot{\theta}(t) T_{ct + \theta(t)}\phi + \dot{v}(t),
\end{equation}
using the convention that the product of a sequence $\theta \in \ell^\infty(\Wholes; \Real)$ with a
planar sequence $v \in \mathcal{X}_{p, \infty}$ is given by $[\theta v]_{nl} = \theta_l v_{nl}$.
Imposing the condition that the Ansatz \sref{eq:cds:ansatzLoc} is a true solution to the LDE \sref{eq:cds:ldeNewCoordsGlb},
we obtain the following evolution equation for $v$:
\begin{equation}
\label{eq:cds:lde:glb:1}
\dot{v}(t) = f\big(\pi^\times T_{ct + \theta(t)} \Phi + \pi^\times v(t) \big) -
  f\big( T_{ct + \theta(t)} \tau \Phi \big) - \dot{\theta}(t) T_{ct + \theta(t) } \phi.
\end{equation}
%

It is desirable to write the term $f(y) - f(x)$ as $f'(x)(y-x) + O((y-x)^2)$.
To this end we analyze the difference $T_{ct}(\pi^\times T_{\theta} - T_\theta \tau)\Phi$ which when added to
$\pi^\times v$ gives us the relevant $(x-y)$ in our Taylor expansion.  Note that $T_\theta$ does not commute with $\pi^\times$ or $\tau$ because $\theta$ is in general $l$-dependent.  This difference is a vector of differences of the wave profile, the first term of which is $\Phi(\xi + \sigma_1 + \theta_{l+\sigma_2}) - \Phi(\xi + \sigma_1 + \theta_l)$, the other terms being similar except with different indices on and signs in front of the $\sigma_j$.  To lowest order this term factors as the product $\Phi'(\xi + \sigma_1)(\theta_{l+\sigma_2}-\theta_l)$.  This is the first time that the anisotropy of the lattice complicates our analysis.  Firstly, the coefficient on $(\theta_{l-\sigma_2}-\theta_l)$ is $\Phi'(\xi-\sigma_1) \not \equiv \Phi'(\xi + \sigma_1)$ and hence the two slowly decaying first differences $(\theta_{l+\sigma_2}-\theta_l)$ and $(\theta_{l-\sigma_2}-\theta_l)$ do not combine to form
the relatively rapidly decaying second difference $\theta_{l+2} + \theta_{l-2}-2\theta_l$ (unless of course $\sigma_1 = 0$).  Secondly, because the coefficients on the first differences do not have identical behaviour in the wave direction, this term cannot be projected away by imposing an orthogonality condition on $v$ that is uniform in $l$.

We now introduce notation that allows us to decompose the vector of differences $T_{ct}(\pi^\times T_{\theta} - T_\theta \tau)\Phi$
and more generally the right hand side of $\eqref{eq:cds:lde:glb:1}$ into linear terms that generate
the evolution semigroup, nonlinear terms whose estimation requires care, and nonlinear terms which are estimated relatively easily.

We use the notation $\theta^\diamond \in \ell^\infty(\Wholes; \Real^5)$
to refer to the sequence
\begin{equation}
\theta^\diamond_l =
\big( \theta_{l + \sigma_2} - \theta_l, \theta_{l - \sigma_1} - \theta_l,
\theta_{l - \sigma_2} - \theta_l, \theta_{l + \sigma_1} - \theta_l,0 \big) \in \Real^5.
\end{equation}
In all of our applications the 5th coordinate is zero, but we include the 5th coordinate
so that dimensions match.
A short computation shows that
\begin{equation}
\begin{array}{lcl}
[\pi^\times T_{ct + \theta}\Phi -  T_{ct + \theta} \tau \Phi]_{nl} & = &
[\mathcal{T}\big(  \theta^\diamond_l \big)  \Phi( \cdot + ct + \theta_l ) ]_n,
\end{array}
\end{equation}
in which we have introduced the symbol
\begin{equation}
\mathcal{T}( \alpha) =
\big(
 T_{   \sigma_1 + \alpha_1  } -  T_{   \sigma_1},
 T_{   \sigma_2 + \alpha_2  } -  T_{   \sigma_2},
 T_{ - \sigma_1 + \alpha_3  } -  T_{ - \sigma_1},
 T_{ - \sigma_2 + \alpha_4  } -  T_{ - \sigma_2},
 0
\big),
\end{equation}
for any $\alpha = (\alpha_1, \ldots, \alpha_5) \in \Real^5$.
For any $n \in \Wholes$ and $\vartheta \in \Real$, we now introduce
the linear operator 
$M(n, \vartheta): \Real^5 \to (\Real^d)^5 $ that acts as
\begin{equation}
\begin{array}{lcl}
M( n, \vartheta) \alpha & = &
\big[ \big( \alpha_1 T_{  \sigma_1},
      \alpha_2 T_{  \sigma_2},
      \alpha_3 T_{ -\sigma_1},
      \alpha_4 T_{ -\sigma_2},
      0
\big) \phi(\cdot + \vartheta) \big]_n \\[0.2cm]
& = &
\big( \alpha_1 \phi(\vartheta + n + \sigma_1 ),
      \alpha_2 \phi(\vartheta + n + \sigma_2 ),
      \alpha_3 \phi(\vartheta + n - \sigma_1),
      \alpha_4 \phi(\vartheta + n - \sigma_2),
      0
\big).
\end{array}
\end{equation}
Upon introducing the nonlinear function $\mathcal{M}: \Wholes \times \Real \times \Real^5 \to (\Real^d)^5$  
given by
\begin{equation}
\mathcal{M}(n, \vartheta; \alpha) = 
\big[ \mathcal{T}\big(\alpha \big)  \Phi(\cdot + \vartheta) \big]_n  - M(n, \vartheta) \alpha,
\end{equation}
the fact that $\Phi''$ is bounded 
yields the estimate
\begin{equation}
\mathcal{M}(n,  \vartheta ; \alpha \big) = O(\abs{\alpha}^2), \qquad \alpha \to 0,
\end{equation}
which holds uniformly for $n \in \Wholes$ and $\vartheta \in \Real$.
We have hence obtained
\begin{equation}
\label{eq:cds:expOfCommut}
[\pi^\times T_{ct + \theta}\Phi - T_{ct + \theta} \tau \Phi]_{nl}  =
M(n, ct +  \theta_l)  \theta^\diamond_l + \mathcal{M}(n, ct +  \theta_l; \theta^\diamond_l ),
\end{equation}
which we write in global form as
\begin{equation}
\pi^\times T_{ct + \theta}\Phi - T_{ct + \theta} \tau \Phi  =
M(ct +  \theta)  \theta^\diamond + \mathcal{M}(ct +  \theta; \theta^\diamond ).
\end{equation}

Thus \eqref{eq:cds:lde:glb:1} may be rewritten informally as
\begin{equation}
\dot{v} = L[\pi^\times v + M\theta^\diamond + \mathcal{M}] + \mathcal{N}(\pi^\times v + M\theta^\diamond + \mathcal{M})
\label{eq:cds:lde:glb:2}
\end{equation}
where we have introduced some new notation.  The linear operator $L$ is given in local form by
\begin{equation}
L(n, \vartheta) v =   D f\big( [\tau \Phi ](n + \vartheta) \big) v,
\end{equation}
while the nonlinear $\mathcal{N}$, given by
\begin{equation}
\mathcal{N}\big(n , \vartheta ;  v \big) =f( [ \tau \Phi](n + \vartheta)  + v)
- f\big( [\tau \Phi](n + \vartheta) \big)  - D f\big( [\tau \Phi](n + \vartheta) \big)v,
\end{equation}
enjoys the estimate
\begin{equation}
\mathcal{N}(n , \vartheta;  v \big) = O \big( \abs{v }^2 \big), \qquad v \to 0,
\end{equation}
which is uniform for $n \in \Wholes$ and $\vartheta \in \Real$ and follows from the $C^2$-smoothness of $f$ together with the uniform boundedness of $\Phi$.  We now begin the process of separating the terms which are easy to estimate from those that require more care by writing

%
\begin{equation}
\label{eq:cds:sysForV:glb:1}
\begin{array}{lcl}
\dot{v}(t) & = &  L(ct +  \theta(t) ) [ \pi^\times v(t) + M\big( ct + \theta(t) \big)  \theta^\diamond(t) \big]
 - \dot{\theta}(t) T_{ct + \theta(t) } \phi \\[0.2cm]
& & \qquad  + \mathcal{R}_{1, ct}\big( v(t), \theta(t) \big),
\end{array}
\end{equation}
in which the nonlinearity
$\mathcal{R}_{1 ,ct}:  \ell^\infty(\Wholes^2; \Real^d)  \times \ell^\infty(\Wholes; \Real)  \to \ell^\infty(\Wholes^2; \Real^d)$
acts as
\begin{equation}
\begin{array}{lcl}
[\mathcal{R}_{1, ct}\big(  v, \theta \big)]_{nl}
& = & L(n, ct + \theta_l) \mathcal{M}(n, ct + \theta_l; \theta^\diamond_l) \\[0.2cm]
& & \qquad + \mathcal{N}\big(n, ct + \theta_l; \pi^\times_{nl} v + M(n, ct + \theta_l) \theta^\diamond_l
   + \mathcal{M}(n,  ct + \theta_l; \theta^\diamond_l)\big).
\end{array}
\end{equation}

%

We now set out to derive an expression for $\dot{\theta}_l$.
Differentiating the condition $Q_{ct} v(t) = 0$ shows that
\begin{equation}
Q_{ct} \dot{v}(t) + c Q'_{ct} v(t) = 0 \in \ell^\infty(\Wholes; \Real).
\end{equation}
Substitution of \sref{eq:cds:sysForV:glb:1} yields
the system
\begin{equation}
\label{eq:cds:eqThetaDot:glb:1}
\begin{array}{lcl}
\dot{\theta}(t) Q_{ct} T_{ct + \theta(t) } \phi & = &
  Q_{ct} L(ct +  \theta(t) ) [ \pi^\times v(t) + M\big( ct + \theta(t) \big) \theta^\diamond \big]
 +  c Q'_{ct} v  \\[0.2cm]
& & \qquad  +  Q_{ct} \mathcal{R}_{1, ct}\big( v(t) , \theta(t) \big).
\end{array}
\end{equation}
Observe that for each $l$ and at each time $t$, the quantity
$Q_{ct} T_{ct + \theta_l(t)}\phi$ is merely a number that can be divided away to isolate $\dot{\theta}_l$.
In view of the identity $Q_{ct} T_{ct} \phi = 1$,
the uniform bound on $\Phi''' = \phi''$ implied by H$\Phi$, Hf, \eqref{eq:mr:waveprofileMFDE} and \eqref{eq:mr:wavelim}
and the expansion $\frac{1}{1+\eps} = 1-\eps + O(\eps^2)$,
we can pick a small $\delta_\theta > 0$
and introduce functions $\mathcal{V}_{\vartheta}: ( - \delta_{\theta}, \delta_{\theta} ) \to \Real$
that are fixed by the relation
\begin{equation}
\label{eq:cds:defMathcalV}
[Q_{ct} T_{ct + \theta_l}\phi]^{-1} - 1 = - \theta_l Q_{ct} T_{ct} \phi' +    \mathcal{V}_{ct} ( \theta_l).
\end{equation}
These functions satisfy the bounds
\begin{equation}
\mathcal{V}_\vartheta( \theta_l) = O( \abs{\theta_l}^2 ), \qquad \theta_l \to 0,
\end{equation}
uniformly for $\vartheta \in \Real$ and can be extended to map $\ell^\infty(\Wholes; (- \delta_{\theta}, \delta_{\theta}) )$
into $\ell^\infty(\Wholes; \Real)$.
Using these functions to rewrite \sref{eq:cds:eqThetaDot:glb:1}, we obtain
\begin{equation}
\label{eq:cds:eqThetaDot:glb:2}
\begin{array}{lcl}
\dot{\theta}(t)  & = &
   Q_{ct} L( ct + \theta(t) ) [ \pi^\times v(t) + M\big(ct + \theta(t) \big)  \theta^\diamond(t) \big]
 +  c Q'_{ct} v(t)  \\[0.2cm]
& &  - \theta Q_{ct} T_{ct} \phi' Q_{ct} L(ct + \theta(t) ) [ \pi^\times v(t) + M\big( ct+ \theta(t) \big) \theta^\diamond(t) \big]
            - c \theta Q_{ct} T_{ct} \phi'  Q'_{ct} v(t)  \\[0.2cm]
& & \qquad  + \mathcal{S}_{1, ct}\big(  v(t), \theta(t) \big),
\end{array}
\end{equation}
in which the nonlinearity
$\mathcal{S}_{1, ct}:  \ell^\infty(\Wholes^2; \Real^d)
                      \times \ell^\infty\big(\Wholes; (-\delta_\theta, \delta_\theta) \big)
                      \to    \ell^\infty(\Wholes; \Real)$
is defined by
\begin{equation}
\label{eq:cds:defnCalS1CT}
\begin{array}{lcl}
\mathcal{S}_{1, ct} (  v, \theta ) & = &
 \big(1  - \theta Q_{ct} T_{ct} \phi' + \mathcal{V}_{ct}(\theta) \big)Q_{ct} \mathcal{R}_{1,ct}\big( v, \theta \big)
    + c \mathcal{V}_{ct}(\theta)Q'_{ct} v
   \\[0.2cm]
& & + \mathcal{V}_{ct}(\theta) Q_{ct} L( ct+ \theta ) \big[ \pi^\times v + M\big( ct+ \theta \big) \theta^\diamond \big]. \\[0.2cm]
\end{array}
\end{equation}
%
%
Notice that we are explicitly excluding terms of order $O(\theta v)$ and $O(\theta \theta^\diamond)$
from the nonlinearity $\mathcal{S}_{1, ct}$ because these terms require special treatment in our nonlinear stability analysis;
see \S\ref{sec:nl}.

Plugging \sref{eq:cds:eqThetaDot:glb:2} back into \sref{eq:cds:sysForV:glb:1}, we arrive at
\begin{equation}
\begin{array}{lcl}
\dot{v}(t) & = &  [I - P_{ct}] L(ct + \theta(t) ) [ \pi^\times v(t) + M\big( ct+ \theta(t) \big)  \theta^\diamond(t) \big]
-c P'_{ct} v(t) \\[0.2cm]
& & \qquad  + \mathcal{R}_{2,ct}\big(  v(t), \theta(t) \big),
\end{array}
\end{equation}
in which the nonlinearity
$\mathcal{R}_{2, ct}:   \ell^\infty(\Wholes^2; \Real^d)
                      \times  \ell^\infty(\Wholes; (-\delta_\theta, \delta_\theta))
                      \to \ell^\infty(\Wholes^2; \Real^d)$
acts as
\begin{equation}
\label{eq:cds:defCalR2}
\begin{array}{lcl}
\mathcal{R}_{2, ct}\big(v, \theta \big)
& = & \mathcal{R}_{1,ct}\big( v, \theta \big) \\[0.2cm]
& & \qquad + \theta T_{ct + \theta(t)}\phi Q_{ct} T_{ct} \phi' Q_{ct} L(ct + \theta(t) )
  \big[ \pi^\times v(t) + M\big( ct + \theta(t) \big) \theta^\diamond(t) \big]  \\[0.2cm]
& &  \qquad +  c \theta T_{ct + \theta(t)}\phi Q_{ct} T_{ct} \phi'  Q'_{ct} v(t)                \\[0.2cm]
& & \qquad - T_{ct + \theta(t)}\phi \mathcal{S}_{1, ct}\big( v(t), \theta(t) \big)  \\[0.2cm]
& & \qquad - [T_{ct + \theta(t)} - T_{ct} ] \phi Q_{ct} L( ct + \theta(t))
    \big[\pi^\times v + M(ct + \theta(t)) \theta^\diamond(t)\big] \\[0.2cm]
& & \qquad - c[T_{ct + \theta(t)} - T_{ct}] \phi Q'_{ct} v(t).
\end{array}
\end{equation}
Notice that we are absorbing the $O( \theta \theta^\diamond )$ and $O( \theta v)$ terms into
the nonlinearity $\mathcal{R}_{2, ct}$,
in contrast to our expression \sref{eq:cds:eqThetaDot:glb:2} for $\dot{\theta}$.

The next step is to expand the linear operators $L(\vartheta)$ and $M(\vartheta)$ around $\vartheta = ct$.
In particular, we write
\begin{equation}
L_{ct} = L(ct), \qquad M_{ct} = M(ct)
\end{equation}
and obtain the system
\begin{equation}
\begin{array}{lcl}
\dot{v}(t) & = &
[I - P_{ct}] L_{ ct} [ \pi^\times v(t) + M_{ct}  \theta^\diamond(t) \big]
-c P'_{ct} v(t) \\[0.2cm]
& & \qquad  + \mathcal{R}_{3, ct}\big( v(t), \theta(t) \big), \\[0.2cm]
\dot{\theta}(t) & = &
Q_{ct} L_{ct} [ \pi^\times v(t) + M_{ ct} \theta^\diamond(t) \big]
 +  c Q'_{ct} v(t)  \\[0.2cm]
& & \qquad + \theta H_{3,ct}[ v(t), \theta^\diamond(t) ]  \\[0.2cm]
& & \qquad  + \mathcal{S}_{3,ct}\big( v(t), \theta(t) \big). \\[0.2cm]
\end{array}
\end{equation}
Here we have introduced the linear operators
$H_{3, ct}: \ell^\infty(\Wholes^2 , \Real^d ) \times \ell^\infty(\Wholes; \Real^5) \to \ell^\infty(\Wholes; \Real)$
that act as
\begin{equation}
\begin{array}{lcl}
H_{3, ct} [ v, \theta^\diamond ] & = &
  Q_{ct} [\partial_\vartheta L(\vartheta)]_{\vartheta = ct} [ \pi^\times v + M_{ct}  \theta^\diamond \big] \\[0.2cm]
& & +   Q_{ct} L_{ ct} [\partial_\vartheta M(\vartheta)]_{\vartheta = ct} \theta^\diamond \\[0.2cm]
& &  -  Q_{ct} T_{ct} \phi' Q_{ct} L_{ct} [ \pi^\times v + M_{ct} \theta^\diamond \big]
            - c  Q_{ct} T_{ct} \phi'  Q'_{ct} v,  \\[0.2cm]
\end{array}
\end{equation}
together with
the nonlinearities
$\mathcal{R}_{3, ct}: \ell^\infty(\Wholes^2; \Real^d)
                      \times \ell^\infty(\Wholes; (-\delta_\theta, \delta_\theta) )
                      \to \ell^\infty(\Wholes^2; \Real^d)$
given by
\begin{equation}
\begin{array}{lcl}
\mathcal{R}_{3,ct}\big( v, \theta)
& = & \mathcal{R}_{2, ct}\big( v , \theta\big) \\[0.2cm]
& & \qquad +  [I - P_{ct}][L(ct + \theta_l) - L(ct)][\pi^\times v + M(ct + \theta)  \theta^\diamond]  \\[0.2cm]
& & \qquad +  [I - P_{ct}] L_{ct} [ M(ct + \theta) - M(ct) ] \theta^\diamond  \\[0.2cm]
\end{array}
\end{equation}
and the nonlinearities
$\mathcal{S}_{3, ct}: \ell^\infty(\Wholes^2; \Real^d)
                      \times \ell^\infty(\Wholes; (-\delta_\theta, \delta_\theta) )
                      \to \ell^\infty(\Wholes; \Real)$
defined by
\begin{equation}
\begin{array}{lcl}
\mathcal{S}_{3, ct} \big( v, \theta \big) & = &
  \mathcal{S}_{1, ct} (v, \theta ) \\[0.2cm]
& & \qquad + Q_{ct}\big[L(ct+ \theta) - L(ct) - \theta [\partial_\vartheta L(\vartheta)]_{\vartheta = ct} \big]
  [\pi^\times v + M(ct + \theta) \theta^\diamond ]   \\[0.2cm]
& & \qquad + Q_{ct}L_{ct} \big[ M(ct + \theta) - M(ct) - \theta [\partial_\vartheta M( \vartheta)]_{\vartheta = ct} \big]  \theta^\diamond  \\[0.2cm]
& & \qquad - \theta Q_{ct} T_{ct} \phi' Q_{ct} [L(ct + \theta) - L(ct)][\pi^\times v + M(ct + \theta) \theta^\diamond] \\[0.2cm]
& & \qquad - \theta Q_{ct} T_{ct} \phi' Q_{ct} L_{ct} [ M(ct + \theta) - M(ct)  ] \theta^\diamond.  \\[0.2cm]
\end{array}
\end{equation}

Our final step consists of rewriting the linear operator $H_{3,ct}$
in such a way that the nonlinear analysis in \S\ref{sec:nl} is simplified.
In particular, we reuse the symbol $\tau$ to write
\begin{equation}
[\tau v]_{nl} = \big( v_{n + \sigma_1, l} , v_{n + \sigma_2, l} , v_{n - \sigma_1, l}, v_{n - \sigma_2, l} , v_{nl} \big) \in \Real^5
\end{equation}
for any $v \in \ell^\infty(\Wholes^2; \Real^d)$. Notice that unlike $\pi^\times$, the operator
$\tau$ does not act on the transverse direction $\ell \in \Wholes$. Indeed,
for any $v \in \ell^\infty(\Wholes^2; \Real^d)$ we have the identity
\begin{equation}
\pi^\times v = \tau v + \tau \cdot v^\diamond,
\end{equation}
in which
\begin{equation}
[v^\diamond]_{nl} = \big( v_{n, l + \sigma_2} - v_{nl},
                          v_{n, l - \sigma_1} - v_{nl},
                          v_{n, l - \sigma_2} - v_{nl},
                          v_{n, l + \sigma_1} - v_{nl},
                          0
                    \big),
\end{equation}
while the dot product is defined in the obvious manner.
A short computation shows that
\begin{equation}
\begin{array}{lcl}
Q_{ct} L_{ct} \tau v & = & \sum_{n \in \Wholes} \langle \psi(n + ct) , Df( [\tau \Phi](n + ct) ) [\tau v]_{nl} \rangle \\[0.2cm]
& = & \sum_{n \in \Wholes} \langle Df( [\tau \Phi](n + ct) ) [\tau^* \psi](n + ct), v_{nl} \rangle \\[0.2cm]
& = & \sum_{n \in \Wholes} \langle -c \psi'(n + ct), v_{nl} \rangle \\[0.2cm]
& = & -c Q'_{ct} v.
\end{array}
\end{equation}
In particular, we see that
\begin{equation}
\begin{array}{lcl}
H_{3, ct} [ v, 0 ] & = &
  Q_{ct} [\partial_\vartheta L(\vartheta)]_{\vartheta = ct}  \pi^\times v
 -  Q_{ct} T_{ct} \phi' Q_{ct} L_{ct}  \tau \cdot v^\diamond.
\end{array}
\end{equation}

Putting everything together, we finally arrive at the system
\begin{equation}
\label{eq:cds:sys:glb:4}
\begin{array}{lcl}
\dot{v}(t) & = &
[I - P_{ct}] L_{ ct} [ \pi^\times v(t) + M_{ct}  \theta^\diamond(t) \big]
-c P'_{ct} v(t) \\[0.2cm]
& & \qquad  + \mathcal{R}_{4, ct}\big( v(t), \theta(t) \big), \\[0.2cm]
\dot{\theta}(t) & = &
Q_{ct} L_{ct} [ \pi^\times v(t) + M_{ ct} \theta^\diamond(t) \big]
 +  c Q'_{ct} v(t)  \\[0.2cm]
& & \qquad + \theta H_{4,ct}[ \tau v(t), \theta^\diamond(t) ]
+ \mathcal{S}_{4,ct}\big( v(t), \theta(t) \big). \\[0.2cm]
\end{array}
\end{equation}
Here we have introduced the  linear operators
$H_{4, ct}: \ell^\infty(\Wholes^2 , (\Real^d)^5 ) \times \ell^\infty(\Wholes; \Real^5) \to \ell^\infty(\Wholes; \Real)$
that act as
\begin{equation}
\begin{array}{lcl}
H_{4, ct} [ \tau v, \theta^\diamond ] & = &
  Q_{ct} [\partial_\vartheta L(\vartheta)]_{\vartheta = ct} [ \tau v + M_{ct}  \theta^\diamond \big] \\[0.2cm]
& & \qquad +   Q_{ct} L_{ ct} [\partial_\vartheta M(\vartheta)]_{\vartheta = ct} \theta^\diamond
 -  Q_{ct} T_{ct} \phi' Q_{ct} L_{ct} M_{ct} \theta^\diamond \\[0.2cm]
\end{array}
\end{equation}
and are bounded uniformly for $t \in \Real$.
In addition, after picking a smooth cut-off function $\chi^c: [0, \infty) \to [0, 1]$ that has
\begin{equation}
\chi^c( [0, 1] ) = 1, \qquad \chi^c([ 2, \infty)) = 0,
\end{equation}
we have introduced
the nonlinearities
$\mathcal{R}_{4, ct}: \ell^\infty(\Wholes^2; \Real^d)
                      \times \ell^\infty(\Wholes; \Real )
                      \to \ell^\infty(\Wholes^2; \Real^d)$
given by
\begin{equation}
\begin{array}{lcl}
\mathcal{R}_{4,ct}\big( v, \theta)
& = & \chi^c\big( \norm{v}_{\ell^\infty(\Wholes^2; \Real^d)} + 2 \delta_{\theta}^{-1} \norm{\theta}_{\ell^\infty(\Wholes; \Real)} \big)  \mathcal{R}_{3, ct}\big( v , \theta\big) \\[0.2cm]
\end{array}
\end{equation}
and the nonlinearities
$\mathcal{S}_{4, ct}: \ell^\infty(\Wholes^2; \Real^d)
                      \times \ell^\infty(\Wholes; \Real )
                      \to \ell^\infty(\Wholes; \Real)$
defined by
\begin{equation}
\begin{array}{lcl}
\mathcal{S}_{4, ct} \big( v, \theta \big) & = &
  \chi^c\big( \norm{v}_{\ell^\infty(\Wholes^2; \Real^d)} + 2 \delta_{\theta}^{-1} \norm{\theta}_{\ell^\infty(\Wholes; \Real)} \big)
    \mathcal{S}_{3, ct} (v, \theta ) \\[0.2cm]
& & + \theta Q_{ct} [\partial_\vartheta L(\vartheta)]_{\vartheta = ct}[  \tau \cdot v^\diamond ]
 - \theta Q_{ct} T_{ct} \phi' Q_{ct} L_{ ct} \tau \cdot v^\diamond.
\end{array}
\end{equation}
We emphasize that the introduction of the cut-offs serves as a notational convenience only. In particular,
they allow us to formulate the following global bounds on the nonlinearities.

\begin{prop}
\label{prp:cds:bnds:Nonlins}
Suppose that (Hf), (H$\Phi$),
$\mathit{(HS)}_{\mathrm{ess}}$ and $\mathit{(HS1)}_p$-$\mathit{(HS3)}_p$
are satisfied and pick $1 \le p \le \infty$.
There exists a constant $C > 0$ such that for any $t \in \Real$, $v \in \mathcal{X}_{p, 2}$
and $\theta \in \ell^2(\Wholes; \Real )$ we have the bounds
\begin{equation}
\label{eq:cds:bnds:R}
\begin{array}{lcl}
\norm{ \mathcal{R}_{4, ct}\big( v, \theta \big)  }_{\mathcal{X}_{p, 1} } & \le & C
\big( \norm{v}_{\mathcal{X}_{p, 2}} + \norm{\theta^\diamond}_{\ell^2(\Wholes; \Real^5)} + \norm{\theta}_{\ell^2(\Wholes; \Real) } \big)
   \big( \norm{v}_{\mathcal{X}_{p, 2} } + \norm{\theta^\diamond}_{\ell^2(\Wholes; \Real^5)} \big),
\\[0.2cm]
\norm{ \mathcal{R}_{4, ct}\big( v, \theta \big)  }_{\mathcal{X}_{p, 2} } & \le & C
\big( \norm{v}_{\mathcal{X}_{p, 2}} + \norm{\theta^\diamond}_{\ell^2(\Wholes; \Real^5)} + \norm{\theta}_{\ell^2(\Wholes; \Real) } \big)
     \big( \norm{v}_{\mathcal{X}_{p, \infty} } + \norm{\theta^\diamond}_{\ell^\infty(\Wholes; \Real^5) } \big),
\\[0.2cm]
\end{array}
\end{equation}
together with
\begin{equation}
\label{eq:cds:bnds:S}
\begin{array}{lcl}
\norm{ \mathcal{S}_{4, ct}\big( v, \theta \big)  }_{\ell^1(\Wholes; \Real) } & \le & C
 \norm{\theta}_{\ell^2(\Wholes; \Real) } \norm{v^\diamond}_{\mathcal{X}^5_{p, 2}}
+  C \big( \norm{v}_{ \mathcal{X}_{p, \infty} } + \norm{\theta^\diamond}_{\ell^\infty(\Wholes; \Real^5) } \big)
     \norm{\theta}^2_{\ell^2(\Wholes; \Real) } \\[0.2cm]
& & \qquad    + C  \big( \norm{v}_{\mathcal{X}_{p, 2}} + \norm{ \theta^\diamond}_{\ell^2(\Wholes; \Real^5) } \big)
     \big( \norm{v}_{\mathcal{X}_{p, 2}} + \norm{ \theta^\diamond}_{\ell^2(\Wholes; \Real^5) } \big), \\[0.2cm]
\norm{ \mathcal{S}_{4, ct}\big( v, \theta \big)  }_{\ell^2(\Wholes; \Real) } & \le & C
 \norm{\theta}_{\ell^2(\Wholes; \Real) } \norm{v^\diamond}_{\mathcal{X}^5_{p, \infty}}
 +  C \big( \norm{v}_{ \mathcal{X}_{p, \infty} } + \norm{\theta^\diamond}_{\ell^\infty(\Wholes; \Real^5) } \big)
  \norm{\theta}_{\ell^2(\Wholes; \Real) } \norm{\theta}_{\ell^\infty(\Wholes; \Real) } \\[0.2cm]
& & \qquad  + C  \big( \norm{v}_{\mathcal{X}_{p, 2}} + \norm{ \theta^\diamond}_{\ell^2(\Wholes; \Real^5) } \big)
     \big( \norm{v}_{\mathcal{X}_{p, \infty}} + \norm{ \theta^\diamond}_{\ell^\infty(\Wholes; \Real^5) } \big). \\[0.2cm]
\end{array}
\end{equation}
\end{prop}
\begin{proof}
Consider two arbitrary sequences $a,b \in \ell^2(\Wholes; \Real)$
and write $ab \in \ell^1(\Wholes; \Real)$ for the product sequence $[ab]_l = a_l b_l$.
Using Cauchy-Schwartz and an elementary estimate it is easy to see that
\begin{equation}
\label{eq:cds:prp:bnd:elm1}
\norm{ab}_{\ell^1(\Wholes; \Real) } \le \norm{a}_{\ell^2(\Wholes; \Real) } \norm{b}_{\ell^2(\Wholes; \Real) },
\qquad \norm{ab}_{\ell^2(\Wholes; \Real) } \le \norm{a}_{\ell^2(\Wholes; \Real)} \norm{b}_{\ell^\infty(\Wholes; \Real)}.
\end{equation}
Another elementary estimate shows that
\begin{equation}
\label{eq:cds:prp:bnd:elm2}
\norm{a^2 b}_{\ell^1(\Wholes; \Real)} \le \norm{a}^2_{\ell^2(\Wholes; \Real) } \norm{b}_{\ell^\infty(\Wholes; \Real)},
\qquad \norm{a^2 b}_{\ell^2(\Z,\R)} \le \norm{a}_{\ell^2(\Wholes; \Real)} \norm{a}_{\ell^\infty(\Wholes; \Real) } \norm{b}_{\ell^\infty(\Wholes; \Real)}.
\end{equation}
The exponential decay of $\phi(\xi)$ and $\phi'(\xi)$ as $\xi \to \pm \infty$
imply the inclusions $T_{ct} \phi \in \ell^p(\Wholes; \Real^d)$
and $T_{ct} \phi' \in \ell^p(\Wholes; \Real^d)$.
In addition, the short-range interaction operators $\pi^\times$ and $\tau$ are bounded as
maps from $\mathcal{X}_{p, 2}$ into $\mathcal{X}^5_{p, 2}$.
Finally, all long-range interactions featuring in $\mathcal{R}_{4, ct}$ and $\mathcal{S}_{4, ct}$
come from the linear operators $Q_{ct}$ and $Q'_{ct}$. These operators are bounded as maps from $\mathcal{X}_{p, 2}$ into $\ell^2(\Wholes; \Real)$.
In particular, the bounds \sref{eq:cds:bnds:R}-\sref{eq:cds:bnds:S}
follow directly from the estimates \sref{eq:cds:prp:bnd:elm1}-\sref{eq:cds:prp:bnd:elm2}
upon inspection of the structure of $\mathcal{R}_{4, ct}$ and $\mathcal{S}_{4, ct}$.
\end{proof}

We conclude this section by summarizing our discussion and explicitly
stating the relation between our modified  system \sref{eq:cds:sys:glb:4}
and the original LDE \sref{eq:mr:mainLDE}.

\begin{prop}
Suppose that (Hf), (H$\Phi$), $\mathit{(HS)}_{\mathrm{ess}}$ and $\mathit{(HS1)}_p$-$\mathit{(HS3)}_p$ are satisfied.
Pick any $t_0 \in \Real$
and consider any pair of functions $v: [t_0, \infty) \to \ell^\infty(\Wholes^2; \Real^d)$
and $\theta: [t_0, \infty) \to \ell^\infty(\Wholes; \Real)$ for which
$Q_{c t_0} v(t_0) = 0$ and for which the inequalities
\begin{equation}
\begin{array}{lcl}
\norm{ \theta(t) }_{\ell^\infty(\Wholes; \Real) } & \le &   \frac{1}{4} \delta_\theta,  \\[0.2cm]
\norm{ v(t) }_{\ell^\infty(\Wholes^2; \Real^d) } & \le &    \frac{1}{2}   \\[0.2cm]
\end{array}
\end{equation}
hold for all $t \ge t_0$. Then if the pair $(v,\theta)$
satisfies \sref{eq:cds:sys:glb:4}, we have $Q_{ct} v(t) = 0$ for all $t \ge t_0$ and the function
\begin{equation}
u_{ij}(t) = \Phi( i \sigma_1 + j \sigma_2 + ct + \theta_{ i \sigma_2 - j \sigma_1 }(t) \big) + v_{ i \sigma_1 + j \sigma_2, i \sigma_2 - j \sigma_1}(t)
\end{equation}
satisfies the LDE \sref{eq:mr:mainLDE} for all $t \ge t_0$.
\end{prop}
\begin{proof}
Exploiting the restriction
that $\abs{\theta_l(t)} < \delta_\theta$ for all $l \in \Wholes$ and $t \ge t_0$,
we can compute
\begin{equation}
\begin{array}{lcl}
\frac{d}{dt} [Q_{ct} v(t)] & = & c Q'_{ct} v(t) + Q_{ct} \dot{v}(t) \\[0.2cm]
 &  = & cQ'_{ct}  v(t) -c Q_{ct} P'_{ct} v(t) +  Q_{ct} \mathcal{R}_{4, ct}\big(v(t), \theta(t) \big) \\[0.2cm]
& = & Q_{ct} \mathcal{R}_{4, ct}\big(v(t), \theta(t) \big) \\[0.2cm]
& = & Q_{ct} \mathcal{R}_{2, ct}\big(v(t), \theta(t) \big), \\[0.2cm]
\end{array}
\end{equation}
in which we have used $Q_{ct} P'_{ct} = Q'_{ct}$.
In view of \sref{eq:cds:defMathcalV}, we may reformulate
\sref{eq:cds:defnCalS1CT} as
\begin{equation}
\begin{array}{lcl}
\mathcal{S}_{1, ct}(v, \theta)
& = & [Q_{ct} T_{ct + \theta} \phi]^{-1} Q_{ct} \mathcal{R}_{1, ct}(v, \theta)
\\[0.2cm]
& & \qquad +c \mathcal{V}_{ct}(\theta) Q'_{ct} v
+ \mathcal{V}_{ct}(\theta) Q_{ct} L(ct + \theta)
[ \pi^\times v + M(ct + \theta) \theta^\diamond].
\end{array}
\end{equation}
Plugging this directly into the definition \sref{eq:cds:defCalR2}
and using the normalization $Q_{ct} T_{ct} \phi = 1$,
we obtain the factorization
\begin{equation}
\begin{array}{lcl}
Q_{ct} \mathcal{R}_{2, ct}\big(v, \theta \big)
& = &
\Big[
1 - Q_{ct} T_{ct + \theta} \phi + Q_{ct} T_{ct + \theta} \phi \mathcal{V}_{ct}(\theta)
+ \theta Q_{ct} T_{ct + \theta} \phi Q_{ct} T_{ct} \phi'
\Big]
\\[0.2cm]
& & \qquad \times
 \Big[
 Q_{ct}L(ct + \theta)\big[   \pi^\times v + M(ct + \theta) \theta^\diamond  \big]
+
c Q'_{ct} v
\Big].
\end{array}
\end{equation}
The first factor disappears on account of \sref{eq:cds:defMathcalV},
showing that $Q_{ct} \mathcal{R}_{2, ct}\big(v(t), \theta(t) \big) = 0$.
Finally, the fact that the function $u(t)$ solves \sref{eq:mr:mainLDE} follows directly
from the constructions in this section.
\end{proof}

%

%

%

%


%

%

%

%

\section{Green's functions}
\label{sec:green}

In this section we focus on the linear part
of the system \sref{eq:cds:sys:glb:4}, which is given by
\begin{equation}
\label{eq:green:linsys}
\begin{array}{lcl}
\dot{v}(t) & = &
[I - P_{ct}] L_{ ct} [ \pi^\times v(t) + M_{ct} \theta^\diamond(t) \big]
-c P'_{ct} v(t) \\[0.2cm]
\dot{\theta}(t) & = &
Q_{ct} L_{ct} [ \pi^\times v(t) + M_{ ct} \theta^\diamond(t) \big]
 +  c Q'_{ct} v(t).  \\[0.2cm]
\end{array}
\end{equation}
Viewing this system as an ODE posed on $\ell^\infty(\Wholes^2; \Real^d) \times \ell^\infty(\Wholes; \Real)$,
it is clear that solutions are well-defined in forward and backward time since the linear operator
on the right hand side is bounded uniformly for $t \in \Real$. In particular,
we can define a Green's function
\begin{equation}
  \mathcal{G}(t, t_0) = \left(\begin{array}{cc}
   \mathcal{G}_{v v}(t, t_0) & \mathcal{G}_{v\theta}(t, t_0) \\
  \mathcal{G}_{\theta v}(t, t_0) & \mathcal{G}_{\theta \theta}(t, t_0) \\
                \end{array}\right)  \in \mathcal{L}\big( \ell^\infty(\Wholes^2; \Real^d) \times \ell^\infty(\Wholes; \Real) \big)
\end{equation}
by writing
\begin{equation}
\left(\begin{array}{l}v(t) \\ \theta(t)\end{array}\right)
 = \mathcal{G}(t, t_0) \left(\begin{array}{l}v_{t_0} \\ \theta_{t_0}\end{array}\right)
\end{equation}
for the unique solution to \sref{eq:green:linsys} that has
the initial condition $v(t_0) = v_{t_0} \in \ell^\infty(\Wholes^2; \Real^d)$ and
$\theta(t_0) = \theta_{t_0} \in \ell^\infty(\Wholes; \Real)$.

Fixing $1 \le p \le \infty$, our goal is to
restrict attention to initial conditions in the smaller spaces $\mathcal{X}^\perp_{p, q}(t_0) \times \ell^q(\Wholes; \Real)$.
In particular, we obtain bounds for the Green's function
viewed as a mapping
\begin{equation}
\mathcal{G}(t, t_0)
  : \mathcal{X}^\perp_{p, q_1}(t_0) \times \ell^{q_1}(\Wholes; \Real)
                \to \mathcal{X}^\perp_{p, q_2}(t) \times \ell^{q_2}(\Wholes; \Real)
\end{equation}
with $q_1 \in \{1, 2\}$ and $q_2 \in \{2, \infty\}$.

%
%
Our first results show that any solution $(v, \theta)$ to \sref{eq:green:linsys}
can in fact be seen as the coordinates of a solution $\overline{w}$ to the system
\begin{equation}
\label{eq:green:linsys:normal}
\frac{d}{dt} \overline{w}(t)  =  L_{ct} \pi^\times \overline{w}(t)
\end{equation}
posed on $\ell^\infty(\Wholes^2; \Real^d)$. In particular,
there is a close connection between the
Green's function $\overline{\mathcal{G}}(t, t_0) \in \mathcal{L}\big( \ell^\infty(\Wholes^2; \Real^d) \big)$
associated to \sref{eq:green:linsys:normal} and the Green's function $\mathcal{G}(t, t_0)$
for \sref{eq:green:linsys}.
\begin{lem}
\label{lem:green:relBetweenProjAndNormal}
Suppose that (Hf), (H$\Phi$) and $\mathit{(HS1)}_p$-$\mathit{(HS3)}_p$ are satisfied.
Consider any $v_{t_0} \in \ell^\infty(\Wholes^2; \Real^d)$
and $\theta_{t_0} \in \ell^\infty(\Wholes; \Real)$
for which $Q_{ct_0} v_{t_0} = 0 \in \ell^\infty(\Wholes; \Real)$.
Then upon writing
$\overline{w}(t)$ for the solution to \sref{eq:green:linsys:normal} with
the initial condition
\begin{equation}
 \overline{w}(t_0) = v_{t_0} + \theta_{t_0} T_{ct_0} \phi,
\end{equation}
the pair
\begin{equation}
v(t) = [ I - P_{ct}] \overline{w}(t), \qquad \theta(t) = Q_{ct} \overline{w}(t) \label{eq:green:vthetadef}
\end{equation}
satisfies \sref{eq:green:linsys} with
\begin{equation}
v(t_0) = v_{t_0}, \qquad \theta(t_0) = \theta_{t_0}.
\end{equation}
\end{lem}
\begin{proof}
By construction, for any $t \ge t_0$ we have
\begin{equation}
v(t) = \overline{w}(t) -  \theta(t)T_{ct}\phi.
\end{equation}
Turning first to the equation for $\theta(t)$, we compute
\begin{equation}
\begin{array}{lcl}
\dot{\theta}(t) & = & cQ'_{ct} \overline{w}(t) + Q_{ct} \dot{\overline{w}}(t) \\[0.2cm]
& = & c Q'_{ct} \overline{w}(t) + Q_{ct} L_{ct} \pi^\times \overline{w}(t) \\[0.2cm]
& = & c Q'_{ct} v(t) + c \theta(t) Q'_{ct} T_{ct} \phi + Q_{ct} L_{ct} \pi^\times v(t)
+  Q_{ct} L_{ct} \pi^\times  \theta(t) T_{ct} \phi.
\end{array}
\end{equation}
It hence suffices to show that for every $\theta \in \ell^\infty(\Wholes; \Real)$ we have
\begin{equation}
\label{eq:lem:grel:idToShow}
c \theta Q'_{ct} T_{ct} \phi + Q_{ct} L_{ct} \pi^\times \theta T_{ct} \phi = Q_{ct} L_{ct} M_{ct}  \theta^\diamond.
\end{equation}
To see this, note that the first order terms in the Taylor expansion of
\sref{eq:cds:expOfCommut} yield
\begin{equation}
\pi^\times \theta T_{ct} \phi -     \theta T_{ct} \tau  \phi = M  \theta^\diamond.
\end{equation}
In particular, we have
\begin{equation}
\label{eq:lem:grel:usefulID}
\begin{array}{lcl}
 L_{ct} \pi^\times \theta T_{ct} \phi & = &  L_{ct} M \theta^\diamond + \theta  L_{ct} T_{ct} \tau \phi \\[0.2cm]
& = &    L_{ct} M \theta^\diamond + c \theta T_{ct} \phi', \\[0.2cm]
\end{array}
\end{equation}
where in the last step we used the fact that $\mathcal{L}_0 \phi = 0$.
Furthermore, differentiating $Q_{ct} T_{ct} \phi = 1$ with respect to $t$, we find
\begin{equation}
c Q'_{ct} T_{ct} \phi + c Q_{ct} T_{ct} \phi' = 0,
\end{equation}
which establishes \sref{eq:lem:grel:idToShow}.

It remains to consider the equation satisfied by $v(t)$.
A short computation yields
\begin{equation}
\begin{array}{lcl}
\dot{v}(t) & = & \dot{\overline{w}}(t) - \dot{\theta}(t)T_{ct} \phi - c \theta(t) T_{ct} \phi' \\[0.2cm]
& = & L_{ct} \pi^\times \overline{w}(t) - P_{ct} L_{ct}[ \pi^\times v(t) + M_{ct} \theta^\diamond(t)]
  - c P'_{ct} v(t) - c \theta(t) T_{ct} \phi' \\[0.2cm]
& = & L_{ct} \pi^\times v(t) + L_{ct} \pi^\times \theta(t) T_{ct} \phi
  - P_{ct} L_{ct}[ \pi^\times v(t) + M_{ct} \theta^\diamond(t)] \\[0.2cm]
& & \qquad    - c P'_{ct}v(t) - c \theta(t) T_{ct} \phi' ,
\end{array}
\end{equation}
which together with \sref{eq:lem:grel:usefulID} completes the proof.
\end{proof}

\begin{cor}
\label{cor:green:expForGreenFull}
Consider the setting of Lemma \ref{lem:green:relBetweenProjAndNormal}.
The Green's function associated to the linear system \sref{eq:green:linsys}
can be represented as
\begin{equation}
  \mathcal{G}(t, t_0) = \left(\begin{array}{cc}
   \mathcal{G}_{v v}(t, t_0) & \mathcal{G}_{v\theta}(t, t_0) \\
  \mathcal{G}_{\theta v}(t, t_0) & \mathcal{G}_{\theta \theta}(t, t_0) \\
                \end{array}\right)
  =
  \left(\begin{array}{cc}
   [I-P_{ct} ]\overline{\mathcal{G}}(t, t_0) [I-P_{ct_0}] & [I-P_{ct}] \overline{\mathcal{G}}(t, t_0) T_{ct_0} \phi \\
   Q_{ct} \overline{\mathcal{G}}(t, t_0) [I-P_{ct_0}] & Q_{ct} \overline{\mathcal{G}}(t, t_0) T_{ct_0} \phi \\
                \end{array}\right).
\end{equation}
\end{cor}

In order to study the Green's function $\overline{\mathcal{G}}(t, t_0)$,
we exploit the fact that the LDE \sref{eq:green:linsys:normal} posed on $\Wholes^2$
decouples into a set of LDEs posed on $\Wholes$ after taking the Fourier transform
in the direction perpendicular to the propagation of the wave.
In particular,
let us introduce for any $1 \le p \le \infty$ and $q \in \{1, 2, \infty\}$ the frequency space
\begin{equation}
\widehat{\mathcal{X}}_{p, q} = \ell^p\big(\Wholes; L^q([ \pi, \pi], \Real^d) \big).
\end{equation}
For any $v \in \mathcal{X}_{p, 1} \cap \mathcal{X}_{p, 2}$ and $\theta \in \ell^1(\Wholes; \Real) \cap \ell^2(\Wholes; \Real)$
we can now introduce the sequence $\widehat{v} \in \widehat{\mathcal{X}}_{p, 2}$
and the function $\widehat{\theta} \in L^2( [ -\pi, \pi], \Real)$
defined by
\begin{equation}
\widehat{v}_n(\omega)  = \sum_{l \in \Wholes} v_{nl} e^{- i l \omega},
\qquad \widehat{\theta}(\omega) = \sum_{l \in \Wholes} \theta_l e^{- i l \omega}.
\end{equation}
In a standard fashion, these maps can be extended to maps $v \to \widehat{v}$ and $\theta \to \widehat{\theta}$
from $\mathcal{X}_{p, 2} \to \widehat{\mathcal{X}}_{p, 2}$ and $\ell^2(\Wholes; \Real) \to L^2([-\pi, \pi], \Real)$
respectively that are bounded and invertible via
\begin{equation}
v_{nl} =\frac{1}{2\pi} \int_{-\pi}^\pi  e^{i l \omega} \widehat{v}_n(\omega) d \omega,
\qquad \theta_l = \frac{1}{2 \pi} \int_{-\pi}^\pi e^{i l \omega} \widehat{\theta}(\omega).
\end{equation}
The following technical result gives useful bounds for frequency dependent convolution operators.

\begin{lem}
\label{lem:green:convKIsBounded}
Consider any set of maps $K_{\omega}: \ell^\infty(\Wholes; \Real^d) \to \ell^{\infty}(\Wholes; \Real^d)$
defined for $\omega \in [-\pi, \pi]$ that can be represented as
\begin{equation}
[K_{\omega} v]_n = \sum_{n_0 \in \Wholes} [K_{\omega}]_{n, n_0} v_{n_0}.
\end{equation}
Suppose that there exist constants $\eta > 0$ and $\vartheta \in \Real$ for which the bounds
\begin{equation}
[K_{\omega}]_{n, n_0} \le m(\omega) e^{ - \eta \abs{n  - n_0 + \vartheta} }, \qquad \omega \in [\pi, \pi]
\end{equation}
hold for some function $m(\cdot) \in L^\infty([-\pi, \pi], \Real)$.
Then there exists $C > 0$ that depends only on $\eta$
such that for any $1 \le p \le \infty$ and any pair $1 \le q_2 \le q_1 \le \infty$,
the map
\begin{equation}
K :   \widehat{\mathcal{X}}_{p, q_1} \to \widehat{\mathcal{X}}_{p, q_2}
\end{equation}
defined by
\begin{equation}
[K v](\omega) = K_{\omega} v(\omega)
\end{equation}
satisfies the estimate
\begin{equation}
\norm{K}_{\mathcal{L} \big(   \widehat{\mathcal{X}}_{p, q_1} , \widehat{\mathcal{X}}_{p, q_2} \big) }
\le C \norm{ m}_{\mathcal{L}\big(L^{q_1}([-\pi, \pi], \Real), L^{q_2}([-\pi, \pi], \Real)\big)}
\le C \norm{ m}_{L^s( [-\pi, \pi], \Real )  },
\end{equation}
in which
\begin{equation}
\frac{1}{s} = \frac{1}{q_2} - \frac{1}{q_1}.
\end{equation}
\end{lem}
\begin{proof}
Writing $M_{q_1, q_2} = \norm{ m}_{\mathcal{L}\big(L^{q_1}([-\pi, \pi], \Real), L^{q_2}([-\pi, \pi], \Real)\big)}$,
the generalized H{\"o}lder's inequality shows that
\begin{equation}
M_{q_1, q_2} \le \norm{ m}_{L^s( [-\pi, \pi], \Real )  }.
\end{equation}
Using Minkowski's inequality, we obtain
\begin{equation}
\begin{array}{lcl}
\norm{[Kv]_n(\cdot)}_{L^{q_2}([-\pi, \pi], \Real^d)}
& \le & \big[ \int_{-\pi}^\pi \abs{ \sum_{n_0 \in \Wholes}  e^{- \eta \abs{n-n_0 + \vartheta} } \abs{ v_{n_0}(\omega) } }^{q_2} \abs{ m(\omega) }^{q_2}  \, d \omega \big]^{1/q_2}\\[0.2cm]
& \le & \sum_{n_0 \in \Wholes}  e^{- \eta \abs{n-n_0 + \vartheta} } [\int_{-\pi}^\pi \abs{ m(\omega)}^{q_2} \abs{ v_{n_0}(\omega)}^{q_2} \, d \omega]^{1/q_2} \\[0.2cm]
& \le & \sum_{n_0 \in \Wholes}  e^{- \eta \abs{n-n_0 + \vartheta} } M_{q_1, q_2}    \norm{v_{n_0}(\cdot)}_{L^{q_1}([-\pi, \pi], \Real^d)}. \\[0.2cm]
\end{array}
\end{equation}
If $1 \le p < \infty$, an application of Young's inequality now yields
\begin{equation}
\begin{array}{lcl}
\norm{Kv}_{\widehat{\mathcal{X}}_{p, q_2} }
& \le & \big[ \sum_{n \in \Wholes} [ \sum_{n_0 \in \Wholes} C e^{- \eta \abs{n-n_0} } M_{q_1, q_2} \norm{v_{n_0}(\cdot)}_{L^{q_1}([-\pi, \pi], \Real^d)}]^p \big]^{1/p} \\[0.2cm]
& \le & C_1 \norm{ e^{ - \eta \abs{ \cdot + \vartheta} } }_{\ell^1(\Wholes; \Real)}
  M_{q_1, q_2} \norm{ v  }_{\widehat{\mathcal{X}}_{p, q_1}} \le C_2 M_{q_1, q_2} \norm{ v  }_{\widehat{\mathcal{X}}_{p, q_1}}.
\end{array}
\end{equation}
A simpler argument can be used for the remaining case $p = \infty$.
\end{proof}

Formally taking the Fourier transform of the LDE \sref{eq:green:linsys:normal}
yields the system of uncoupled LDEs
\begin{equation}
\partial_t \widehat{\overline{w}}(t, \omega) = L_{ct} \widehat{\pi}^\times_{\omega} \widehat{ \overline{w}}(t, \omega),
\qquad \omega \in [-\pi, \pi],
\end{equation}
in which we have introduced the map $\widehat{\pi}^\times_{\omega} \in \mathcal{L} \big( \ell^p(\Wholes; \Real^d) , \ell^p(\Wholes; (\Real^d)^5) \big)$
that acts as
\begin{equation}
[\widehat{\pi}^\times_{\omega} v]_n =
 \big(e^{i \sigma_2 \omega} v_{n + \sigma_1}, e^{-i \sigma_1 \omega} v_{n + \sigma_2},
 e^{-i \sigma_2 \omega} v_{n - \sigma_1}, e^{i \sigma_1 \omega} v_{n - \sigma_2}, v_n \big).
\end{equation}
We write $\overline{\mathcal{G}}_{\omega}(t, t_0) \in \mathcal{L}\big( \ell^p(\Wholes; \Real^d) \big)$
for the Green's function associated to the 1d LDE
\begin{equation}
\label{eq:cds:lde:1d:after:transform}
\frac{d}{dt} \overline{w}(t) = L_{ct} \widehat{\pi}^\times_{\omega} \overline{w}(t).
\end{equation}
The following preliminary result provides some basic information for the Green's function $\overline{\mathcal{G}}_{\omega}(t, t_0)$
that allows us to show that \sref{eq:green:linsys:normal} is well-posed as an evolution on $\mathcal{X}_{p, 2}$.
\begin{lem}
\label{lem:cds:greenFncProperForm}
Suppose that (Hf), (H$\Phi$),  $\mathit{(HS)}_{\mathrm{ess}}$,
$\mathit{(HS1)}_p$-$\mathit{(HS3)}_p$ and $\mathit{(HS1)}_\omega$-$\mathit{(HS2)}_{\omega}$ are all satisfied
and pick any $\epsilon > 0$.
Then the Green's function $\overline{\mathcal{G}}_{\omega}(t, t_0)$ can be represented as the convolution
\begin{equation}
[\overline{\mathcal{G}}_{\omega}(t, t_0) \overline{w} ]_n = \sum_{n_0 \in \Wholes}
[\overline{\mathcal{G}}_{\omega}(t, t_0)]_{n, n_0} \overline{w}_{n_0}
\end{equation}
and there exist constants $C > 0$ and $\eta > 0$  such that
we have the bound
\begin{equation}
\label{eq:green:greenFncProperForm:Bnd}
[\overline{\mathcal{G}}_{\omega}(t, t_0)]_{n, n_0} \le C e^{\epsilon ( t - t_0) } e^{ - \eta \abs{ n + c t - n_0 - ct_0 } }
\end{equation}
for every $t \ge t_0$ and $\omega \in [-\pi, \pi]$.
\end{lem}
\begin{proof}
Without loss of generality, we assume that $c > 0$. An important observation due to Benzoni-Gavage,
Huout and Rousset \cite[Thm. 4.2]{BGV2003} shows that
\begin{equation}
\label{eq:green:formulaForGreensFunction}
[\overline{\mathcal{G}}_{\omega}(t, t_0)]_{n, n_0} = - \frac{1}{2 \pi i } \int_{\gamma - i \pi c}^{\gamma + i \pi c}
  e^{\lambda (t - t_0) } G_{\omega}( n + ct, n_0 + c t_0 , \lambda) \, d \lambda,
\end{equation}
for any sufficiently large $\gamma \gg 1$,
in which the functions $G_{\omega}(\xi, \xi_0, \lambda)$ solve the MFDE
\begin{equation}
(\mathcal{L}_{\omega} - \lambda) G_{\omega}(\cdot, \xi_0, \lambda) = \delta(\cdot - \xi_0)
\label{eq:green:greenMFDE}
\end{equation}
in the sense of distributions. In view of the symmetry
\sref{eq:mr:spCompact}, we see that
\begin{equation}
G_{\omega}(\xi, \xi_0, \lambda + 2 \pi i c ) = e^{2 \pi i (\xi_0 - \xi) } G(\xi, \xi_0, \lambda),
\end{equation}
which implies that the integration contour in \sref{eq:green:formulaForGreensFunction} can be shifted to the left as long
as $G_{\omega}(\xi, \xi_0, \cdot)$ is analytic. The construction of the functions $G_{\omega}$
in \cite[Lem. 2.6]{HJHSTBFHN} shows that
on regions where $\mathcal{L}_{\omega} - \lambda$ is invertible,
the map $(\lambda, \omega) \mapsto G_{\omega}(\xi, \xi_0, \lambda)$
is continuous and analytic in the first variable.
In particular, our spectral assumptions allow us to pick $\gamma = \epsilon > 0$
for all $\omega \in [-\pi, \pi]$.

The bound \sref{eq:green:greenFncProperForm:Bnd} follows
from \cite[Eq. (2.35)]{HJHSTBFHN}. The exponent $\eta > 0$
can be chosen to be independent of $\omega$
due to the fact that the characteristic equations
\begin{equation}
\label{eq:green:greenFncProperFrom:charEq}
\det [ \Delta^\pm_{\omega}(z) - \epsilon I ]= 0
\end{equation}
have no roots in sufficiently narrow vertical strips $\abs{\Re z} \le 2 \eta$.
This is a consequence of $\mathrm{(HS)}_{\mathrm{ess}}$ together
with the {\it a-priori} bound on $\Im z$ established
in \cite[Lem. 3.1]{HJHCM}.
\end{proof}

%


\begin{lem}
\label{lem:green:wellPosed}
Suppose that (Hf), (H$\Phi$),  $\mathit{(HS)}_{\mathrm{ess}}$,
$\mathit{(HS1)}_p$-$\mathit{(HS3)}_p$ and $\mathit{(HS1)}_\omega$-$\mathit{(HS2)}_{\omega}$ are all satisfied.
Pick any $t_0 \in \Real$ and $\overline{w}_{t_0} \in  \mathcal{X}_{p, 2}$.
Write $\overline{w}(t)$ for the solution to \sref{eq:green:linsys:normal} with $\overline{w}(t_0) = \overline{w}_{t_0}$.
Then for all $t \ge t_0$ we have $\overline{w}(t) \in \mathcal{X}_{p, 2}$,
together with
\begin{equation}
\label{eq:cds:lem:fourier:welldefined}
\overline{w}_{nl} (t) =  \frac{1}{2 \pi} \int_{-\pi}^\pi e^{i l \omega}
 \sum_{n_0 \in \Wholes} [\overline{\mathcal{G}}_{\omega} (t, t_0)]_{n, n_0} [\widehat{\overline{w}_{t_0}}(\omega)]_{n_0} \, d \omega.
\end{equation}
\end{lem}
\begin{proof}
Let us first consider $\overline{w}_{t_0} \in \mathcal{X}_{p, 1} \cap \mathcal{X}_{p, 2}$.
The Fourier transform $\widehat{\overline{w}_{t_0}}(\omega)$
is now well-defined for all $\omega \in [-\pi, \pi]$ and
$\norm{ \widehat{\overline{w}_{t_0}} }_{\widehat{\mathcal{X}}_{p, \infty}} \le \norm{\overline{w}_{t_0} }_{\mathcal{X}_{p, 1}}$.
Upon writing $\zeta(t)(\omega) = \overline{\mathcal{G}}_{\omega} (t, t_0) \widehat{\overline{w}_{t_0}}(\omega)$,
we see by Lemma \ref{lem:green:convKIsBounded} and Lemma \ref{lem:cds:greenFncProperForm}
that  $\zeta(t) \in \widehat{\mathcal{X}}_{p, \infty} \subset \widehat{\mathcal{X}}_{p, 2}$ for all $t \ge t_0$.
Upon writing $\widetilde{w}(t) \in \mathcal{X}_{p, 2}$ for the inverse Fourier transform of $\zeta(t)$,
we find by construction that $\widetilde{w}$ solves \sref{eq:green:linsys:normal} with $\widetilde{w}(t_0) = \overline{w}_{t_0}$.
By uniqueness of solutions, we now must have $\widetilde{w}(t) = \overline{w}(t)$ for all $t \ge t_0$.
These conclusions can now be extended to $\overline{w}_{t_0} \in \mathcal{X}_{p, 2}$ in a standard fashion.
\end{proof}

In light of Proposition $\ref{prp:mr:melnikov}$,
we can define for any $t \in \Real$ and $\omega \in (-\omega_0, \omega_0)$
the operator
\begin{equation}
Q_{\omega,ct}(\omega)v  =  \sum_{n \in \Wholes} \psi_{\omega}(n + ct) v_n,
\end{equation}
together with the projection $P_{\omega, ct} = T_{ct} \phi_{\omega} Q_{\omega, ct}$.
These frequency dependent projections can be used to
isolate the leading order dynamics of \sref{eq:cds:lde:1d:after:transform}
associated to the eigenvalues $\lambda_{\omega}$.
This allows us to considerably refine
the crude estimates in Lemma \ref{lem:cds:greenFncProperForm}.

\begin{lem}
\label{lem:green:estForOverlineG:I}
Suppose that (Hf), (H$\Phi$),  $\mathit{(HS)}_{\mathrm{ess}}$,
$\mathit{(HS1)}_p$-$\mathit{(HS3)}_p$ and $\mathit{(HS1)}_\omega$-$\mathit{(HS2)}_{\omega}$ are all satisfied.
  Then there exist constants $\omega_0 > 0$, $C > 0$, $\beta > 0$ and $\eta > 0$ such that for any $\omega \in [-\pi, \pi]$
  with $\abs{\omega} \ge \omega_0$ we have
  \begin{equation}
    [\overline{\mathcal{G}}_{\omega}(t, t_0)]_{n, n_0} \le C e^{- \beta (t - t_0) } e^{ - \eta  \abs{n + ct - n_0 - ct_0}  },
  \end{equation}
  while for any $\omega \in (-\omega_0, \omega_0 )$ we have 
  \begin{equation}
     \overline{\mathcal{G}}_{\omega}(t , t_0) =  e^{\lambda_{\omega} (t - t_0)} T_{ct} \phi_{\omega}  Q_{\omega, ct_0}
 + \overline{\mathcal{H}}_{1, \omega}(t, t_0),
\end{equation}
together with the bound
\begin{equation}
\abs{ [ \overline{\mathcal{H}}_{1 ,\omega}( t, t_0)]_{n, n_0} } \le C e^{ - \beta(t - t_0) } e^{- \eta \abs{n + ct - n_0 - ct_0} }.
\end{equation}
\end{lem}
\begin{proof}
Recall the representation \sref{eq:green:formulaForGreensFunction}
and pick a sufficiently small $\delta_\lambda > 0$.
For convenience, write
\begin{equation}
\Omega = \{(\lambda, \omega) \in \Complex \times (-\omega_0, \omega_0 ) \mid 0 < \abs{\lambda -\lambda_\omega} < \delta_{\lambda} \}.
\end{equation}
In order to obtain the desired splittings for $\abs{\omega} < \omega_0$,
we adapt the approach in \cite[Lem. 2.7]{HJHSTBFHN} to the current $\omega$-dependent setting.
The crucial technical ingredient is the isolation
of the singular portion of the operators $(\mathcal{L}_{\omega} - \lambda)^{-1}$.

We first introduce the notation
$\sigma_{\max} = \max\{ \abs{\sigma_1}, \abs{\sigma_2} \}$
together with the integral operators
\begin{equation}
M_{\omega} (f) = \int_{-\infty}^{\infty}
 \langle \psi_{\omega}(\xi) ,  f(\xi) \rangle \, d\xi
\end{equation}
that map $L^{\infty}(\Real, \Complex^d)$ into $\Complex$.
Proceeding as in \cite[{\S}6]{HJHLIN},
we can show that for any $f \in L^{\infty}(\Real, \Complex^d)$,
a bounded pair $(v^-, v^+) =: L_0 f$
with $v^- \in BC_0( (-\infty, \sigma_{\max}], \Complex^d)$ and
$v^+ \in BC_0([-\sigma_{\max}, \infty), \Complex^d)$
can be constructed in such a way that $\mathcal{L}_0 v^\pm = f$ on $\Real_{\pm}$.
In addition, $v^+$ and $v^-$ agree on $[-\sigma_{\max},\sigma_{\max}]$
if and only $M_0(f) = 0$.
In particular, if $\mathcal{L}_0 p = f$ for some
$f \in L^{\infty}(\Real, \Complex^d)$ and $p \in BC_0(\Real, \Complex^d)$,
then in fact $p = L_0 f + \kappa \phi_0$ for some $\kappa \in \Complex$.

Possibly after decreasing $\omega_0 > 0$,
we note that for all $\omega \in (-\omega_0, \omega_0)$
the fixed point problem
\begin{equation}
(v^-, v^+) = L_0 f + L_0 [\mathcal{L}_0 - \mathcal{L}_{\omega} + \lambda_{\omega}](v^-, v^+),
\end{equation}
has a unique solution that we will write as $(v^-, v^+) = L_{\omega} f$.
Here the functions $f$ and $v^\pm$ belong to the same function spaces as above
and the map $\omega \to L_{\omega}$ is continuous.
We now set out to show that
if $[\mathcal{L}_\omega  - \lambda_{\omega} ]p = f$
for some $f \in L^{\infty}(\Real, \Complex^d)$ and $p \in BC_0(\Real, \Complex^d)$,
then in fact $p = L_{\omega} f + \kappa \phi_{\omega}$
for some $\kappa \in \Complex$.

To see this, it suffices to write $q_{\kappa} = p - \kappa \phi_{\omega}$
and choose $\kappa \in \Complex$ in such a way that $q_{\kappa} = L_{\omega} f$,
which can be formulated as
\begin{equation}
\begin{array}{lcl}
q_{\kappa} &  = &
  L_0 f + L_0 [\mathcal{L}_0 - \mathcal{L}_{\omega} + \lambda_{\omega}]q_{\kappa}
\\[0.2cm]
& = &
    L_0 f + L_0 [ \mathcal{L}_0 q_{\kappa} ] - L_0 f
\\[0.2cm]
& = & L_0 \mathcal{L}_0 q_{\kappa} \\[0.2cm]
& = & L_0 \mathcal{L}_0 p - \kappa L_0 \mathcal{L}_0 \phi_{\omega}.
\end{array}
\end{equation}
As stated above, we have
\begin{equation}
\begin{array}{lcl}
L_0 \mathcal{L}_0 p & = & p + \nu_p \phi_0,
\\[0.2cm]
L_0 \mathcal{L}_0 \phi_{\omega} &=& \phi_{\omega} + \nu_\omega \phi_{0},
\end{array}
\end{equation}
for some $\nu_p \in \Complex$ and $\nu_{\omega} \in \Complex$.
Since $\mathcal{L}_0 \phi_0 = 0$, we have $\nu_{0} = -1$
and by continuity we may hence assume $\nu_{\omega} \neq 0$.
We hence require the identity
\begin{equation}
\begin{array}{lcl}
q_{\kappa}
& = & p + \nu_p \phi_0 - \kappa \phi_{\omega} - \kappa \nu_{\omega} \phi_0
\\[0.2cm]
& = & q_{\kappa}  + [\nu_p - \kappa \nu_{\omega}] \phi_0,
\end{array}
\end{equation}
which can be satisfied by choosing $\kappa = \nu_{\omega}^{-1} \nu_p$.

Fix $(\lambda, \omega) \in \Omega$
and $f \in L^{\infty}(\Real, \Complex^d)$.
Upon writing $v = (\mathcal{L}_{\omega} - \lambda)^{-1}f$,
the discussion above implies that
\begin{equation}
v = L_{\omega} f + (\lambda - \lambda_{\omega}) L_{\omega} v + \kappa \phi_{\omega}
\end{equation}
for some $\kappa \in \Complex$.
In addition, since
\begin{equation}
(\mathcal{L}_{\omega} - \lambda_{\omega}) v = f +(\lambda -\lambda_{\omega}) v
\end{equation}
we must necessarily have
\begin{equation}
M_{\omega}( f + (\lambda -\lambda_{\omega}) v ) = 0,
\end{equation}
which in view of the normalization $M_{\omega}(\phi_{\omega}) = 1$
leads directly to the identity
\begin{equation}
\kappa = -(\lambda - \lambda_{\omega})^{-1} M_{\omega}(f)
- M_{\omega}(L_{\omega} f)
- (\lambda - \lambda_{\omega}) M_{\omega} (L_{\omega} v).
\end{equation}
The rescaled function
$\widetilde{v} = (\lambda - \lambda_{\omega}) v$
must therefore satisfy the fixed point problem
\begin{equation}
\widetilde{v} = -M_{\omega}(f) \phi_{\omega}
+ (\lambda - \lambda_{\omega}) [ L_{\omega} f - M_{\omega}(L_{\omega} f) \phi_{\omega}]
+ ( \lambda-\lambda_{\omega}) [L_{\omega} \widetilde{v} - M_{\omega}(L_{\omega} \widetilde{v}) \phi_{\omega}],
\end{equation}
which shows that we may write
\begin{equation}
(\mathcal{L}_{\omega} - \lambda)^{-1} f
= - \frac{M_{\omega}(f)}{ \lambda - \lambda_{\omega} }
+ \mathcal{B}( \lambda, \omega) f,
\end{equation}
for some linear operator $\mathcal{B}(\lambda, \omega)$
that maps $L^{\infty}(\Real, \Complex^d)$
into $BC_0(\Real, \Complex^d)$ and depends
continuously on the pair $(\lambda, \omega)$
and analytically on $\lambda$. In addition,
$\norm{\mathcal{B}(\lambda, \omega)}$
can be uniformly bounded for $(\lambda, \omega) \in \Omega$.

We can now proceed as in \cite[Lem. 2.7]{HJHSTBFHN} to write
\begin{equation}
\label{eq:green:rep:meromorphic}
G_{\omega}(\xi, \xi_0, \lambda) =
E_{\omega}(\xi, \xi_0, \lambda) + H_{\omega}(\xi, \xi_0, \lambda)
\end{equation}
for all pairs $( \lambda, \omega) \in \Omega$,
in which $E_{\omega}$ is explicitly given by
\begin{equation}
E_{\omega}(\xi, \xi_0, \lambda) =
 -(\lambda - \lambda_\omega)^{-1}
 \phi_\omega(\xi) \psi_{\omega}(\xi_0)^*,
\end{equation}
while the remainder $( \lambda, \omega) \mapsto H_{\omega}(\xi, \xi_0, \lambda)$
is continuous and analytic in the first variable. In addition,
the continuity properties of $\mathcal{B}(\lambda, \omega)$
can be used to show that there exist $C > 0$ and $\eta > 0$ for which the uniform bound
\begin{equation}
\abs{H_{\omega}(\xi, \xi_0, \lambda) }\le C e^{ - \eta \abs{\xi - \xi_0} }
\end{equation}
holds for all $(\lambda, \omega) \in \Omega$.

The meromorphic expansion \sref{eq:green:rep:meromorphic}
allows us to shift the integration contour in \sref{eq:green:formulaForGreensFunction}
to the left of the imaginary axis, picking up the appropriate residues.
Indeed, the constant $\beta > 0$ appearing in the statement comes from the fact that
Propositions \ref{prp:mr:melnikov}-\ref{prp:mr:omegaLarge} imply that
$\mathcal{L}_\omega - \lambda$ is Fredholm with index zero on the half-plane $\Re \lambda \ge - \beta$
and is invertible on this half plane except for the simple poles at $\lambda = \lambda_{\omega}$
which we have now accounted for.
\end{proof}


%

From now on, we assume without loss of generality that $\sigma_1 \neq 0$,
noting that this can always be realized by  flipping the horizontal and vertical directions if necessary.
The key is that we can write
\begin{equation}
\label{eq:green:taylorExpansionProjection}
\begin{array}{lcl}
T_{ct} \phi_{\omega}  Q_{\omega, ct_0} & = & T_{ct} \phi Q_{ct_0}
-i \sigma_1^{-1} [ e^{i \sigma_1\omega} - 1 ] T_{ct}  [\partial_{\omega} \phi]_{\omega = 0} Q_{ct_0}
-i \sigma_1^{-1} [ e^{i \sigma_1\omega} - 1 ] T_{ct}   \phi  [\partial_{\omega} Q_{\omega, ct_0}]_{\omega = 0} \\[0.2cm]
& & \qquad + \mathcal{Q}_{\omega} (t, t_0)
\end{array}
\end{equation}
for any $\omega \in (- \omega_0, \omega_0)$,
where the $C^2$-smoothness of the branch $\omega \mapsto (\phi_{\omega}, \psi_{\omega}) \in BC_{-\eta}(\Real, \Real^d)^2$
allows us to estimate
\begin{equation}
\abs{ [\mathcal{Q}_{\omega}(t, t_0)]_{n, n_0} }  \le C \omega^2 e^{ - \eta \abs{n + ct - n_0 - ct_0} }.
\end{equation}
Note that we are using $e^{i \sigma_1 \omega} -1$ to replace the regular
$O(\omega)$ term in the Taylor expansion of $T_{ct} \phi_{\omega} Q_{\omega, ct_0}$.
The interpretation of the formal symbol as a difference operator in the transverse spatial direction
is exploited heavily in the nonlinear argument in \S\ref{sec:nl}.

For any $t \ge 0$, we now introduce the function $m_t \in  L^{\infty}([-\pi, \pi], \Real)$
that is given by
\begin{equation}
m_t(\omega) = \left\{
  \begin{array}{lcl}
      e^{ \lambda_{\omega} t} & & \omega \in (-\omega_0, \omega_0), \\[0.2cm]
      0 & & \omega \notin (-\omega_0, \omega_0).
  \end{array}
\right.
\end{equation}
Exploiting the representation \sref{eq:green:taylorExpansionProjection}
allows us to obtain the following estimate.
\begin{cor}
\label{cor:green:expressionGOmega}
Consider the setting of Lemma \ref{lem:green:estForOverlineG:I}.
There exist constants $\omega_0 > 0$, $C > 0$, $\beta > 0$, $\kappa > 0$ and $\eta > 0$ such that for any
  $\omega \in [-\pi, \pi]$
  we may write
  \begin{equation}
    \label{eq:green:representation:II}
    \begin{array}{lcl}
     \overline{\mathcal{G}}_{\omega}(t , t_0) & = &  m_{t-t_0}(\omega) T_{ct} \phi  Q_{ ct_0}
  -i \sigma_1^{-1} [e^{i \sigma_1 \omega} - 1 ] m_{t-t_0}(\omega) T_{ct} [\partial_{\omega} \phi]_{\omega = 0} Q_{ct_0} \\[0.2cm]
& & \qquad  -i \sigma_1^{-1} [e^{i \sigma_1\omega} - 1 ]  m_{t-t_0}(\omega) T_{ct}   \phi  [\partial_{\omega} Q_{\omega, ct_0}]_{\omega = 0}
     \\[0.2cm]
& & \qquad
 + \overline{\mathcal{H}}_{2, \omega}( t, t_0) ,
 \end{array}
\end{equation}
in which we have the bound
\begin{equation}
\abs{ [\overline{\mathcal{H}}_{2, \omega}( t, t_0)]_{n, n_0} } \le C [\omega^2 e^{ - \kappa \omega^2 (t - t_0) } +  e^{ - \beta(t - t_0) } ] e^{- \eta \abs{n + ct - n_0 - ct_0} }.
\end{equation}
\end{cor}

Throughout the remainder of this section, we exploit the representation \sref{eq:green:representation:II}
to obtain detailed bounds for the components of the Green's function $\mathcal{G}(t, t_0)$.
We start by providing some basic preliminary estimates.

\begin{lem}
\label{lem:green:estOnMultOperators}
Fix $\kappa > 0$ and consider the multiplication operators induced by the functions $\widetilde{m}_{k,t}(\omega) = \abs{\omega}^k e^{- \kappa \omega^2 t }$.
There exists $C > 0$ such that the following estimates hold for all integers $k \in \{0, 1, 2\}$.
\begin{equation}
\begin{array}{lcl}
 \norm{\widetilde{m}_{k,t} }_{L^2([-\pi, \pi], \Real) \to L^2([-\pi, \pi], \Real)} & \le & C(1+t)^{-k/2}, \\[0.2cm]
\norm{\widetilde{m}_{k,t} }_{L^\infty([-\pi, \pi], \Real) \to L^2([-\pi, \pi], \Real)}& \le & C(1+t)^{-(k/2+1/4)}, \\[0.2cm]
\norm{\widetilde{m}_{k,t} }_{L^2([-\pi, \pi], \Real) \to L^1([-\pi, \pi], \Real)} & \le & C(1+t)^{-(k/2+1/4)}, \\[0.2cm]
\norm{\widetilde{m}_{k,t}}_{L^\infty([-\pi, \pi], \Real) \to L^1([-\pi, \pi], \Real)} & \le & C(1+t)^{-(k/2+1/2)}. \\[0.2cm]
\end{array}
\end{equation}
\end{lem}
\begin{proof}
These bounds follow from Lemma \ref{lem:green:convKIsBounded}, upon estimating $\norm{\widetilde{m}_{k,t}}_{L^s([-\pi,\pi], \Real)}$
with $s = \infty$ for the first estimate, $s = 2$ for the second and third estimates and $s = 1$ for the last estimate.
\end{proof}

For convenience, we write $\mathcal{A}(t, t_0): \ell^2(\Wholes; \Real) \to \ell^2(\Wholes; \Real)$ for the
Fourier multiplication operator that acts in frequency space as
\begin{equation}
\mathcal{A}(t, t_0):  \widehat{\theta}  \mapsto m_{t - t_0}(\omega) \widehat{\theta}(\omega) , \qquad \omega \in [-\pi, \pi].
\end{equation}
%
Using the bounds
\begin{equation}
\norm{\widehat{\theta}}_{L^\infty([-\pi, \pi], \Real) } \le \norm{\theta}_{\ell^1(\Wholes; \Real)},
\qquad
 \norm{\theta}_{\ell^\infty(\Wholes; \Real) } \le \frac{1}{2 \pi}\norm{\widehat{\theta}}_{L^1([-\pi, \pi], \Real)},
\end{equation}
the estimates in Lemma \ref{lem:green:estOnMultOperators} with $k = 0$ now imply that
\begin{equation}
\begin{array}{lcl}
\norm{ \mathcal{A}(t, t_0)}_{\ell^2(\Wholes; \Real) \to \ell^2(\Wholes; \Real)} & \le & C, \\[0.2cm]
\norm{ \mathcal{A}(t, t_0)}_{\ell^1(\Wholes; \Real) \to \ell^2(\Wholes; \Real)} & \le & C (1 + t - t_0)^{-1/4}, \\[0.2cm]
\norm{ \mathcal{A}(t, t_0)}_{\ell^2(\Wholes; \Real) \to \ell^\infty(\Wholes; \Real)}  & \le & C (1 + t - t_0)^{-1/4},\\[0.2cm]
\norm{ \mathcal{A}(t, t_0)}_{\ell^1(\Wholes; \Real) \to \ell^\infty(\Wholes; \Real)} & \le & C (1 + t - t_0)^{-1/2}. \\[0.2cm]
\end{array}
\end{equation}
We also introduce the shorthand $\mathcal{A}^\diamond(t, t_0)$ for the operator
$\theta \mapsto [\mathcal{A}(t, t_0) \theta]^\diamond
= \mathcal{A}(t, t_0) \theta^\diamond
$,
which in view of Lemma \ref{lem:green:estOnMultOperators} with $k = 1$ has the bounds
\begin{equation}
\begin{array}{lcl}
\norm{ \mathcal{A}^\diamond(t, t_0)}_{\ell^2(\Wholes; \Real) \to \ell^2(\Wholes; \Real^5)} & \le & C (1 + t - t_0)^{-1/2}, \\[0.2cm]
\norm{ \mathcal{A}^\diamond(t, t_0)}_{\ell^1(\Wholes; \Real) \to \ell^2(\Wholes; \Real^5)} & \le & C (1 + t - t_0)^{-3/4}, \\[0.2cm]
\norm{ \mathcal{A}^\diamond(t, t_0)}_{\ell^2(\Wholes; \Real) \to \ell^\infty(\Wholes; \Real^5)}  & \le & C (1 + t - t_0)^{-3/4},\\[0.2cm]
\norm{ \mathcal{A}^\diamond(t, t_0)}_{\ell^1(\Wholes; \Real) \to \ell^\infty(\Wholes; \Real^5)} & \le & C (1 + t - t_0)^{-1}. \\[0.2cm]
\end{array}
\end{equation}
Finally, for any unit vector $e \in \Real^5$ we write
$\mathcal{A}^{\diamond \diamond}_e(t, t_0)$
for the operator $\theta \mapsto [e \cdot \mathcal{A}^\diamond(t, t_0) \theta]^\diamond$,
which in view of Lemma \ref{lem:green:estOnMultOperators} with $k = 2$ has the bounds
\begin{equation}
\begin{array}{lcl}
\norm{ \mathcal{A}^{\diamond\diamond}_e(t, t_0)}_{\ell^2(\Wholes; \Real) \to \ell^2(\Wholes; \Real^5)} & \le & C (1 + t - t_0)^{-1}, \\[0.2cm]
\norm{ \mathcal{A}^{\diamond\diamond}_e(t, t_0)}_{\ell^1(\Wholes; \Real) \to \ell^2(\Wholes; \Real^5)} & \le & C (1 + t - t_0)^{-5/4}, \\[0.2cm]
\norm{ \mathcal{A}^{\diamond\diamond}_e(t, t_0)}_{\ell^2(\Wholes; \Real) \to \ell^\infty(\Wholes; \Real^5)}  & \le & C (1 + t - t_0)^{-5/4},\\[0.2cm]
\norm{ \mathcal{A}^{\diamond\diamond}_e(t, t_0)}_{\ell^1(\Wholes; \Real) \to \ell^\infty(\Wholes; \Real^5)} & \le & C (1 + t - t_0)^{-3/2}. \\[0.2cm]
\end{array}
\end{equation}
%
We now obtain the following estimates on the Green's function $\mathcal{G}(t, t_0)$
and the associated operators $\mathcal{G}^\diamond(t, t_0)$ that are defined by
\begin{equation}
\mathcal{G}^\diamond(t, t_0) \left( \begin{array}{l} v \\ \theta \end{array} \right)
= \big[ \mathcal{G}(t, t_0) \left( \begin{array}{l} v \\ \theta \end{array} \right) \big]^\diamond.
\end{equation}
\begin{prop}
\label{prp:green:bndsOnGreen}
Suppose that (Hf), (H$\Phi$),  $\mathit{(HS)}_{\mathrm{ess}}$,
$\mathit{(HS1)}_p$-$\mathit{(HS3)}_p$ and $\mathit{(HS1)}_\omega$-$\mathit{(HS2)}_{\omega}$ are all satisfied.
Pick any $1 \le p \le \infty$
and introduce the shorthand
\begin{equation}
\chi_{ct} = - i [I - P_{ct}] T_{ct} [\partial_{\omega} \phi]_{\omega = 0}.
\end{equation}
Then there exists $C > 0$ such that for any $t \ge t_0$, the Green's function $\mathcal{G}(t, t_0)$
for \sref{eq:green:linsys} satifies the bounds
\begin{equation}
\begin{array}{lcl}
\norm{\mathcal{G}_{v v}(t, t_0)}_{\mathcal{X}_{p, 2} \to \mathcal{X}_{p, 2}}
+ \norm{\mathcal{G}_{v \theta}(t, t_0) - \chi_{ct} \sigma_1^{-1} e_4 \cdot \mathcal{A}^\diamond(t, t_0) }_{\ell^2(\Wholes; \Real) \to \mathcal{X}_{p, 2}}
 & \le & C (1 + t - t_0)^{-1},  \\[0.2cm]
\norm{\mathcal{G}_{v v}(t, t_0)}_{\mathcal{X}_{p, 1} \to \mathcal{X}_{p, 2}}
+ \norm{\mathcal{G}_{v \theta}(t, t_0) -\chi_{ct} \sigma_1^{-1} e_4 \cdot \mathcal{A}^\diamond(t, t_0) }_{\ell^1(\Wholes; \Real) \to \mathcal{X}_{p, 2}}
& \le & C  (1 + t - t_0)^{-5/4},  \\[0.2cm]
\norm{\mathcal{G}_{v v}(t, t_0)}_{\mathcal{X}_{p, 2} \to \mathcal{X}_{p, \infty}}
+\norm{\mathcal{G}_{v \theta}(t, t_0) -\chi_{ct} \sigma_1^{-1} e_4 \cdot \mathcal{A}^\diamond(t, t_0) }_{\ell^2(\Wholes; \Real) \to \mathcal{X}_{p, \infty}}
 & \le & C (1 + t - t_0)^{-5/4},  \\[0.2cm]
\norm{\mathcal{G}_{v v}(t, t_0)}_{\mathcal{X}_{p, 1} \to \mathcal{X}_{p, \infty}}
+\norm{\mathcal{G}_{v \theta}(t, t_0) - \chi_{ct} \sigma_1^{-1} e_4 \cdot \mathcal{A}^\diamond(t, t_0) }_{\ell^1(\Wholes; \Real) \to \mathcal{X}_{p,\infty}}
 & \le & C (1 + t - t_0)^{-3/2},   \\[0.2cm]
\end{array}
\end{equation}
in which $e_4 = (0, 0, 0, 1, 0)$, together with
\begin{equation}
\begin{array}{lcl}
\norm{\mathcal{G}_{\theta v}(t, t_0)}_{\mathcal{X}_{p, 2} \to \ell^2(\Wholes; \Real) }
+ \norm{\mathcal{G}_{\theta \theta}(t,t_0) - \mathcal{A}(t, t_0)}_{\ell^2(\Wholes; \Real) \to \ell^2(\Wholes; \Real)}
 & \le & C (1 + t - t_0)^{-1/2}, \\[0.2cm]
\norm{\mathcal{G}_{\theta v}(t, t_0)}_{\mathcal{X}_{p, 1} \to \ell^2(\Wholes; \Real) }
+\norm{\mathcal{G}_{\theta \theta}(t,t_0) - \mathcal{A}(t, t_0)}_{\ell^1(\Wholes; \Real) \to \ell^2(\Wholes; \Real)}
& \le & C (1 + t - t_0)^{-3/4}, \\[0.2cm]
\norm{\mathcal{G}_{\theta v}(t, t_0)}_{\mathcal{X}_{p, 2} \to \ell^\infty(\Wholes; \Real) }
+\norm{\mathcal{G}_{\theta \theta}(t,t_0) - \mathcal{A}(t, t_0)}_{\ell^2(\Wholes; \Real) \to \ell^\infty(\Wholes; \Real)}
& \le & C (1 + t - t_0)^{-3/4}, \\[0.2cm]
\norm{\mathcal{G}_{\theta v}(t, t_0)}_{\mathcal{X}_{p, 1} \to \ell^\infty(\Wholes; \Real) }
+\norm{\mathcal{G}_{\theta \theta}(t,t_0) - \mathcal{A}(t, t_0)}_{\ell^1(\Wholes; \Real) \to \ell^\infty(\Wholes; \Real)}
& \le & C (1 + t - t_0)^{-1}. \\[0.2cm]
\end{array}
\end{equation}
In addition, for any $t \ge t_0$ the components of the function $\mathcal{G}^\diamond(t, t_0)$ have the bounds
\begin{equation}
\begin{array}{lcl}
\norm{\mathcal{G}^\diamond_{v v}(t, t_0)}_{\mathcal{X}_{p, 2} \to \mathcal{X}_{p, 2}^5}
+ \norm{\mathcal{G}^\diamond_{v \theta}(t, t_0)  }_{\ell^2(\Wholes; \Real) \to \mathcal{X}_{p, 2}^5}
 & \le & C (1 + t - t_0)^{-1},  \\[0.2cm]
\norm{\mathcal{G}^\diamond_{v v}(t, t_0)}_{\mathcal{X}_{p, 1} \to \mathcal{X}_{p, 2}^5}
+ \norm{\mathcal{G}^\diamond_{v \theta}(t, t_0) }_{\ell^1(\Wholes; \Real) \to \mathcal{X}_{p, 2}^5}
& \le & C  (1 + t - t_0)^{-5/4},  \\[0.2cm]
\norm{\mathcal{G}^\diamond_{v v}(t, t_0)}_{\mathcal{X}_{p, 2} \to \mathcal{X}_{p, \infty}^5}
+\norm{\mathcal{G}^\diamond_{v \theta}(t, t_0)}_{\ell^2(\Wholes; \Real) \to \mathcal{X}_{p, \infty}^5}
 & \le & C (1 + t - t_0)^{-5/4},  \\[0.2cm]
\norm{\mathcal{G}^\diamond_{v v}(t, t_0)}_{\mathcal{X}_{p, 1} \to \mathcal{X}_{p, \infty}^5}
+\norm{\mathcal{G}^\diamond_{v \theta}(t, t_0)  }_{\ell^1(\Wholes; \Real) \to \mathcal{X}_{p,\infty}^5}
 & \le & C (1 + t - t_0)^{-3/2},   \\[0.2cm]
\end{array}
\end{equation}
together with
\begin{equation}
\begin{array}{lcl}
\norm{\mathcal{G}^\diamond_{\theta v}(t, t_0)}_{\mathcal{X}_{p, 2} \to \ell^2(\Wholes; \Real^5) }
+ \norm{\mathcal{G}^\diamond_{\theta \theta}(t,t_0) - \mathcal{A}^\diamond(t, t_0)}_{\ell^2(\Wholes; \Real) \to \ell^2(\Wholes; \Real^5)}
 & \le & C (1 + t - t_0)^{-1}, \\[0.2cm]
\norm{\mathcal{G}^\diamond_{\theta v}(t, t_0)}_{\mathcal{X}_{p, 1} \to \ell^2(\Wholes; \Real^5) }
+\norm{\mathcal{G}^\diamond_{\theta \theta}(t,t_0) - \mathcal{A}^\diamond(t, t_0)}_{\ell^1(\Wholes; \Real) \to \ell^2(\Wholes; \Real^5)}
& \le & C (1 + t - t_0)^{-5/4}, \\[0.2cm]
\norm{\mathcal{G}^\diamond_{\theta v}(t, t_0)}_{\mathcal{X}_{p, 2} \to \ell^\infty(\Wholes; \Real^5) }
+\norm{\mathcal{G}^\diamond_{\theta \theta}(t,t_0) - \mathcal{A}^\diamond(t, t_0)}_{\ell^2(\Wholes; \Real) \to \ell^\infty(\Wholes; \Real^5)}
& \le & C (1 + t - t_0)^{-5/4}, \\[0.2cm]
\norm{\mathcal{G}^\diamond_{\theta v}(t, t_0)}_{\mathcal{X}_{p, 1} \to \ell^\infty(\Wholes; \Real^5) }
+\norm{\mathcal{G}^\diamond_{\theta \theta}(t,t_0) - \mathcal{A}^\diamond(t, t_0)}_{\ell^1(\Wholes; \Real) \to \ell^\infty(\Wholes; \Real^5)}
& \le & C (1 + t - t_0)^{-3/2}. \\[0.2cm]
\end{array}
\end{equation}
\end{prop}
\begin{proof}
In view of the representation of $\mathcal{G}(t, t_0)$
given in Corollary \ref{cor:green:expForGreenFull},
the bounds above can be derived from the expressions
for $\overline{\mathcal{G}}_{\omega}(t, t_0)$
given in Corollary \ref{cor:green:expressionGOmega}.
To this end, one can use the estimates in Lemma's \ref{lem:green:convKIsBounded} and
\ref{lem:green:estOnMultOperators}, together with the facts that $Q_{ct}T_{ct}\phi = 1$,
that $[I-P_{ct}]T_{ct}Q_{ct} = 0$ and that $Q_{ct_0}v_{t_0} = 0$.
\end{proof}

In the sequel, it is convenient to decompose the function $v(t)$ as
\begin{equation}
v(t) = w(t) + \chi_{ct} \sigma_1^{-1} e_4 \cdot \theta^\diamond(t).
\label{eq:green:wdef}
\end{equation}
As a direct consequence of Proposition \ref{prp:green:bndsOnGreen}
we can now obtain the following bounds for the linear evolution of $w(t)$.
\begin{cor}
\label{cor:green:bndsOnGreenForW}
Consider the setting of Proposition \ref{prp:green:bndsOnGreen}.
Upon writing
\begin{equation}
\mathcal{G}_{w}(t, t_0) = \big( \mathcal{G}_{wv}(t, t_0) \qquad \mathcal{G}_{w \theta}(t, t_0) \big)
\end{equation}
for the $1 \times 2$ operator valued matrix with components
\begin{equation}
\begin{array}{lcl}
\mathcal{G}_{w v}(t, t_0) & := & \mathcal{G}_{v v}(t, t_0) - \chi_{ct} \sigma_1^{-1} e_4 \cdot \mathcal{G}^\diamond_{\theta v}(t, t_0), \\[0.2cm]
\mathcal{G}_{w \theta}(t, t_0) & := & \mathcal{G}_{v \theta}(t, t_0) - \chi_{ct} \sigma_1^{-1} e_4 \cdot \mathcal{G}^\diamond_{\theta \theta}(t, t_0), \\[0.2cm]
\end{array}
\label{eq:green:Gwdefs}
\end{equation}
there exists $C > 0$ such that the following estimates hold for all $t \ge t_0$,
\begin{equation}
\begin{array}{lcl}
\norm{\mathcal{G}_{w v}(t, t_0)}_{\mathcal{X}_{p, 2} \to \mathcal{X}_{p, 2}}
+ \norm{\mathcal{G}_{w \theta}(t, t_0)}_{\ell^2(\Wholes; \Real) \to \mathcal{X}_{p, 2}}
 & \le & C (1 + t - t_0)^{-1},  \\[0.2cm]
\norm{\mathcal{G}_{w v}(t, t_0)}_{\mathcal{X}_{p, 1} \to \mathcal{X}_{p, 2}}
+ \norm{\mathcal{G}_{w \theta}(t, t_0) }_{\ell^1(\Wholes; \Real) \to \mathcal{X}_{p, 2}}
& \le & C  (1 + t - t_0)^{-5/4},  \\[0.2cm]
\norm{\mathcal{G}_{w v}(t, t_0)}_{\mathcal{X}_{p, 2} \to \mathcal{X}_{p, \infty}}
+\norm{\mathcal{G}_{w \theta}(t, t_0) }_{\ell^2(\Wholes; \Real) \to \mathcal{X}_{p, \infty}}
 & \le & C (1 + t - t_0)^{-5/4},  \\[0.2cm]
\norm{\mathcal{G}_{w v}(t, t_0)}_{\mathcal{X}_{p, 1} \to \mathcal{X}_{p, \infty}}
+\norm{\mathcal{G}_{w \theta}(t, t_0)  }_{\ell^1(\Wholes; \Real) \to \mathcal{X}_{p,\infty}}
 & \le & C (1 + t - t_0)^{-3/2}.   \\[0.2cm]
\end{array}
\end{equation}
\end{cor}

%

%

%

%

%


%

%

%

\section{Proof of Main Results}
\label{sec:nl}

In this section we set up the nonlinear iteration argument that is used to prove our main results.
In particular, we fix $1 \le p \le \infty$, pick an initial condition
\begin{equation}
v_0 \in \mathcal{X}^\perp_{p, 1}(0), \qquad \theta_0 \in \ell^1(\Wholes; \Real)
\end{equation}
and consider for $t \ge 0$ the fixed point problem
\begin{equation}
\label{eq:nl:fixpoint}
\begin{array}{lcl}
\left(\begin{array}{l}v(t) \\ \theta(t)\end{array}\right)
& = &     \mathcal{G}(t, 0) \left(\begin{array}{l}v_{0} \\ \theta_0\end{array}\right)
 + \int_{0}^t \mathcal{G}(t, t_0) \left(\begin{array}{l} 0 \\  \theta(t_0) H_4[ \tau v(t_0), \theta^\diamond(t_0) ] \end{array}\right) \, d t_0 \\[0.4cm]
& & \qquad +  \int_{0}^t \mathcal{G}(t, t_0) \left(\begin{array}{l} \mathcal{R}_{4, c t_0}\big( v(t_0), \theta(t_0) \big) \\  \mathcal{S}_{4, ct_0}\big( v(t_0), \theta(t_0) \big)\end{array}\right) \, d t_0.
\end{array}
\end{equation}
Upon writing
\begin{equation}
w(t) = v(t) - \chi_{ct} \sigma_1^{-1} e_4 \cdot \theta^\diamond(t),
\end{equation}
we observe by linearity that any solution to \sref{eq:nl:fixpoint} for some $t \ge 0$
necessarily satisfies the system
\begin{equation}
\label{eq:nl:fixpoint:derivatives}
\begin{array}{lcl}
\left(\begin{array}{l}v^\diamond(t) \\ \theta^\diamond(t) \\ w(t) \end{array}\right)
& = &     \left( \begin{array}{l}\mathcal{G}^\diamond(t, 0) \\ \mathcal{G}_w(t, 0) \end{array} \right)
   \left(\begin{array}{l}v_{0} \\ \theta_0\end{array}\right)
 + \int_{0}^t \left( \begin{array}{l}\mathcal{G}^\diamond(t, t_0) \\ \mathcal{G}_w(t, t_0) \end{array} \right)
  \left(\begin{array}{l} 0 \\  \theta(t_0) H_4[ \tau v(t_0), \theta^\diamond(t_0) ] \end{array}\right) \, d t_0 \\[0.4cm]
& & \qquad +  \int_{0}^t \left( \begin{array}{l}\mathcal{G}^\diamond(t, t_0) \\ \mathcal{G}_w(t, t_0) \end{array} \right)
  \left(\begin{array}{l} \mathcal{R}_{4, c t_0}\big( v(t_0), \theta(t_0) \big) \\  \mathcal{S}_{4, ct_0}\big( v(t_0), \theta(t_0) \big)\end{array}\right) \, d t_0.
\end{array}
\end{equation}

We write
\begin{equation}
E_0 = \norm{v_0}_{\mathcal{X}_{p, 1}} + \norm{\theta_0}_{\ell^1(\Wholes; \Real)}.
\end{equation}
We also introduce the non-decreasing functions
\begin{equation}
\begin{array}{lcl}
M_{v, 2}(t) & = & \sup_{0 \le t_0 \le t } (1 + t_0)^{3/4}  \norm{v(t_0)}_{ \mathcal{X}_{p, 2} }, \\[0.2cm]
M_{v, \infty}(t) & = & \sup_{0 \le t_0 \le t } (1 + t_0)^{1}  \norm{v(t_0)}_{ \mathcal{X}_{p, \infty}  }, \\[0.2cm]
M_{v^\diamond, 2}(t) & = & \sup_{0 \le t_0 \le t } (1 + t_0)^{9/8}  \norm{v^\diamond(t_0)}_{  \mathcal{X}^5_{p, 2} }, \\[0.2cm]
M_{v^\diamond, \infty}(t) & = & \sup_{0 \le t_0 \le t } (1 + t_0)^{5/4}  \norm{v^\diamond(t_0)}_{ \mathcal{X}^5_{p, \infty} }, \\[0.2cm]
\end{array}
\end{equation}
together with
\begin{equation}
\begin{array}{lcl}
M_{\theta, 2}(t) & = & \sup_{0 \le t_0 \le t } (1 + t_0)^{1/4}  \norm{\theta(t_0)}_{\ell^2(\Wholes; \Real)}, \\[0.2cm]
M_{\theta, \infty}(t) & = & \sup_{0 \le t_0 \le t } (1 + t_0)^{1/2}  \norm{\theta(t_0)}_{\ell^\infty(\Wholes; \Real)}, \\[0.2cm]
M_{\theta^\diamond, 2}(t) & = & \sup_{0 \le t_0 \le t } (1 + t_0)^{3/4}  \norm{\theta^\diamond(t_0)}_{\ell^2(\Wholes; \Real^5)}, \\[0.2cm]
M_{\theta^\diamond, \infty}(t) & = & \sup_{0 \le t_0 \le t } (1 + t_0)^{1}  \norm{\theta^\diamond(t_0)}_{\ell^\infty(\Wholes; \Real^5)} \\[0.2cm]
\end{array}
\end{equation}
and
\begin{equation}
\begin{array}{lcl}
M_{w, 2}(t) & = & \sup_{0 \le t_0 \le t }
(1 + t_0)^{9/8}
\norm{w(t_0)}_{  \mathcal{X}_{p, 2} }, \\[0.2cm]
M_{w, \infty}(t) & = & \sup_{0 \le t_0 \le t } (1 + t_0)^{5/4}  \norm{w(t_0)}_{ \mathcal{X}_{p, \infty} }. \\[0.2cm]
\end{array}
\end{equation}
Finally, we write
\begin{equation}
\begin{array}{lcl}
M(t) & = & M_{v, 2}(t) + M_{v, \infty}(t) + M_{v^\diamond, 2}(t) + M_{v^\diamond, \infty}(t) \\[0.2cm]
 & & \qquad + M_{\theta, 2}(t) + M_{\theta, \infty}(t) + M_{\theta^\diamond, 2}(t) + M_{\theta^\diamond, \infty}(t) \\[0.2cm]
& & \qquad + M_{w, 2}(t) + M_{w, \infty}(t).
\end{array}
\end{equation}

\begin{prop}
\label{prp:nl:bndOnFxP}
Suppose that (Hf), (H$\Phi$),  $\mathit{(HS)}_{\mathrm{ess}}$,
$\mathit{(HS1)}_p$-$\mathit{(HS3)}_p$ and $\mathit{(HS1)}_\omega$-$\mathit{(HS2)}_{\omega}$ are all satisfied
and pick any $1 \le p \le \infty$. Then there exists a constant $C > 0$ such that the following holds true.
Consider any initial condition $(v_0, \theta_0 ) \in \mathcal{X}^\perp_{p, 1}(0) \times \ell^1(\Wholes; \Real)$
and suppose that for some $T \ge 0$ the functions
\begin{equation}
v: [0, T] \to \mathcal{X}_{p, 2}, \qquad \theta: [0, T] \to \ell^2(\Wholes; \Real)
\end{equation}
satisfy the fixed point problem \sref{eq:nl:fixpoint} for all $0 \le t \le T$.
Then the functions $(v, \theta)$
satisfy the nonlinear LDE \sref{eq:cds:sys:glb:4} for all $0 \le t \le T$
and we have the bound
\begin{equation}
\label{eq:nl:prp:bndFxP:estOnM}
M(T) \le C E_0 + C M^2(T) + CM^3 (T).
\end{equation}
In addition, for all sufficiently small $\epsilon > 0$
there exist functions
\begin{equation}
\widetilde{v}: [0, T + \epsilon] \to \mathcal{X}_{p, 2}, \qquad \widetilde{\theta}: [0, T + \epsilon] \to \ell^2(\Wholes; \Real)
\end{equation}
that satisfy the fixed point problem \sref{eq:nl:fixpoint} for all $0 \le t \le T + \epsilon$
and have $\widetilde{v}(t) = v(t)$ and $\widetilde{\theta}(t) = \theta(t)$
for all $0 \le t \le T$.
\end{prop}


Notice that the exponents
$9/8$ and $5/4$
appearing in $M_{v^\diamond, q}(t)$ and $M_{w, q}(t)$
are slighly less than the exponents $5/4$ and $3/2$ that one would expect from the
linear evolution of the corresponding variables.
This is related to the fact that, even if we assume {\it a priori} that $M(t)$ is bounded,
the nonlinear term $\mathcal{R}_{4, ct_0}\big( v(t_0) , \theta(t_0) \big)$ decays at a rate of $O(t_0^{-1})$
in $\mathcal{X}_{p, 1}$
and $O(t_0^{-5/4})$ in $\mathcal{X}_{p, 2}$.
These rates are typically too slow to close a nonlinear estimate.
The same is true of the term $\theta(t_0) H_{4, ct}[ \tau v(t_0) , \theta^\diamond v(t_0)]$,
which also decays at a rate of $O(t_0^{-1})$ in $\ell^1(\Wholes; \Real)$.
The key is that we will exploit the special structure of the term $H_{4, ct}$
to show that the $O(t_0^{-1})$ decay is not a problem, as is the case when dealing with conservation laws.
The slow decay of $\mathcal{R}_{4, ct_0}$ is not a problem for the estimates on $M_{v, q}$, $M_{\theta, q}$
and $M_{\theta^\diamond, q}$ because the decay of the Green's functions $\mathcal{G}_{v v}$,
$\mathcal{G}_{\theta v}$ and $\mathcal{G}^\diamond_{\theta v}$
is already faster than that of $\mathcal{G}_{v \theta}$,
$\mathcal{G}_{\theta \theta}$ and $\mathcal{G}^\diamond_{\theta \theta}$,
respectively.

\begin{lem} \label{lem:nl:bnds}
Consider the setting of Proposition \ref{prp:nl:bndOnFxP}.
For any $0 \le t_0 \le T$ we have the estimates
\begin{equation}
\label{eq:nl:bnds:R}
\begin{array}{lcl}
\norm{ \mathcal{R}_{4, ct_0}\big( v(t_0), \theta(t_0) \big)  }_{\mathcal{X}_{p, 1} } & \le & C M^2(t_0) (1 + t_0)^{-1},
\\[0.2cm]
\norm{ \mathcal{R}_{4, ct_0}\big( v(t_0), \theta(t_0) \big)  }_{\mathcal{X}_{p, 2} } & \le & C M^2(t_0) (1 + t_0)^{-5/4},
\\[0.2cm]
\end{array}
\end{equation}
together with
\begin{equation}
\label{eq:nl:bnds:S}
\begin{array}{lcl}
\norm{ \mathcal{S}_{4, ct_0}\big( v(t_0), \theta(t_0) \big)  }_{\ell^1(\Wholes; \Real) } & \le & C[ M^2(t_0) + M^3(t_0)] (1 + t_0)^{-11/8},
\\[0.2cm]
\norm{ \mathcal{S}_{4, ct_0}\big( v(t_0), \theta(t_0) \big)  }_{\ell^2(\Wholes; \Real) } & \le & C [M^2(t_0) + M^3(t_0)]
(1 + t_0)^{-3/2}
\\[0.2cm]
\end{array}
\end{equation}
and
\begin{equation}
\label{eq:nl:bnds:H}
\begin{array}{lcl}
\norm{ \theta(t_0) H_{4, ct_0} [ \tau w(t_0), 0] }_{\ell^1(\Wholes; \Real) } & \le & C M^2(t_0)
(1 + t_0)^{-11/8}
, \\[0.2cm]
\norm{ \theta(t_0) H_{4, ct_0} [ \tau w(t_0), 0] }_{\ell^2(\Wholes; \Real) } & \le & C M^2(t_0) (1 + t_0)^{-3/2}. \\[0.2cm]
\end{array}
\end{equation}
\end{lem}
\begin{proof}
The estimates for $\mathcal{R}_{c t_0}$ and $\mathcal{S}_{c t_0}$ follow directly from Proposition \ref{prp:cds:bnds:Nonlins},
while the bounds for $H_{4, ct_0}$ follow from \sref{eq:cds:prp:bnd:elm1}.
\end{proof}

We will make use of the following integral inequality, a refined version of \eqref{eq:liuduhamelalg}, in order to estimate terms on the right hand side of the variation of constants formula.  It is often the case that given a fixed DuHamel term, there is a tradeoff between how good an estimate one can use on the semigroup and how good an estimate one can use on the nonlinear term.  In one case we have a difference that we can put on either term.  In other cases changing the norms will make one term decay more rapidly and the other decay more slowly.  The point of this lemma is that you don't have to choose whether to make the semigroup decay rapidly or whether to make the nonlinear term decay rapidly.
By using a fast estimate on the semigroup for $t_0 \in [0,t/2]$ and a fast estimate on the nonlinear term for $t_0 \in [t/2,t]$ you can have it both ways.  The observation is far from new, but its utility persists.

\begin{lem}[Lemma 3.2 in \cite{ChernLiu1987} (restated)]
\label{lem:nl:basicIntID}
Let $\alpha, \beta, \gamma$ and $\delta$ be positive real numbers,
with $\gamma \neq 1$ and $\delta \neq 1$.
Then there is a constant $C > 0$ such that
\begin{equation}
\int_0^{t/2} (1 + t-t_0)^{-\alpha}(1 + t_0)^{-\gamma}dt_0 + \int_{t/2}^t (1 + t-t_0)^{-\delta}(1 + t_0)^{-\beta}dt_0 \le C(1 + t)^{-\min\{\alpha + \gamma - 1,\delta + \beta - 1,\alpha,\beta\}}.
\label{eq:liuduhamelt/2}
\end{equation}
\end{lem}
\begin{proof}
In view of the fact that $\gamma \neq 1$, we may compute
\begin{equation}
\begin{array}{lcl}
\int_{0}^{\frac{1}{2}t} (1 + t - t_0)^{ - \alpha} (1 + t_0)^{-\gamma} d t_0 & \le &
 (1 + \frac{1}{2}t)^{-\alpha} \int_{0}^{\frac{1}{2}t} (1 +  t_0)^{ - \gamma} d t_0  \\[0.2cm]
& \le &  (1 + \frac{1}{2}t)^{-\alpha} \frac{1}{1 - \gamma} [ (1 + \frac{1}{2} t)^{1 - \gamma} - 1 ] \\[0.2cm]
& \le &  \frac{2}{\abs{1 - \gamma}}(1 + \frac{1}{2}t)^{- \min\{ \alpha , \alpha + \gamma - 1 \} }.
\end{array}
\end{equation}
Similarly, using $\delta \neq 1$, we compute
\begin{equation}
\begin{array}{lcl}
\int_{\frac{1}{2}t}^t (1 + t - t_0)^{ - \delta} (1 + t_0)^{-\beta} d t_0 & \le &
 (1 + \frac{1}{2}t)^{-\beta} \int_{\frac{1}{2}t}^t (1 + t - t_0)^{ - \delta} d t_0  \\[0.2cm]
& \le &  (1 + \frac{1}{2}t)^{-\beta} \frac{1}{1 - \delta} [ (1 + \frac{1}{2} t)^{1 - \delta} - 1] \\[0.2cm]
& \le &  \frac{2}{\abs{1 - \delta}}  (1 + \frac{1}{2} t)^{ - \min\{ \beta , \beta + \delta - 1 \} }.
\end{array}
\end{equation}
\end{proof}


\begin{cor} \label{cor:nl:dh}
Consider the setting of Proposition \ref{prp:nl:bndOnFxP}. Let $S(\cdot, \cdot)$ denote any of the semigroups
appearing in Proposition \ref{prp:green:bndsOnGreen}
and Corollary \ref{cor:green:bndsOnGreenForW}.
In addition, let $\mathcal{N}(t_0)$ denote any of the terms
$\mathcal{R}_{4,ct_0}(v(t_0),\theta(t_0))$, $\mathcal{S}_{4,ct_0}(v(t_0),\theta(t_0))$,
$\theta(t_0) H_{4, ct_0}[\tau v(t_0), 0]$ or $\theta(t_0) H_{4,ct_0}[\tau w(t_0),0]$.
Assume that the expression $S(t, t_0) \mathcal{N}(t_0)$ is well defined for all $0 \le t_0 \le t$.
Let $\| \mathcal{N}(t_0)\|_q$
denote
either the $\ell^q(\Wholes; \Real)$ or $\mathcal{X}_{p,q}$ norm of $\mathcal{N}(t_0)$
as appropriate and let $\|S(t, t_0)\|_{q_1 \to q_2}$
denote the norm of $S(t, t_0)$ viewed as a map from
one of $\ell^{q_1}(\Wholes; \Real)$ or $\mathcal{X}_{p, q_1}$
into one of $\ell^{q_2}(\Wholes; \Real)$ or $\mathcal{X}_{p, q_2}$, as appropriate.

Pick constants $\nu > 0$, $\alpha \ge 0$, $\beta \ge 0$, $\gamma \ge 0$ and $\delta \ge 0$.
Write $\gamma^c = \gamma - 1$ if $\gamma \neq 1$ and $\gamma^c = -\nu$ if $\gamma = 1$.
Similarly, write $\delta^c = \delta - 1$ if $\delta \neq 1$ and $\delta^c = -\nu$ if $\delta = 1$.

Then there exists a constant $C > 0$
such that for any $C_1 > 0$, $C_2 > 0$ and $t \ge 0$,
the estimates
\begin{equation}
\|S(t,t_0)\|_{1 \to q} \le C_1(1 + t-t_0)^{-\alpha},
\qquad
\|\mathcal{N}(t_0)\|_{1} \le C_2(1 + t_0)^{-\gamma}, \qquad 0 \le t_0 \le t,
\end{equation}
together with
\begin{equation}
\|S(t,t_0)\|_{2 \to q} \le C_1(1 + t-t_0)^{-\delta},
\qquad
\|\mathcal{N}(t_0)\|_{2} \le C_2(1 + t_0)^{-\beta}, \qquad 0 \le t_0 \le t,
\end{equation}
imply the bound
\begin{equation}
\left\|\int_0^t S(t,t_0)\mathcal{N}(t_0) dt_0 \right\|_{q} \le C C_1C_2(1 + t)^{-\min\{ \alpha + \gamma^c, \beta + \delta^c,  \alpha,\beta  \}}.
\label{eq:nl:switcheroo}
\end{equation}
\end{cor}
\begin{proof}
Apply Lemma \ref{lem:nl:basicIntID}, replacing $\gamma$ with $\gamma = 1 - \nu$ in case $\gamma = 1$
and treating $\delta$ similarly.
\end{proof}

Our applications of Corollary \ref{cor:nl:dh} will always have $\alpha > \delta$
and $\beta > \gamma$, which is particulary useful
since $\gamma^c$ and $\delta^c$ do not appear by themselves in the arguments of the minimum
appearing in \sref{eq:nl:switcheroo}.
In addition, we always have $\gamma \ge 1$
and $\beta \ge \frac{5}{4}$. Notice that this allows us to show that the
part of $M(t)$ associated to the terms
\begin{equation}
\label{eq:nl:termsWithWAndNonlins}
\begin{array}{l}
 \int_{0}^t \left( \begin{array}{c} \mathcal{G}(t, t_0) \\ \mathcal{G}^\diamond(t, t_0) \\ \mathcal{G}_w(t, t_0) \end{array} \right)
   \left(\begin{array}{l} 0 \\  \theta(t_0) H_4[ \tau w(t_0), 0 ] \end{array}\right) \, d t_0
  +  \int_{0}^t \left( \begin{array}{c} \mathcal{G}(t, t_0) \\ \mathcal{G}^\diamond(t, t_0) \\ \mathcal{G}_w(t, t_0) \end{array} \right)
    \left(\begin{array}{l} \mathcal{R}_{4, c t_0}\big( v(t_0), \theta(t_0) \big)
   \\  \mathcal{S}_{4, ct_0}\big( v(t_0), \theta(t_0) \big)\end{array}\right) \, d t_0
\end{array}
\end{equation}
appearing in \sref{eq:nl:fixpoint} and \sref{eq:nl:fixpoint:derivatives} can be bounded as desired.
The same holds true for the part of $M(t)$ associated to the terms
\footnote{
  Note that there is some overlap between the terms
  appearing in \sref{eq:nl:termsWithWAndNonlins} and \sref{eq:nl:termsWithVThetaI}.
}
\begin{equation}
\label{eq:nl:termsWithVThetaI}
\begin{array}{l}
 \int_{0}^t \left( \begin{array}{cc} \mathcal{G}_{vv}^\diamond(t, t_0) &
                                     \mathcal{G}_{v \theta}^\diamond(t, t_0) \\
                                     \mathcal{G}_{wv}(t, t_0)  &
                                     \mathcal{G}_{w\theta}(t, t_0) \\
                   \end{array} \right)
   \left(\begin{array}{l} 0 \\  \theta(t_0) H_4[ \tau v(t_0), \theta^\diamond(t_0) ] \end{array}\right) \, d t_0 .
\end{array}
\end{equation}
Indeed, the critical case $\gamma = 1$
occurs in \sref{eq:nl:termsWithWAndNonlins} only when coupling $\mathcal{R}_4$ with $\mathcal{G}_{vv}$,
$\mathcal{G}_{\theta v}$, $\mathcal{G}^\diamond_{vv}$, $\mathcal{G}^\diamond_{\theta v}$
and $\mathcal{G}_{wv}$, which all have $1 \to q$ decay rates $\alpha$ that are faster than
those of $M_{v, q}$, $M_{\theta, q}$,
$M_{v^\diamond, q}$, $M_{\theta^\diamond, q}$ and $M_{w, q}$ respectively.
Similarly, the semigroups $\mathcal{G}^\diamond_{v \theta}$
and $\mathcal{G}_{w \theta}$ appearing in \sref{eq:nl:termsWithVThetaI}
have $1 \to q$ decay rates $\alpha$ that are faster than those of $M_{v^\diamond, q}$
and $M_{w , q}$ respectively. Finally, only $M_{v^\diamond, \infty}$
and $M_{w, \infty}$ feature the exponent $\frac{5}{4}$. Since $\beta \ge \frac{5}{4}$
and the $2 \to \infty$ norms
of $\mathcal{G}^\diamond_{v v}$, $\mathcal{G}^\diamond_{v \theta}$,
$\mathcal{G}_{w v}$ and $\mathcal{G}_{w \theta}$ all  have $\delta = \frac{5}{4}$,
the critical case $\delta = 1$ occurs only
when the decay rate $\beta$ of the $\norm{ \cdot}_2$ norm of the nonlinearity is
strictly larger
than the decay rate of the relevant component of $M(t)$.

The exact same arguments can also be applied for the components of $M(t)$ associated to
\begin{equation}
\label{eq:nl:termsWithVThetaII}
\begin{array}{l}
 \int_{0}^t \left( \begin{array}{cc} \mathcal{G}_{vv}(t, t_0) &
                                     \mathcal{G}_{v\theta}(t, t_0) - \chi_{ct} \sigma_1^{-1} e_4 \cdot \mathcal{A}^\diamond(t, t_0) \\
                                     \mathcal{G}_{\theta v}(t, t_0) &
                                     \mathcal{G}_{\theta \theta}(t, t_0) - \mathcal{A}(t, t_0) \\
                                     \mathcal{G}^\diamond_{\theta v}(t, t_0) &
                                     \mathcal{G}^\diamond_{\theta \theta}(t, t_0) - \mathcal{A}^\diamond(t, t_0) \\
                                     \end{array} \right)
   \left(\begin{array}{l} 0 \\  \theta(t_0) H_4[ \tau v(t_0), \theta^\diamond(t_0) ] \end{array}\right) \, d t_0.
\end{array}
\end{equation}
It hence remains to consider the terms
\begin{equation}
\label{eq:nl:termsWithThetaTheta}
\begin{array}{l}
 \int_{0}^t \left( \begin{array}{cc} 0 & \chi_{ct} \sigma_1^{-1} e_4 \cdot \mathcal{A}^\diamond(t, t_0) \\
                                     0 & \mathcal{A}(t, t_0) \\
                                     0 & \mathcal{A}^\diamond(t, t_0)
                   \end{array} \right)
   \left(\begin{array}{l} 0 \\  \theta(t_0) H_4[ \tau \chi_{c t_0} \sigma_1^{-1} e_4 \cdot \theta^\diamond(t_0), \theta^\diamond(t_0) ] \end{array}\right) \, d t_0.
\end{array}
\end{equation}
These components require a more refined analysis and are discussed in our next result.

\begin{lem}
\label{lem:nl:firsttermest}
Consider the setting of Proposition \ref{prp:nl:bndOnFxP}.
There exists a constant $C > 0$
such that for any unit vector $e \in \Real^5$ and $t \ge 0$,
the quantities
\begin{equation}
\begin{array}{lcl}
\mathcal{I}_{q}(t) := \int_{ 0}^t \norm{ \mathcal{A}(t, t_0)  \theta(t_0) \theta^\diamond(t_0) }_{\ell^q(\Wholes; \Real^5)}  \, d t_0, \\[0.2cm]
\mathcal{I}^\diamond_{q,e}(t) := \int_{ 0}^t \norm{ \mathcal{A}^\diamond(t, t_0) e \cdot \theta(t_0) \theta^\diamond(t_0) }_{\ell^q(\Wholes; \Real^5)}  \, d t_0
\end{array}
\end{equation}
satisfy the estimates
\begin{equation}
  \begin{array}{lcl}
    \mathcal{I}_{2}(t) & \le & C M^2(t) (1 + t)^{-1/4}, \\[0.2cm]
    \mathcal{I}_{\infty}(t) & \le & C M^2(t) (1 + t)^{-1/2}, \\[0.2cm]
    \mathcal{I}^\diamond_{2, e}(t) & \le & C M^2(t) (1 + t)^{-3/4}, \\[0.2cm]
    \mathcal{I}^\diamond_{\infty, e}(t) & \le & C M^2(t) (1 + t)^{-1}. \\[0.2cm]
  \end{array}
\end{equation}
\end{lem}
\begin{proof}
Restricting ourselves for the moment to $\mathcal{I}_{q}(t)$, we make use of \eqref{eq:liuduhamelt/2} with exponents given by the following table.

\begin{tabular}{c|c|c|}
& $\mathcal{I}_2$ & $\mathcal{I}_\infty$ \\
--- & --- & --- \\
$\alpha_a$ & $\|\mathcal{A}(t,t_0)\|_{\ell^1 \to \ell^2} \sim (t-t_0+1)^{-1/4}$ & $\|\mathcal{A}(t,t_0)\|_{\ell^1 \to \ell^\infty} \sim (t-t_0+1)^{-1/2}$ \\
$\gamma_a$ & $\|\theta^\diamond(t_0)\|^2_{\ell^2} \sim M(t)^2(t_0+1)^{-3/2}$ & $\|\theta^\diamond(t_0)\|_{\ell^2}^2 \sim M(t)^2(t_0+1)^{-3/2}$ \\
$\alpha_b$ & $\|\mathcal{A}^\diamond(t,t_0)\|_{\ell^1 \to \ell^2} \sim (t-t_0+1)^{-3/4}$ & $\|\mathcal{A}^\diamond(t,t_0)\|_{\ell^1 \to \ell^\infty} \sim (t-t_0+1)^{-1}$ \\
$\gamma_b$ & $\|\theta(t_0)\|_{\ell^2}^2 \sim M(t)^2(t_0+1)^{-1/2}$ & $\|\theta(t_0)\|_{\ell^2}^2 \sim M(t)^2(t_0+1)^{-1/2}$ \\
$\delta$ & $\|\mathcal{A}(t,t_0)\|_{\ell^2 \to \ell^2} \sim C$ & $\|\mathcal{A}(t,t_0)\|_{\ell^2 \to \ell^\infty} \sim (t-t_0+1)^{-1/4}$ \\
$\beta$ & $\|\theta^\diamond\|_{\ell^\infty} \|\theta\|_{\ell^2} \sim M(t)^2(t-t_0+1)^{-5/4}$ & $\|\theta^\diamond\|_{\ell^\infty} \|\theta\|_{\ell^2} \sim M(t)^2(t-t_0+1)^{-5/4}$
\end{tabular}

\bigskip

The last two rows are straightforward to understand.  This corresponds to the time interval $[t/2,t]$ on which the semigroup is
weak regardless of its rate of decay and the accumulation of the nonlinear terms is what is dangerous.
Our nonlinear terms $\|\theta \theta^\diamond \|_{\ell^2} \le  \|\theta\|_{\ell^2} \|\theta^\diamond\|_{\ell^\infty} $
decay so rapidly that we can use even the incredibly weak $\ell^2 \to \ell^q$ estimates for $\mathcal{A}$.
This is because the exponent $\delta$ describing the decay of $\mathcal{A}$ appears on the right hand side of \eqref{eq:liuduhamelt/2}
only in the form $\beta + \delta - 1$ and never by itself, thus a large $\beta$ can compensate for a small $\delta$.
In our case we have $\beta + \delta - 1 = 1/4$ for $q = 2$ and $1/2$ for $q = \infty$,
enough to establish the desired rates of decay for $\mathcal{I}_q$, as long as the $[0,t/2]$ intervals cooperate.

The interval $[0,t/2]$ is more subtle.  As per \eqref{eq:uux} we decompose the nonlinear term $\theta \theta^\diamond$
into the sum of an (a) term, $(\theta^\diamond)^2$ (where we are interpreting the square of a vector to be the vector of the squares of the components) and a (b) term $(\theta^2)^\diamond$.  The first of these we estimate in the usual way using the $\ell^1 \to \ell^2$ norm on the semigroup.  For the second of these, we make use of the definition of $\mathcal{A}^\diamond$ to transfer the $\diamond$ from $\theta^2$ to $\mathcal{A}$, effectively summing by parts.

To establish the estimates for $\mathcal{I}^\diamond_{q,e}$ simply replace
$\mathcal{A}$ everywhere by $\mathcal{A}^\diamond$ and $\mathcal{A}^\diamond$ by $\mathcal{A}^{\diamond\diamond}_e$,
adding 1/2 to all values of $\alpha$ and $\delta$.
Notice that the critical situations $\gamma = 1$ and $\delta = 1$ are never encountered.
\end{proof}

\begin{proof}[Proof of Proposition \ref{prp:nl:bndOnFxP}]
The fact that solutions to the fixed point problem \sref{eq:nl:fixpoint} are solutions to the LDE
follows from direct differentiation and the linear well-posedness result
Lemma \ref{lem:green:wellPosed}.
The short term existence claims involving the constant $\epsilon > 0$
can be established in a standard fashion by applying the contraction mapping principle.
Finally, the bound \sref{eq:nl:prp:bndFxP:estOnM}
has been verified by the discussion above.

\end{proof}

\begin{proof}[Proof of Theorem \ref{thm:mr:stbAnyDirection}]
In view of the bound \sref{eq:nl:prp:bndFxP:estOnM},
one sees that whenever $E_0$ is sufficiently small
the inequality $M(T) \le 3 C E_0$ must hold.
In particular, if $E_0$ is sufficiently small
it is possible to
construct a solution to the fixed point problem
\sref{eq:nl:fixpoint} that is defined for all $t \ge 0$
by repeatedly applying Proposition \ref{prp:nl:bndOnFxP}.
The bounds for $\theta_{\mathrm{diff}}$
stated in item (iii) can be established in the same way
as the bounds for $\theta^\diamond(t)$.
The remainder of the decay rates given in the statement of the theorem
follow directly from the definition of $M(t)$.
\end{proof}

\begin{proof}[Proof of Theorem \ref{thm:mr:stbDiagonal}]
Upon writing
\begin{equation}
\theta^{\diamond \diamond}_l = \big(\theta_{l + \sigma_1} + \theta_{l - \sigma_1} - 2 \theta_l, \theta_{l + \sigma_2} + \theta_{l - \sigma_2} - 2 \theta_l, 0 \big),
\end{equation}
the estimates for the nonlinearities $\mathcal{R}_{4, ct}$ and $\mathcal{S}_{4, ct}$
that appear in Proposition \ref{prp:cds:bnds:Nonlins} remain valid
upon replacing all occurences of $\norm{\theta} \norm{\theta^{\diamond}}$ by $\norm{\theta} \norm{\theta^{\diamond \diamond}}$.
In addition, there exists an operator $H_{5, ct}$ such that $\theta H_{4, ct}[0, \theta^{\diamond}] = \theta H_{5, ct}\theta^{\diamond \diamond}$.
Finally, we have $[\partial_{\omega} \phi_{\omega}]_{\omega = 0} = 0$
and $[\partial_{\omega} \psi_{\omega}]_{\omega = 0} = 0$,
which means that $v(t) = w(t)$. In particular,
upon writing
\begin{equation}
\begin{array}{lcl}
M_{\theta^{\diamond\diamond}, 2}(t) & = & \sup_{0 \le t_0 \le t }
  (1 + t_0)^{5/4}  \norm{\theta^{\diamond\diamond}(t_0)}_{\ell^2(\Wholes; \Real^3)}, \\[0.2cm]
M_{\theta^{\diamond\diamond}, \infty}(t) & = & \sup_{0 \le t_0 \le t }
  (1 + t_0)^{3/2}  \norm{\theta^{\diamond\diamond}(t_0)}_{\ell^\infty(\Wholes; \Real^3)} \\[0.2cm]
\end{array}
\end{equation}
and replacing all references to $M_{v, q}$ by $M_{\theta^{\diamond\diamond}, q}$,
all the nonlinear terms that decay at the rate $O(t_0^{-1})$ have disappeared.
This allows us to 
replace the exponent $9/8$ appearing in the definition of $M_{w, 2}$ by $5/4$ and
the exponent $5/4$ appearing in the definition of $M_{w, \infty}$ by $3/2$.

With these choices, all the nonlinearities decay at a rate of $O( t_0^{-3/2})$
in $\ell^1(\Wholes; \Real)$ or $\mathcal{X}_{p, 1}$ and $O(t_0^{-7/4})$
in $\ell^2(\Wholes; \Real)$ or $\mathcal{X}_{p, 2}$. This is sufficiently
fast to close the nonlinear argument with direct crude estimates.
\end{proof}

\begin{proof}[Proof of Theorem \ref{thm:mr:stbParallel}]
The claims concerning the triplets $(\lambda_{\omega}, \phi_{\omega}, \psi_{\omega})$
follow directly upon inspection of \sref{eq:mr:defTauOmega}, which shows that
\begin{equation}
  \mathcal{L}_{\omega} p = \mathcal{L}_0 p + 2 \kappa ( \cos(\omega) - 1) p.
\end{equation}
This in turn implies $\lambda_{\omega} = 2 \kappa (\cos(\omega) - 1)$, $\phi_{\omega} = \phi$
and $\psi_{\omega} = \psi$,
which immediately establishes the two conditions $\mathrm{(HS1)}_{\omega}$ and $\mathrm{(HS2)}_{\omega}$.
Corollary \ref{cor:green:expForGreenFull} now implies that both $\mathcal{G}_{vv}(t, t_0)$
and $\mathcal{G}_{v \theta}(t, t_0)$ decay exponentially instead of algebraically.
As above in the proof of Theorem \ref{thm:mr:stbDiagonal},
the nonlinearities decay
at a rate of $O(t_0^{-7/4})$
in $\ell^2(\Wholes; \Real)$ or $\mathcal{X}_{p, 2}$.
Since these nonlinear decay rates are now the only restriction
on the decay of $w$, we can replace
both of the exponents  $9/8$ and $5/4$ appearing in the definitions of $M_{w, 2}$ and $M_{w, \infty}$
respectively by $7/4$.
\end{proof}

\section{The Nagumo Equation}
\label{sec:nag}

We return now to the scalar equation \eqref{eq:int:nagumo:lde} which was discussed in the introduction.
The goal of this section is twofold: firstly we show that the main theorems are not vacuous by exhibiting an
example to which they apply; and secondly we present numerical studies which highlight
the sensitive nature of anisotropy in the problem.

We again write
\begin{equation}
\label{eq:nag:LDE}
\dot{u}_{ij}(t) = [\Delta^+ u(t)]_{ij} + g(u_{ij}(t) ; \rho),
\end{equation}
in which the nonlinearity $g$ is given by
\begin{equation}
\label{eq:nagumo:g}
g(u; \rho) =   - \frac{5}{2}(u^2 - 1)(u - \rho),
\end{equation}
with $\rho \in (-1, 1)$.  The parameter $\rho$ determines the location of the
middle zero of the cubic and can be regarded as a measure of the difference in
the depth of the wells of the potential $G(u) = \int g(u)du$.   When $\rho$ is far
from zero, the ``local free energy'' of the equilibria at $\pm 1$ are significantly different
and hence the more energetically favorable equilibrium invades the spatial domain rapidly.
In such a parameter regime, the lattice discreteness is hardly felt; phenomenologically
\eqref{eq:int:nagumo:lde} behaves very much like its PDE limit.  However, when $\rho$ is close to zero,
the two equilibria are comparably stable, the more stable of the two invades the spatial domain slowly,
and the lattice discreteness is felt keenly.

Stated more precisely, it is known \cite{MPB}
that for each $\rho \in (-1, 1)$ and for each $\zeta \in \Real$,
there is a unique $c = c_{\rho}(\zeta)$ for which
the travelling wave MFDE
\begin{equation}
\label{eq:nag:trvWaveMFDENoGamma}
\begin{array}{lcl}
c\Phi'(\xi) & = &
  \Phi(\xi + \cos \zeta) + \Phi(\xi - \cos \zeta) + \Phi(\xi + \sin \zeta) + \Phi(\xi - \sin \zeta) - 4 \Phi(\xi)
   + g\big(\Phi(\xi) ; \rho \big)
\end{array}
\end{equation}
augmented with the limits
\begin{equation}
\lim_{\xi \to - \infty} \Phi(\xi) = -1, \qquad \lim_{\xi \to + \infty} \Phi(\xi) = + 1
\end{equation}
admits a monotonic solution $\Phi$. If $c_{\rho}(\zeta) > 0$, then the profile $\Phi$
can be uniquely fixed by demanding $\Phi(0) = 0$ and we will refer
to this profile as $\Phi_{\rho;\zeta}$.

Symmetry conditions imply that $c_{\rho}(\zeta) = - c_{-\rho}(\zeta)$
and hence $c_{0}(\zeta) =0$ for all $\zeta \in \Real$.
In addition, upon writing $\rho_*(\zeta) = \sup \{\rho: c_{\rho}(\zeta) = 0\}$,
it is known \cite{MPB} that $0 \le \rho_*(\zeta) < 1$ for all $\zeta \in \Real$.
In particular, if $\rho_*( \zeta) <  \abs{\rho} < 1$, then
\sref{eq:nag:LDE}
has a travelling wave solution
\begin{equation}
\label{eq:nag:trvWaveAnsatz}
u_{ij}(t) = \Phi_{\rho;\zeta}\big( i \cos \zeta + j \sin \zeta + c_{\rho}(\zeta) t \big).
\end{equation}

\begin{figure}[t]
\centering\includegraphics[width=\textwidth]{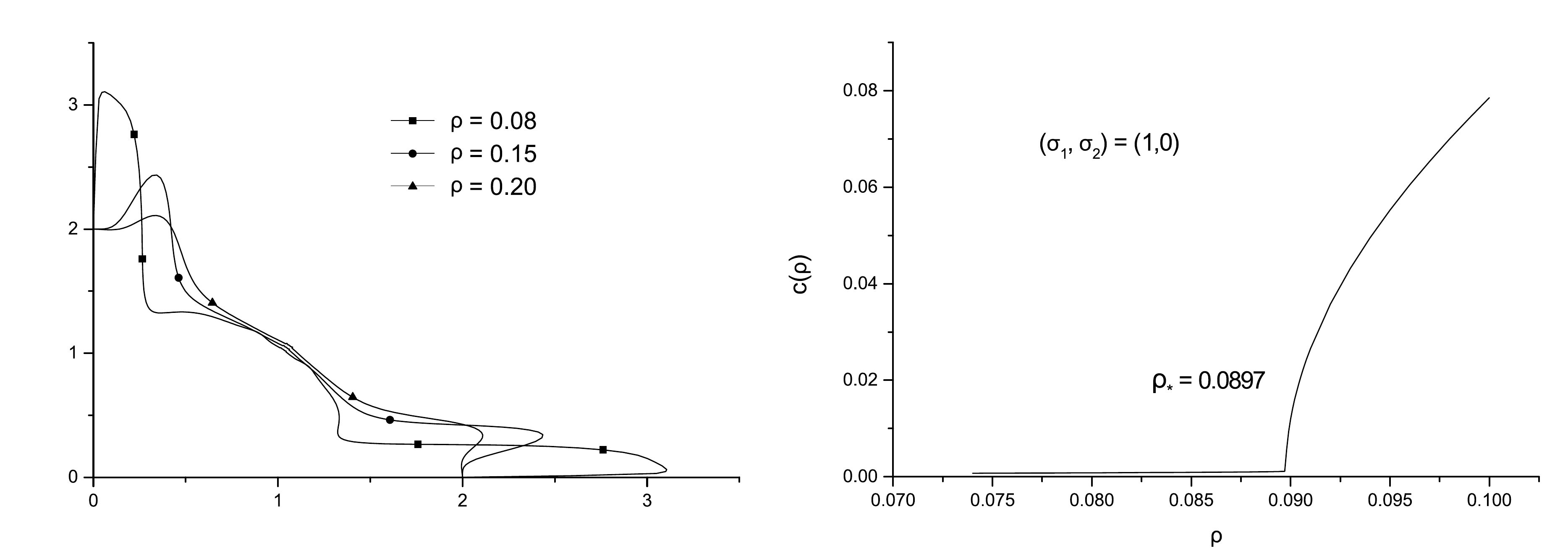}
\caption[]{Left: plot depicting
$M_\rho(\zeta) = - [\delta^2_\omega \lambda_{\omega; \rho, \zeta}]_{\omega = 0} $
for \sref{eq:nag:LDE} with detuning parameter $\rho$
and direction of propagation $(\sigma_1, \sigma_2) = (\cos\zeta , \sin \zeta)$.
It is a polar plot, so actually plotted are $(\cos \zeta M_\rho(\zeta) , \sin \zeta M_\rho(\zeta) )$.
The computations were performed
using the approximating systems \sref{eq:ex:mfdeNumSolve},
\sref{eq:num:approxForPsi} and \sref{eq:num:approxForPhi1} with $\gamma = 10^{-5}$.
Right: $\rho \mapsto c_{\rho}(0)$ plot for the fixed direction $\zeta = 0$,
which corresponds to $(\sigma_1, \sigma_2) = (1, 0)$.
The critical value for $\rho$ at which pinning sets in is given. The computations were
performed using \sref{eq:ex:mfdeNumSolve} with $\gamma = 10^{-6}$.
}
\label{fig:nag:rho:large}
\end{figure}

Whenever $\tan \zeta \in \mathbb{Q}$,
the theory developed in this paper can in principle be used
to establish the nonlinear stability of the wave \sref{eq:nag:trvWaveAnsatz}
under small perturbations. Of course a number of conditions must first be checked
and we now turn to this task.

As a preliminary step, we state a comparison principle for the linearization of
\eqref{eq:int:nagumo:lde} about its traveling wave, represented in $(n,l)$ coordinates as
\begin{equation}
\dot{v}_{nl} -\left\{v_{n+\sigma_1,l-\sigma_2} + v_{n-\sigma_1,l+\sigma_2} + v_{n+\sigma_2,l+\sigma_1}
 + v_{n-\sigma_2,l-\sigma_1} - 4v_{nl} + g'(\Phi(n+ct))v_{nl}\right\} = 0.
\label{eq:nagumo:lin}
\end{equation}
In case the equality in \eqref{eq:nagumo:lin} is replaced with the inequality $\ge$ ($\le$) we
say that $v$ is a super- (sub-) solution.
The following standard comparison principle now applies.

\begin{lem}
Let $\Lambda \subset \Z^2$ denote the lattice spanned by $(\sigma_1,\sigma_2)$ and $(\sigma_1,-\sigma_2)$ on which \eqref{eq:nagumo:lin} is naturally posed.  Let $q^0,r^0 \in \ell^\infty(\Lambda,\R)$ with $q^0_{nl} \ge r^0_{nl}$ for all $(n,l) \in \Lambda$ and $q^0 \not \equiv r^0$.  Let $q(t)$ and $r(t)$ denote super- and sub-solutions to \eqref{eq:nagumo:lin} with initial conditions $q^0$ and $r^0$ respectively.  Then $q_{nl}(t) \ge r_{nl}(t)$ for all $t > 0$.  Moreover, if $q$ and $r$ are both true solutions, then the inequality is strict, i.e. $q_{nl}(t) > r_{nl}(t)$ for all $t > 0$.
\label{lem:nagumo:comp}
\end{lem}

%
Exploiting this comparison principle,
we can verify all but one of our main hypotheses.
\begin{lem}
Consider the model equation \eqref{eq:int:nagumo:lde} with nonlinearity $g = g(\cdot;\rho)$ given
by \eqref{eq:nagumo:g}. Pick
$(\sigma_1, \sigma_2) \in \Wholes^2$ with $\mathrm{gcd}(\sigma_1, \sigma_2) = 1$
and choose $\zeta$ in such a way that
\begin{equation}
(\sigma_1, \sigma_2) =\sqrt{\sigma_1^2 + \sigma_2^2} ( \cos\zeta, \sin \zeta ).
\end{equation}
Then for any $\rho$ with $\rho_*(\zeta) < \abs{\rho} < 1$,
the hypotheses (Hf), (H$\Phi$), $\mathit{(HS1)}_p$,
$\mathit{(HS2)}_p$, $\mathit{(HS3)}_p$, $\mathit{(HS)}_{\mathrm{ess}}$ and $\mathit{(HS2)}_\omega$ all hold.
\end{lem}

\begin{proof}
In the context of our abstract framework \eqref{eq:mr:mainLDE} we have $f(u) = f(u;\rho)$ with
\begin{equation}
 f(u_1,u_2,u_3,u_4,u_5;\rho) = u_1 + u_2 + u_3 + u_4 - 4 u_5 + g(u_5;\rho),
\end{equation}
which is clearly $C^2$ and vanishes when all of the $u_m$ are equal to $u_\pm \in \{-1,\rho,1\}$.
Thus (Hf) is satisfied. The discussion above concerning the existence of
$c = c_{\rho}(\zeta) \neq 0$ and $\Phi = \Phi_{\rho; \zeta}$
shows that (H$\Phi$) is satisfied as well.  To verify $\mathrm{(HS)}_{\mathrm{ess}}$, we compute
\begin{equation}
\label{eq:nag:ess}
\Delta_\omega^\pm(i\nu) -\lambda = -ci\nu + 2\cos(\nu\sigma_1 + \omega\sigma_2) + 2\cos(\nu\sigma_2 - \omega \sigma_1)
- 4 + g'(u_\pm) -\lambda.
\end{equation}
Since $g'(u_\pm) < 0$ and since the negative constant $-4$ dominates the cosine terms, it follows that $\Re \lambda < 0$
whenever \sref{eq:nag:ess} vanishes, establishing $\mathrm{(HS)}_{\mathrm{ess}}$.
The hypotheses $\mathrm{(HS2)}_p$ and $\mathrm{(HS3)}_p$ follow from
\cite[Thm 4.1]{MPB}.

We now consider the remaining hypotheses
$\mathit{(HS1)}_p$ and  $\mathit{(HS2)}_\omega$.
In view of the identity \sref{eq:nag:ess},
the arguments in \cite[Lem. 3.2]{HJHSTBFHN}
involving the spectral flow formula stated in \cite[Thm. C]{MPA}
can be copied verbatim. This yields the existence of a constant $\beta > 0$ for which
the operators $\mathcal{L}_{\omega} - \lambda$
are Fredholm with index zero for all $\omega \in \Real$
and $\lambda \ge - \beta$. It hence
remains to look for eigenvalues for $\mathcal{L}_{\omega}$,
which, as argued in the proof of \cite[Lem 3.3]{HJHSTBFHN},
occur at isolated points.

The key ingredient towards ruling out eigenvalues other than $\lambda_{\omega}$,
is the existence of strictly positive constants $\beta$, $K_1$, $K_2$ and $K_3$ with $K_3 > K_2$
for which the function
\begin{equation}
 v^+_{nl}(t) = e^{-\beta t}(K_1 - K_2 \phi(n+ct )) + K_3\phi(n+ct )
\label{eq:nagumo:super}
\end{equation}
is a super-solution for \eqref{eq:nagumo:lin}, where we recall the notation $\phi = \Phi'$.
This was proved in \cite[Lem. 8.3]{VL9} for the $\sigma = (1, 0)$
case, but can also easily be verified for our general $\sigma$ case;
see also \cite[Lem. 7.6]{SHEN1999a}) for
a related approach.

Let us now assume that for some $\omega$ and
non-zero $w \in BC_0(\Real, \Complex)$
we have $\mathcal{L}_\omega w = (\alpha + i\nu) w$ for some $\alpha \ge 0$ and $\nu \in \Real$.
First note that $w$ is exponentially localized
as a consequence of $\mathrm{(HS)}_{\mathrm{ess}}$ and \cite[Thm. A]{MPA}.
In addition, one may readily verify that
\sref{eq:nagumo:lin} admits real-valued solutions
\begin{equation}
v_{nl; \vartheta}(t) = e^{\alpha t}\cos(\nu t+\omega l + \vartheta)\Re w(n+ct )
 - e^{\alpha t} \sin(\nu t+\omega l + \vartheta)\Im w(n+ct)
\end{equation}
for every $\vartheta \in \Real$.
Without loss, we assume that $2\norm{w}_{\infty} \le K_1$, in which case we have
\begin{equation}
\abs{v_{nl; \vartheta}(0)} \le K_1 \le  v^+_{nl}(0), \qquad (n,l) \in \Lambda,
\end{equation}
which by the comparison principle stated in Lemma \ref{lem:nagumo:comp} now implies
\begin{equation}
\abs{v_{nl; \vartheta}(t)} \le v^+_{nl}(t), \qquad t >0, \qquad (n,l) \in \Lambda.
\end{equation}
This immediately implies $\alpha = 0$.

In addition, pick any $\Omega \in \Real$ and $\xi \in \Real$.
We can now pick a sequence $\{\vartheta_k, t_k, n_k \} \in \Real \times [0, \infty) \times \Wholes$
for which we have the limit $t_k \to \infty$ as $k \to \infty$
together with the identities $n_k + c t_k = \xi$
and $\nu t_k +  \vartheta_k = \Omega$.
Since
\begin{equation}
\abs{v_{n_k, 0 ; \zeta_k}(t_k) } \le v^+_{n_k, 0}(t_k) =
e^{-\beta t_k}(K_1 - K_2 \phi(\xi )) + K_3\phi(\xi ),
\end{equation}
we can take the limit $k \to \infty$ to obtain
\begin{equation}
\abs{ \cos ( \Omega ) \Re w(\xi)  - \sin ( \Omega ) \Im w(\xi) } \le K_3 \phi(\xi).
\end{equation}
Since both $w(\xi)$ and $\phi(\xi)$
decay exponentially as $\xi \to \pm \infty$,
there exists $\kappa > 0$ so that
for every $(\Omega, \xi) \in \Real^2$
we have
\begin{equation}
  \abs{ \cos (\Omega) \Re w(\xi)  - \sin (\Omega) \Im w(\xi) } \le \kappa \phi(\xi),
\end{equation}
with equality at some pair $(\Omega_*, \xi_*) \in \Real^2$.

Now, pick a pair $t_* > 0$ and $n_* \in \Wholes$ with $n_* + c t_* = \xi_*$.
In addition, pick $\vartheta_*$ so that $\nu t_* + \vartheta_* = \Omega_*$.
Then we have
\begin{equation}
\label{eq:nag:lem:cond:eqAtT_*}
\abs{ v_{n_*, 0; \vartheta_*}(t_*) } = \kappa\phi(\xi_*) = \kappa \phi(n_* + c t_*),
\end{equation}
together with
\begin{equation}
\abs{v_{nl ; \vartheta_*}(0) } \le \kappa \phi(n),    \qquad (n,l) \in \Lambda.
\end{equation}
The comparison principle hence implies
\begin{equation}
\label{eq:nag:lem:cond:CmpVWithPhi}
\abs{v_{nl; \vartheta_*}(t)} \le \kappa \phi(n + ct), \qquad t > 0, \qquad (n,l) \in \Lambda.
\end{equation}
In view of \sref{eq:nag:lem:cond:eqAtT_*} the inequality
in \sref{eq:nag:lem:cond:CmpVWithPhi} is not strict, which is only possible if
\begin{equation}
v_{n,l; \vartheta_*}(t) = \cos( \nu t+ \omega l + \vartheta_*) \Re w(n + ct)
   - \sin(\nu t+ \omega l + \vartheta_*) \Im w(n + ct)= \phi(n + ct)
\end{equation}
for all $(n,l) \in \Lambda$ and all $t \ge 0$. This implies $\omega \in 2 \pi \Wholes$
together with $\nu \in 2 \pi c \Wholes$, as desired.
\end{proof}

The remaining hypothesis $\mathrm{(HS1)}_\omega$ is more difficult to verify.
We note that the construction of the curve $\omega \mapsto \lambda_{\omega}$
in Proposition \ref{prp:mr:melnikov} can be performed even if
$(\sigma_1, \sigma_2) = (\cos \zeta , \sin \zeta) \notin \Wholes^2$,
provided of course that $c_{\rho}(\zeta) \neq 0$.
We use the notation $\lambda_{\omega} = \lambda_{\omega; \rho, \zeta}$
to make this dependence explicit.
For any pair $(\rho, \zeta) \in (-1,1) \times \Real$ with $c_{\rho}(\zeta) \neq 0$,
we may then define the quantity
\begin{equation}
M_{\rho}(\zeta) = - [\partial^2_{\omega} \lambda_{\omega; \rho,\zeta}]_{\omega = 0}.
\end{equation}

Now consider any pair $(\sigma_1, \sigma_2) \in \Wholes^2$ with $\mathrm{gcd}(\sigma_1, \sigma_2) = 1$
and choose $\zeta$ in such a way that
\begin{equation}
(\sigma_1, \sigma_2) =\sqrt{\sigma_1^2 + \sigma_2^2} ( \cos\zeta, \sin \zeta )
\end{equation}
We note that condition $\mathrm{(HS1)}_\omega$ is satisfied for
the model system \sref{eq:nag:LDE} with this pair $(\sigma_1, \sigma_2)$
whenever $c_{\rho}(\zeta) \neq 0$ and $M_{\rho}(\zeta) > 0$.
The following result shows that this can be guaranteed
for an open set of angles $\zeta$.

\begin{lem}
Fix any $\rho \in (-1, 1)$.
If $\rho_*(0) < \abs{\rho} < 1$,
then there exists
$\delta_{\mathrm{hor}} = \delta_{\mathrm{hor}}(\rho) > 0$ so that
for any $\zeta \in \Real$ that has $\abs{\theta - k \frac{\pi}{2} } < \delta_{\mathrm{hor}}$
for some $k \in \Wholes$, we have
\begin{equation}
c_{\rho}(\zeta) \neq 0,  \qquad M_{\rho}(\zeta) > 0.
\end{equation}

In addition, if $\rho_*(\frac{\pi}{4} < \abs{\rho} < 1$,
then there exists
$\delta_{\mathrm{diag}} = \delta_{\mathrm{diag}}(\rho) > 0$ so that
for any $\zeta \in \Real$ that has $\abs{\theta - \frac{\pi}{4} - k \frac{\pi}{2} } < \delta_{\mathrm{diag}}$
for some $k \in \Wholes$, we have
\begin{equation}
c_{\rho}(\zeta) \neq 0,  \qquad M_{\rho}(\zeta) > 0.
\end{equation}
\end{lem}
\begin{proof}
For fixed $(\sigma_1, \sigma_2) \in \Real^2 \setminus \{(0, 0)\}$,
we can compute
\begin{equation}
 \big[  [ \partial^2_\omega \mathcal{L}_\omega]_{\omega = 0} \phi\big](\xi) =
 -\sigma_2^2(\phi(\xi+\sigma_1)+\phi(\xi-\sigma_1)) - \sigma_1^2(\phi(\xi+\sigma_2)+\phi(\xi-\sigma_2)) < 0,
\end{equation}
giving us the desired sign
\begin{equation}
\int_\R \langle \psi(\xi) , \big[ [ \partial^2_\omega \mathcal{L}_\omega]_{\omega = 0} \phi \big](\xi) \rangle d\xi < 0
\end{equation}
for the first term in \eqref{eq:mr:melnikov}, because $\psi > 0$.

Unfortunately, the second term in \sref{eq:mr:melnikov} is harder to analyze in general.
However, in the special cases $\sigma = (1,0)$ and $\sigma = (1,1)$, the
map $\omega \mapsto \mathcal{L}_\omega$ is
symmetric with respect to $\omega$, which implies that
$[\partial_\omega \lambda_\omega]_{\omega =0}$ and $[\partial_\omega \phi_\omega]_{\omega =0}$
both vanish and hence the second term in \sref{eq:mr:melnikov} disappears.
The statements now follow from the fact that $c_{\rho}(\zeta)$ and $M_{\rho}(\zeta)$
depend continuously on $\zeta$ whenever $c_{\rho}(\zeta) \neq 0$,
which can be established as in \cite[Prop. 3.7]{HJHOBST2D}.
\end{proof}

\begin{figure}[t]
\centering\includegraphics[width=\textwidth]{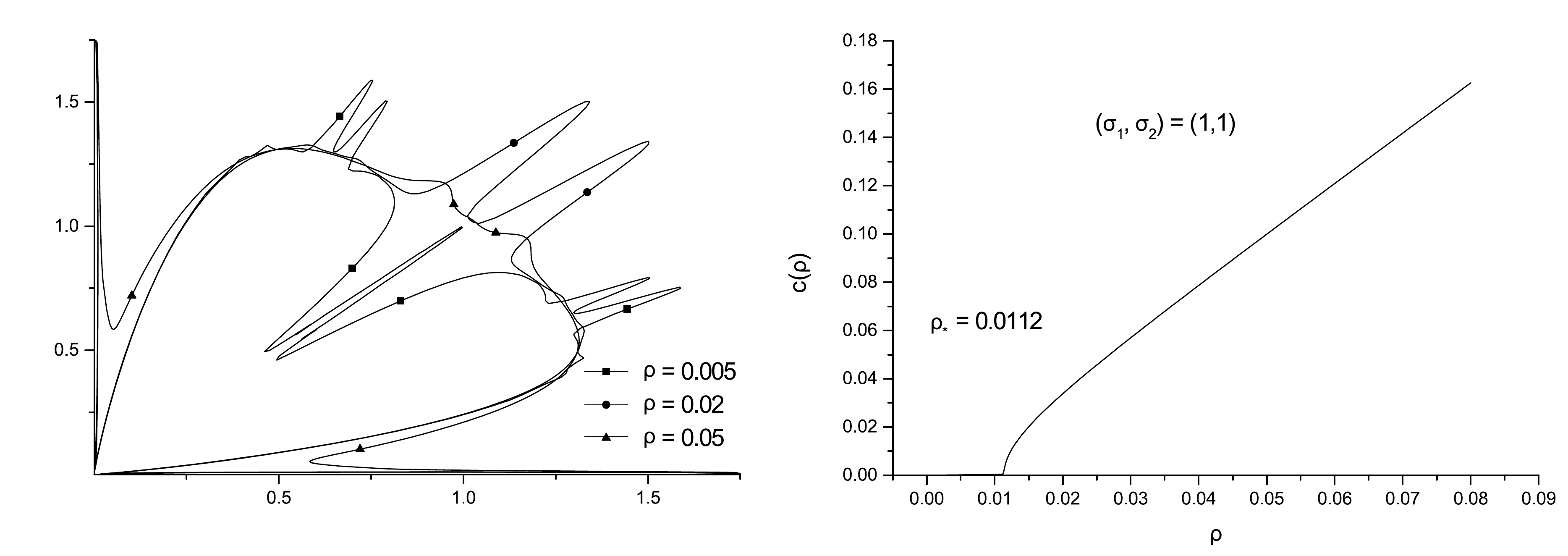}
\caption[]{Left: plot depicting $M_\rho(\zeta) = -[\delta^2_\omega \lambda_{\omega;\rho, \zeta}]_{\omega = 0  }$,
as in Figure \ref{fig:nag:rho:large}.
The computations were performed using
\sref{eq:ex:mfdeNumSolve},
\sref{eq:num:approxForPsi} and \sref{eq:num:approxForPhi1}
with $\gamma = 10^{-4}$.
Right: $\rho \mapsto c_{\rho}(\frac{\pi}{4})$ plot for the fixed direction
$\zeta = \frac{\pi}{4}$, which corresponds to $(\sigma_1, \sigma_2) = (1, 1)$.
The critical value for $\rho$ at which pinning sets in is given.
The computations were performed using \sref{eq:ex:mfdeNumSolve} with $\gamma = 10^{-6}$.
}
\label{fig:nag:rho:small}
\end{figure}

At present we can only rigorously establish $\mathrm{(HS1)}_\omega$
for the discrete Nagumo system with cubic nonlinearity
for a subset of rational angles, namely those close to the horizontal, diagonal and vertical
directions as stated in the result above.
Let us remark that there do exist other nonlinearities, e.g., nonlinearities similar to the cubic-like nonlinearity
in \cite{EVVRCW}, for which explicit solutions are known for general rational directions
of propagation. This should considerably aid the verification of $\mathrm{(HS1)}_\omega$.

To make further statements for the Nagumo equation with cubic nonlinearity and
more general $\sigma$, we now turn to numerics.  Figures 1 and 2 show the dependence of the
Melnikov integral on $\zeta$ and $\rho$.  For the parameters we have investigated we see
non-vanishing Melnikov integrals so long as $\rho > \rho_*(\zeta)$.
However, there does not appear to be any obvious structure such as monotonicity that we can exploit
to prove analytically that the Melnikov integral is non-vanishing.

The computations were performed using the numerical method developed in \cite{EVVRCW,EVV,HJHVL2005}.
In particular, we search for wave solutions
by numerically solving the functional differential equation\footnote{
Here the $\pm$ signs mean both terms are evaluated and added together.
}
\begin{equation}
  \label{eq:ex:mfdeNumSolve}
 \begin{array}{rcl}
c(\xi) \Phi'(\xi) & = & \gamma \Phi''(\xi) + [\Phi(\xi \pm \cos \zeta) + \Phi(\xi \pm \sin \zeta) - 4 \Phi(\xi) ]
+      g\big( \Phi(\xi) ; \rho)
\\[0.2cm]
c'(\xi) & = & 0,
  \end{array}
\end{equation}
with boundary conditions $\Phi(-L) = -1$, $\Phi(L) = 1$ and $\Phi(0) = 0$
for some large $L > 0$.
Notice the extra $\gamma \Phi''(\xi)$ term appearing
in \sref{eq:ex:mfdeNumSolve} as compared to \sref{eq:nag:trvWaveMFDENoGamma}.
To make use of robust codes to solve two-point boundary value
differential equations
and to avoid numerical issues that arise in the singular limit $c \to 0$,
  we follow the approach used in \cite{EVVRCW,EVV,HJHVL2005,HJH,HJHDNGPF}
and add this extra term in  \sref{eq:ex:mfdeNumSolve},
  with $\gamma = 10^{-6}$, $\gamma = 10^{-5}$ or $\gamma = 10^{-4}$ (depending on context; see captions).
  We refer the reader to \cite{HJHVL2005, HJH} for numerical
  and theoretical results concerning the limit $\gamma \to 0$. These results strongly
  suggest that the region of propagation failure can be accurately determined by
  using small but practical values of $\gamma >0$. In particular, the
  numerical results presented in this section were verified by repeating
  a subset of the calculations with different values of $\gamma$ where possible.

To compute the Melnikov constants,
one needs to have detailed information on the adjoint eigenfunction $\psi$,
and the rotated first derivatives $\lambda_1 := - i [\partial_{\omega} \lambda_{\omega}]_{\omega = 0}$
and $\phi_1 := -i [\partial_{\omega} \phi_{\omega}]_{\omega = 0}$.
These were also computed numerically.
  In particular, the actual equations used to determine $\psi$ are given by
  \begin{equation}
  \label{eq:num:approxForPsi}
    \begin{array}{rcl}
       - \gamma \psi''(\xi) - c \psi'(\xi)  &=&
         [\psi(\xi \pm \cos \theta) + \psi(\xi \pm \sin \theta) - 4 \psi(\xi) ]
     + g'(\Phi(\xi) ; \rho) \psi(\xi) \\[0.2cm]
         & & \qquad   - \lambda_0(\xi) \psi(\xi), \\[0.2cm]
       \lambda_0'(\xi)            &=&  0, \\[0.2cm]
    \end{array}
  \end{equation}
  with boundary conditions $\psi(-L) = 0$, $\psi(L) = 0$ and $\psi(0) = 1$
  for some large $L > 0$. The actual equations used to compute $\phi_1$ and $\lambda_1$
  are given by
  \begin{equation}
    \label{eq:num:approxForPhi1}
    \begin{array}{rcl}
       - \gamma \phi_1''(\xi) + c \phi_1'(\xi)  &=&
         [\phi_1(\xi \pm \cos \theta) + \phi_1(\xi \pm \sin \theta) - 4 \phi_1(\xi) ]
      +g'(\Phi(\xi); \rho) \phi_1(\xi) \\[0.2cm]
         & & \qquad \pm \sin( \theta) \Phi'(\xi \pm \cos \theta) \pm \cos(\theta) \Phi'(\xi \mp \sin \theta)
  - \lambda_1(\xi) \Phi'(\xi), \\[0.2cm]
       \lambda_1'(\xi)            &=&  0, \\
    \end{array}
  \end{equation}
  again with boundary conditions $\phi_1(-L) = 0$, $\phi_1(L) = 0$ and $\phi_1(0) = 1$.


\section{Discussion}
\label{sec:dis}

In this paper we have established the stability of planar fronts
travelling in rational directions through two-dimensional lattices
with nearest-neighbour interactions, under reasonable spectral hypotheses.
Lattices in three or more spatial dimensions can be treated
using the exact same techniques. Based on
the PDE results \cite{KAP1997}, we even expect
faster rates of decay to occur. This could
even considerably simplify the nonlinear analysis by
eliminating the problematic slowly decaying nonlinearities.

The restriction to nearest-neighorbour interactions is
of course purely artificial and many other variants could be considered.
For example, in the case of the Nagumo LDE \sref{eq:int:nagumo:lde},
it would be interesting to see what the effects of a cross-shaped Laplacian
would be on the decay rate of perturbations to waves
travelling in the direction $(1,1)$.

Our restriction to $(\sigma_1, \sigma_2) \in \Wholes^2$ is
of a more serious nature. However,
we do believe that the spirit of our approach can be extended to irrational directions.
Indeed, the difficulties that need to be overcome seem to be primarily technical in nature.
For example, the Fourier transform in the transverse direction would not
map the sequence space $\ell^2(\Wholes; \Real^d)$ into $L^2([-\pi, \pi], \Real^d)$,
but would be a linear map from $L^2(\Real; \Real^d)$ onto itself. As a first consequence,
one can no longer use the embeddings $\ell^1(\Wholes; \Real^d) \subset \ell^2(\Wholes; \Real^d) \subset \ell^\infty(\Wholes; \Real^d)$,
which implies that any bounds on the nonlinear terms would rely on the Sobolev embedding $H^1(\Real; \Real^d) \subset L^\infty(\Real; \Real^d)$.
A second consequence is that one can no longer restrict the Fourier frequency $\omega$ to the compact interval $[-\pi, \pi]$.
This forces a sharpening of the condition $\mathrm{(HS2)}_{\omega}$. However, similar complications were succesfully
handled in \cite{KAP1997}, emboldening us in our view that they will not present a fundamental obstacle
in the current setting.

In establishing our spectral hypotheses for \eqref{eq:int:nagumo:lde} in {\S}\ref{sec:nag},
we made extensive use of the fact that the Nagumo equation is
scalar and enjoys a comparison principle.
It is natural to consider what can be achieved
by using the comparison principles directly and avoiding the intermediate
"spectral stability implies nonlinear stability" theorem.
The work \cite{MNT} uses comparison principles to establish multidimensional
stability for the PDE $u_t = \Delta u + g(u)$ and many aspects of the argument can be
readily carried over the lattice,
especially if one restricts attention to the special cases
$\sigma = (1,0)$ and $\sigma = (1,1)$, in which
the troublesome terms generated by the anisotropy of the lattice are absent or benign.
We note that \cite{BatesChen2002}, which studies multidimensional stability for traveling waves in higher
dimensional anisotropic nonlocal equations, also restricts attention to the $\sigma = (1,0)$ case
and its higher dimensional analogs.
Very recently, in the revision phase of this paper,
we were able to generalize \cite{MNT} to the lattice setting even for oblique
angles \cite{HJHOBST2D}.
The advantage of having an approach based on comparison principles
is that follow-up questions concerning  the basin of attraction of the travelling wave
and the presence of obstacles are much easier to address.



%

%
%
%
%
%

%

%

%

%

%

%

%

%

%

\bibliographystyle{klunumHJ}
\bibliography{ref}

\end{document}